\documentclass[reqno, 11pt]{amsart}
\title[Decay estimates for discrete bi-Laplace operators with potentials]{Decay estimates for discrete bi-Laplace operators with potentials on the lattice $\Z$ }
\author{ Sisi Huang and Xiaohua Yao }
\address{Sisi Huang, Department of Mathematics, Central China Normal University, Wuhan, 430079, P.R. China}
\email{hss@mails.ccnu.edu.cn}

\address{Xiaohua Yao, Department of Mathematics, Key Laboratory of Nonlinear Analysis and Applications(Ministry of Education), and Hubei Key Laboratory of Mathematical Sciences, Central China Normal University, Wuhan, 430079, P.R. China}
\email{yaoxiaohua@ccnu.edu.cn}
\thanks{The work is partially supported by NSFC (No.12671120, No.12531005) and the Fundamental Research Funds for the Central Universities}
\date{\today}
\usepackage{color}
\usepackage{amsmath}
\usepackage{amssymb}
\usepackage[mathscr]{euscript}
\usepackage{paralist}
\usepackage{amsthm}
\usepackage{ulem}
\usepackage{bm}
\usepackage{extarrows}
\usepackage{diagbox}
\usepackage{multirow}
\usepackage{makecell}
\usepackage{threeparttable}
\usepackage{tikz}
\usetikzlibrary{patterns}
\usetikzlibrary{decorations.markings}
\usepackage{caption}
\usepackage{verbatim}

\usepackage{cite}
\usepackage{hyperref}
\newtheorem{definition}{Definition}[section]
\newtheorem{theorem}[definition]{Theorem}
\newtheorem{lemma}[definition]{Lemma}
\newtheorem{remark}[definition]{Remark}
\newtheorem{proposition}[definition]{Proposition}

\newtheorem{corollary}[definition]{Corollary}
\newcommand\R{\mathbb{R}}
\newcommand\Z{\mathbb{Z}}
\newcommand\B{\mathbb{B}}
\newcommand\C{\mathbb{C}}
\newcommand\N{\mathbb{N}}
\newcommand\T{\mathbb{T}}
\newcommand\mscH{\mathcal{H}}
\newcommand\mcaF{\mathcal{F}}
\newcommand\mcaI{\mathcal{I}}
\newcommand\mcaJ{\mathcal{J}}
\newcommand\mcaA{\mathcal{A}}
\newcommand\mcaB{\mathcal{B}}
\newcommand\mcaC{\mathcal{C}}
\newcommand\mcaD{\mathcal{D}}

\newcommand\mcaM{\mathcal{M}}
\newcommand\mcaR{\mathcal{R}}
\newcommand\mcaN{\mathcal{N}}
\newcommand\mcaP{\mathcal{P}}
\newcommand\mcaE{\mathcal{E}}
\numberwithin{equation}{section}
\begin{document}

\maketitle

\begin{abstract}
It is known that the discrete Laplace operator $\Delta$ on the lattice $\mathbb{Z}$ satisfies the following sharp time decay estimate:
$$\big\|e^{it\Delta}\big\|_{\ell^1\rightarrow\ell^{\infty}}\lesssim|t|^{-\frac{1}{3}},\quad t\neq0,$$
which is slower than the usual $ O(|t|^{-\frac{1}{2}})$ decay in the continuous case on $\mathbb{R}$. However, this paper shows that the discrete bi-Laplacian $\Delta^2$ on $\mathbb{Z}$ actually exhibits the same sharp decay estimate $|t|^{-\frac{1}{4}}$ as its continuous counterpart.

In view of the free decay estimate, we further investigate the discrete bi-Schr\"{o}dinger operators of the form $H=\Delta^2+V$ on the lattice space $\ell^2(\mathbb{Z})$, where $V$ is a class of real-valued decaying potentials on $\Z$. 
First, we establish the limiting absorption principle for $H$, and then derive the full asymptotic expansions of the resolvent of $H$ near the thresholds $0$ and $16$, including resonance cases. In particular, we provide a complete characterization of the different resonance types in $\ell^2$-weighted spaces. 

Based on these results above, we establish the following sharp $\ell^1-\ell^{\infty}$ decay estimates for all different resonances types of $H$ under suitable decay conditions on $V$:
$$\big\|e^{-itH}P_{ac}(H)\big\|_{\ell^1\rightarrow\ell^{\infty}}\lesssim|t|^{-\frac{1}{4}},\quad t\neq0,$$
where $P_{ac}(H)$ denotes the spectral projection onto the absolutely continuous spectrum space of $H$. 
Additionally, the  decay estimates for the evolution flow of discrete beam equation are also derived:
$$\|{\cos}(t\sqrt H)P_{ac}(H)\|_{\ell^1\rightarrow\ell^{\infty}}+\Big\|\frac{{\sin}(t\sqrt H)}{t\sqrt H}P_{ac}(H)\Big\|_{\ell^1\rightarrow\ell^{\infty}}\lesssim|t|^{-\frac{1}{3}},\quad t\neq0.$$
\end{abstract}

\tableofcontents

\section{Introduction and main results}
\subsection{Introduction}
The analysis of time decay estimates for Schr\"{o}dinger operators has long been a central theme in mathematical physics, originating from the fundamental studies of quantum dynamics in Euclidean space (see e.g., \cite{JK79,JSS91}). Following these pioneering works, significant advances has been made across various dimensions and potentials (cf. \cite{ES04,GG15,GG17,RS04} and references therein), with recent breakthroughs extending to higher-order Schr\"{o}dinger operators~(cf.\cite{FSY18,FSWY20,EGT21,CHHZ24a,CHHZ24b}). These estimates, deeply connected to scattering theory and spectral analysis, have become indispensable tools in quantum dynamics and nonlinear analysis~(cf.\cite{Tao06, Sch07, RS78,Mar06}).

As a natural extension of the continuous theory, the study of discrete Schr\"{o}dinger operators on the lattice $\Z^d$ has emerged as an equally important but mathematically distinct research area. A basic fact is that the discrete Laplacian $\Delta_{\Z^d}$ exhibits the sharp decay estimates (cf.\cite{SK05})
\begin{equation}\label{decay est for LaplaceZd}
\big\|e^{-it\Delta_{\Z^d}}\big\|_{\ell^1(\Z^d)\rightarrow\ell^{\infty}(\Z^d)}\lesssim |t|^{-\frac{d}{3}},\qquad t\neq0,
\end{equation}
where $\Delta_{\Z^d}$ is defined by
\begin{equation}\label{definition of Laplacian}
(\Delta_{\Z^d}\phi)(n):=\sum_{j=1}^d(\phi(n+e_j)+\phi(n-e_j)-2\phi(n)), \qquad\forall\ \phi\in\ell^2(\Z^d),
\end{equation}
with $$e_j=(0,0,\cdots,0,\underset{j}{1},0,\cdots,0),\quad j=1,2,\cdots,d.$$ Notably, this decay rate is slower than the corresponding $|t|^{-\frac{d}{2}}$ rate in $\R^d$, reflecting the distinct dispersive properties of lattice systems. Subsequent works have established analogous decay estimates for perturbed discrete Schr\"{o}dinger operators on $\Z$ (see e.g., \cite{PS08,CT09,EKT15,MZ22}), while the corresponding analysis for higher-order discrete Schr\"{o}dinger operators remains largely unexplored.

Motivated by these considerations,  in this paper we investigate the time decay estimates for solutions to the following fourth-order Schr\"{o}dinger equation on the lattice $\Z$:
\begin{equation}\label{Bi-Schrodinger equation}
\left\{\begin{aligned}&i(\partial_tu)(t,n)-(\Delta^2u+Vu)(t,n)=0,\qquad (t,n)\in\R\times\Z,\\
&u(0,n)=\varphi_0(n),\end{aligned}\right.
\end{equation}
and the discrete beam equation on the lattice $\Z$:
\begin{equation}\label{Beam equation}
\left\{\begin{aligned}&(\partial_{tt}v)(t,n)+(\Delta^2v+Vv)(t,n)=0,\qquad (t,n)\in\R\times\Z,\\
&v(0,n)=\varphi_1(n),\ \left(\partial_{t}v\right)(0,n)=\varphi_2(n).\end{aligned}\right.
\end{equation}
Here $(\Delta\phi)(n):=\phi(n+1)+\phi(n-1)-2\phi(n)$, and $V$ is a real-valued potential satisfying $|V(n)|\lesssim \left<n\right>^{-\beta}$ for some $\beta>0$ with $\left<n\right>=(1+|n|^2)^{\frac{1}{2}}$.

The discrete bi-Laplace operator $\Delta^2$ on the lattice $\Z$ is the discrete analogue of the fourth-order differential operator $\frac{d^4}{dx^4}$ on the real line. The equations \eqref{Bi-Schrodinger equation} and \eqref{Beam equation} are discrete counterparts of the continuous models studied in \cite{SWY22} and \cite{CLSY25}, respectively. Such discrete models are not only useful in numerical analysis but are also of interest in mathematical physics. For instance, discrete Schr\"{o}dinger equations provide standard models for dynamics in random media, as discussed by Aizenman and Warzel \cite{AW15}, while discrete beam equations can be used to describe the deformation of elastic beams under external forces (cf.~\"{O}chsner \cite{Och21}).

Denote $H:=\Delta^2+V$. Then both $\Delta^2$ and $H$ are bounded self-adjoint operators on $\ell^2(\Z)$, generating the associated unitary groups $e^{-it\Delta^2}$ and $e^{-itH}$, respectively. The solutions to equations \eqref{Bi-Schrodinger equation} and \eqref{Beam equation} are given as follows:
\begin{equation}\label{solutions for Bi-Schrodinger}
u(t,n)=e^{-itH}\varphi_0(n),
\end{equation}
\begin{equation}\label{solutions for Beam-equation}
v(t,n)={\rm cos}(t\sqrt H)\varphi_1(n)+\frac{{\rm sin}(t\sqrt H)}{\sqrt H}\varphi_2(n).
\end{equation}
The expression \eqref{solutions for Beam-equation} above depends on the branch chosen of $\sqrt z$ with $\Im z\geq0$, so the solution $v(t,n)$ is well-defined even if $H$ is not positive. In the sequel, we are devoted to establishing the time decay estimates of the propagator operators $e^{-itH}$, ${\rm cos}(t\sqrt H)$ and $\frac{{\rm sin}(t\sqrt H)}{\sqrt H}$.

For the free case,~i.e., $V\equiv0$, then $\sqrt{H}=-\Delta$. It follows from \eqref{decay est for LaplaceZd} that
\begin{equation}\label{eitlaplacian}
\left\|e^{it\Delta}\right\|_{\ell^1\rightarrow\ell^{\infty}}\lesssim|t|^{-\frac{1}{3}},\quad t\neq0.
\end{equation}
As a consequence of \eqref{eitlaplacian}, one has
\begin{equation}\label{cos-sin lapacian}
\|{\rm cos}(t\Delta )\|_{\ell^1\rightarrow\ell^{\infty}}+\Big\|\frac{{\rm sin}(t\Delta )}{t\Delta }\Big\|_{\ell^1\rightarrow\ell^{\infty}}\lesssim|t|^{-\frac{1}{3}}.
\end{equation}
As stated before, the decay estimates \eqref{eitlaplacian} and \eqref{cos-sin lapacian} are slower than the usual $|t|^{-\frac{1}{2}}$ in the continuous case. However, interestingly, for the discrete bi-Laplacian on $\Z$, we can prove that
\begin{equation}\label{eitlap2}
\big\|e^{-it\Delta^2}\big\|_{\ell^1\rightarrow\ell^{\infty}}\lesssim|t|^{-\frac{1}{4}},
\end{equation}
which is sharp and the same as the continuous case on the line \cite{SWY22}, for details see Section \ref{Sec of decay for free}.

When $V\not\equiv0$, the decay estimates for the solution operators of equation \eqref{Beam equation} are affected by the spectrum of $H$, which in turn depends on the conditions of potential $V$. In this paper, we assume that the potential $V$ decays sufficiently fast and $H$ has no embedded positive eigenvalues in the continuous spectrum interval $(0,16)$. Under such assumptions, let $\lambda_j$ be the discrete eigenvalues of $H$ and $H\phi_j=\lambda_j\phi_j$ for $\phi_j\in\ell^2(\Z)$, $P_{ac}(H)$ denote the spectral projection onto the absolutely continuous spectrum of $H$ and $P_j$ be the projection on the eigenspace corresponding to the discrete eigenvalue $\lambda_j$. Then the solutions of the equations \eqref{Bi-Schrodinger equation} and \eqref{Beam equation} can be respectively further written as
\begin{equation}\label{decomp of eitH}
u(t,n)=\sum\limits_{j}e^{-it\lambda_j}P_j\varphi_{0}(n)+e^{-itH}P_{ac}(H)\varphi_0(n):=u_{d}(t,n)+u_{c}(t,n),
\end{equation}
\begin{equation}\label{decom of v}
v(t,n)=v_{d}(t,n)+v_{c}(t,n),
\end{equation}
where
\begin{align*}
v_{d}(t,n)&=\sum\limits_{j}^{}{\rm cosh}(t\sqrt{-\lambda_j})\left<\varphi_1,\phi_j\right>\phi_j(n)+\frac{{\rm sinh}(t\sqrt{-\lambda_j})}{\sqrt{-\lambda_j}}\left<\varphi_2,\phi_j\right>\phi_j(n),\\
v_{c}(t,n)&={\rm cos}(t\sqrt{H})P_{ac}(H)\varphi_1(n)+\frac{{\rm sin}(t\sqrt H)}{\sqrt H}P_{ac}(H)\varphi_2(n).
\end{align*}
Observe that the discrete part $u_{d}(t,n)$ of $u$ has no any time decay estimates. Similarly,  the existence of discrete negative/positive eigenvalues of $H$ will lead to the exponential growth/dissipation of $v_{d}(t,n)$ as $t$ becomes large. Therefore, the main goal of this paper is to investigate the time decay estimates for the continuous components $u_c(t,n)$ and $v_c(t,n)$ in \eqref{decomp of eitH} and \eqref{decom of v} under suitable decay conditions on $V$ for all resonance types of $H$ (see Definition \ref{defin of resonance types 0 and 16} below).

To achieve this, our starting point is Stone's formula. We first establish the limiting absorption principle for the operator $H$, and then study the asymptotic expansions of $R^{\pm}_V(\lambda)$ near the thresholds $0$ and $16$ for all resonance types. Finally, we apply the Van der Corput lemma to derive the desired decay estimates.

\subsection{Main results}
In this subsection, we state our main results. We begin with some definitions. Unlike the continuous setting, where zero is the only critical value, the discrete problem has two critical values, namely
$0$ (degenerate) and $16$ (non-degenerate); see Subsection
\ref{Subsec of free resol}. This leads to a richer threshold structure, which can be characterized in terms of solutions to the difference equations $H\phi=0$ and $H\phi=16\phi$ in the weighted spaces
$W_{\sigma}(\Z)$:=$\underset{s>\sigma}{\bigcap}\ell^{2,-s}(\Z)$, where $\sigma\in\R$ and
$$\ell^{2,s}(\Z)=\Big\{\phi=\left\{\phi(n)\right\}_{n\in\Z}:\|\phi\|^2_{\ell^{2,s}}=\sum_{n\in\Z}^{}\left<n\right>^{2s}|\phi(n)|^2<\infty\Big\}.$$
For $a,b\in\R^{+}$, we write $a\lesssim b$ if $a\leq cb$ for some constant $c>0$.
\begin{definition}\label{defin of resonance types 0 and 16}
 Let $H=\Delta^2+V$ be defined on the lattice $\Z$ and $|V(n)|\lesssim \left<n\right>^{-\beta}$ for some $\beta>0$. 
\vskip0.2cm
{\bf{(I)~Classification of resonances at  0}}
\vskip0.2cm
\begin{itemize}
\item[{\rm(i)}] $0$ is a regular point of $H$ if $H\phi=0$ has only the zero solution in $W_{3/2}(\Z)$.
    \vskip0.1cm
\item [{\rm(ii)}] $0$ is a first kind resonance of $H$ if  $H\phi=0$ has a nonzero solution in $W_{3/2}(\Z)$ but only the zero solution in $W_{1/2}(\Z)$.
\vskip0.1cm
\item [{\rm(iii)}] $0$ is a second kind resonance of $H$ if $H\phi=0$ has a nonzero solution in $W_{1/2}(\Z)$ but only the zero solution in $\ell^2(\Z)$.
\vskip0.1cm
\item [{\rm(iv)}] $0$ is an eigenvalue of $H$ if  $H\phi=0$ has a nonzero solution in $\ell^2(\Z)$.
\end{itemize}
\vskip0.2cm
\ \ {\bf{(II)~Classification of resonances at  16}}
\vskip0.2cm
\begin{itemize}
\item [{\rm(i)}]  $16$ is a regular point of $H$ if $H\phi=16 \phi$ has only the zero solution in $W_{1/2}(\Z)$.
\vskip0.1cm
\item [{\rm(ii)}]  $16$ is a resonance of $H$ if  $H\phi=16\phi$ has a nonzero solution in $W_{1/2}(\Z)$ but only the zero solution in $\ell^2(\Z)$.
\vskip0.1cm
\item [{\rm(iii)}] $16$ is an eigenvalue of $H$ if $H\phi=16\phi$ has a nonzero solution in $\ell^2(\Z)$.
\end{itemize}
\end{definition}
Clearly, when $V\equiv0$, both 0 and 16 are resonances of $H$. Indeed, taking $\phi_1(n)=cn+d$ and $\phi_2(n)=(-1)^n c$ with $c\neq0$,
we have $\Delta^2\phi_1=0$ and $\Delta^2\phi_2=16\phi_2$. Beyond this special case, one can also construct nontrivial examples of resonances at the thresholds $0$ and $16$.

{\bf{1. Examples of resonance.}} Let $\phi(n)=2$ for $n=0$ and $\phi(n)=1$ for $n\neq0$ and define $$V_1(n)=-\frac{(\Delta^2\phi)(n)}{\phi(n)},\qquad V_2(n)=16+V_1(n).$$ Then
$$(\Delta^2+V_1)\phi=0,\qquad (\Delta^2+V_2)\phi=16\phi.$$
In this case, we have
$$V_1(n)=\left\{\begin{aligned}-3,&\quad n=0,\\4,&\quad n=\pm1,\\-1,&\quad n=\pm2,\\
0,&\quad {\rm else},\end{aligned}\right.,\qquad V_2(n)=\left\{\begin{aligned}13,&\quad n=0,\\20,&\quad n=\pm1,\\15,&\quad n=\pm2,\\
16,&\quad {\rm else}.\end{aligned}\right.$$
It indicates that $0$ remains a resonance even for such compactly supported potential, and by shifting the potential by $16$, one can turn the resonance at $0$ into a resonance at $16$.

{\bf{2. Examples of eigenvalue.}} Take $\phi(n)=(1+n^2)^{-s}$ with $s>\frac{1}{4}$, $V_1(n)=-\frac{(\Delta^2\phi)(n)}{\phi(n)}$ and $V_2(n)=16+V_1(n)$. In this case,
$$V_1(n)=O(\left<n\right>^{-4}),\qquad |n|\rightarrow\infty.$$
This implies that $0$ becomes an eigenvalue of $H$ for such slowly decaying potentials. However, as shown later in Lemma \ref{lemma of kernel and operaor}, zero-energy eigenvalues are excluded under the stronger decay assumptions imposed on $V$. Consequently, the zero-eigenvalue case need not be considered in Theorem \ref{main-theorem} below.

\vskip0.2cm
The main results of this paper are summarized as follows.
\begin{theorem}\label{main-theorem}
Let $H=\Delta^2+V$ with $|V(n)|\lesssim \left<n\right>^{-\beta}$ for some $\beta>0$. Suppose that $H$ has no positive eigenvalues in the interval $\mcaI=(0,16)$, and let $P_{ac}(H)$ denote the spectral projection onto the absolutely continuous spectrum space of $H$. If
 \begin{equation*}
 \beta>\left\{\begin{aligned}&15,\ 0\ is\ a\ regular\ point\ of\ H,\\
&19,\ 0\ is\ a\ first\ kind\ resonance\ of\ H,\\
&27,\ 0\ is\ a\ second\ kind\ resonance\ of\ H,
\end{aligned}\right.
\end{equation*}
then the following decay estimates hold:
\begin{equation}\label{eitH decay-estimate}
\|e^{-itH}P_{ac}(H)\|_{\ell^1\rightarrow\ell^{\infty}}\lesssim|t|^{-\frac{1}{4}},\qquad t\neq0,
\end{equation}
and
\begin{equation}\label{cos-sin decay-estimate}
\|{\rm cos}(t\sqrt H)P_{ac}(H)\|_{\ell^1\rightarrow\ell^{\infty}}+\left\|\frac{{\rm sin}(t\sqrt H)}{t\sqrt H}P_{ac}(H)\right\|_{\ell^1\rightarrow\ell^{\infty}}\lesssim|t|^{-\frac{1}{3}},\qquad t\neq0.
\end{equation}
\end{theorem}
\begin{remark}\label{Rem of M}{\rm Some remarks on Theorem \ref{main-theorem} are given as follows.}
{\rm \begin{itemize}
\item[{\rm (i)}] When $V\equiv0$, the decay rates in \eqref{eitH decay-estimate} and \eqref{cos-sin decay-estimate} are sharp (see Section \ref{Sec of decay for free}). In the perturbed case, it shows that the presence of zero-energy resonances does not alter these decay rates.
\vskip0.1cm
\item[{\rm (ii)}] Although resonances of possibly different types may occur simultaneously at the two thresholds $0$ and $16$, we emphasize that the decay assumption on $V$ is fundamentally determined by the resonance type at zero energy. The required decay exponent $\beta$ of $V$ in Theorem \ref{main-theorem}, arising from the asymptotic expansion of $R^\pm_V(\mu^4)$ (see Theorem \ref{Asymptotic expansion theorem} below), might not be optimal. It would therefore be interesting to investigate whether there exists a minimal decay condition on $V$ under which the above dispersive estimates remain valid.
\vskip0.1cm
\item[{\rm (iii)}] The absence of embedded positive eigenvalues is a standard assumption in the derivation of dispersive estimates of this type. For $H=\Delta^2+V$ on $\Z$, Horishima and L\H{o}rinczi showed in \cite{HL14} that no eigenvalues occur in $\mcaI$ for the point interaction $V(n)=c\delta_0(n)$ with $c\neq0$. In contrast to the extensive literature on embedded eigenvalues for discrete Schr\"odinger operators $-\Delta+V$ (see, e.g., \cite{HMO11,HSSS12,Kru12,IM14,BS12,HHNO16,Man19,Liu19,HMK22,Liu22,LMT24}), corresponding results for higher-order discrete operators appear to be much less developed. It would therefore be of interest to establish more general criteria ensuring the absence of embedded positive eigenvalues in the higher-order setting.

\end{itemize}
}
\end{remark}

\subsection{The idea of proof}\label{subsec of idea of proof}
In this subsection, we outline the main ideas behind the proof of Theorem \ref{main-theorem}.
Throughout this paper, we denote by $K$ the operator with kernel $K(n,m)$, that is,
 $$(Kf)(n):=\sum\limits_{m\in\Z}^{}K(n,m)f(m).$$

To prove Theorem \ref{main-theorem}, we use the identities
 \begin{equation*}
{\rm cos}(t\sqrt H)=\frac{e^{-it\sqrt H}+e^{it\sqrt H}}{2},\quad \frac{{\rm sin}(t\sqrt H)}{t\sqrt H}=\frac{1}{2t}\int_{-t}^{t}{\rm cos}\left(s\sqrt H\right)ds.
\end{equation*}
 It therefore suffices to establish the estimates \eqref{eitH decay-estimate} and \eqref{cos-sin decay-estimate} for $e^{-itH}P_{ac}(H)$ and $e^{-it\sqrt H}P_{ac}(H)$, respectively. By Stone's formula, their kernels are given by
\begin{align}
(e^{-itH}P_{ac}(H))(n,m)&=\frac{2}{\pi i}\int_{0}^{2}e^{-it\mu^4}\mu^3\big(R^{+}_V(\mu^4)-R^{-}_V(\mu^4)\big)(n,m)d\mu,\label{kernel of eitHPacH}\\
\big(e^{-it\sqrt H}P_{ac}(H)\big)(n,m)&=\frac{2}{\pi i}\int_{0}^{2}e^{-it\mu^2}\mu^3\big(R^{+}_V(\mu^4)-R^{-}_V(\mu^4)\big)(n,m)d\mu,\label{kernel of eitsqrtHPacH}
\end{align}
where $R^{\pm}_{V}(\mu^4)$  denote the boundary values of the resolvent of $H$:
\begin{equation*}
  R^{\pm}_{V}(\mu^4):=\lim\limits_{\varepsilon>0,\ \varepsilon\rightarrow0}(H-(\mu^4\pm i\varepsilon ))^{-1},\qquad\mu\in(0,2).
  \end{equation*}
The essential difference between \eqref{kernel of eitHPacH} and
\eqref{kernel of eitsqrtHPacH} lies in the oscillatory phases
$e^{-it\mu^4}$ and $e^{-it\mu^2}$, respectively. This difference leads to the decay rates $|t|^{-1/4}$ and $|t|^{-1/3}$. Since the arguments for the two evolutions are largely parallel, in what follows we focus on three main issues arising in the estimate of \eqref{kernel of eitHPacH}.
\subsubsection{Limiting absorption principle}\label{subsubsection of LAP}
In view of \eqref{kernel of eitHPacH}, the first issue is to establish the existence of the boundary values $R_V^\pm(\mu^4)$ for every $\mu\in(0,2)$, assuming that $H$ has no positive eigenvalues in $\mcaI$.

The limiting absorption principle (LAP) concerns the existence of suitable boundary values of the resolvent as the spectral parameter approaches the continuous spectrum and plays a fundamental role in spectral and scattering theory. For instance, we refer to Agmon's work \cite{Agm75} for Schr\"odinger operators $-\Delta_{\R^d}+V$ on $\R^d$. In the discrete setting, the LAP for discrete Schr\"odinger operators $-\Delta_{\Z^d}+V$ on $\Z^d$ has also been extensively studied; see, e.g., \cite{Esk67,SV01,BS98,BS99,KKK06,KKV08,IK12,Man17,PS08,TT19} and references therein.

However, to the best of our knowledge, it seems that LAP is  open for higher-order Schr\"{o}dinger operators on the lattice $\Z^d$. Motivated by this, we use the commutator estimates of Jensen, Mourre, and Perry \cite{JMP84}, together with Mourre theory (see, e.g., \cite{ABG96,Mou81,Mou83}), to establish the LAP for $H=\Delta^2+V$. More precisely, under suitable assumptions on $V$, we show that the boundary values $R_V^\pm(\mu^4)$ exist in $\B(s,-s)$ for every $s>\frac12$ (see Theorem \ref{LAP-theorem}).

\subsubsection{Asymptotic expansions of $R^{\pm}_{V}(\mu^4)$}\label{subsubsection of asy}
Once the boundary values $R_V^\pm(\mu^4)$ are established, the second main issue is to derive their asymptotic expansions near $\mu=0$ and $\mu=2$.

To this end, let $R^{\pm}_{0}(\mu^4)$ be the boundary value of the free resolvent $R_{0}(z):=(\Delta^2-z)^{-1}$. For any $\mu\in(0,2)$, set
\begin{equation}\label{Mpmmu}
M^{\pm}(\mu)=U+vR^{\pm}_{0}(\mu^4)v,\qquad U(n)=\begin{cases}
1,&V(n)\geq0\\
-1,&V(n)<0
\end{cases}\ ,\qquad v(n)=\sqrt{|V(n)|}.
\end{equation}
By Corollary \ref{corollary inverse of M},
$M^{\pm}(\mu)$ is invertible on $\ell^2(\Z)$, and we have 
 \begin{equation}\label{reso identity 1}
R^{\pm}_V(\mu^4)=R^{\pm}_{0}(\mu^4)-R^{\pm}_0(\mu^4)v\left(M^{\pm}(\mu\right))^{-1}vR^{\pm}_0(\mu^4).
\end{equation}
The problem is therefore reduced to analyzing the asymptotic expansions of $R^{\pm}_0(\mu^4)$ and $\bigl(M^{\pm}(\mu)\bigr)^{-1}$ near $\mu=0$ and $\mu=2$. 

Unlike the continuous kernel (see Remark \ref{remark of R0 kernel lema}), the kernels of boundary values $R^{\pm}_0(\mu^4)$ exhibit the following more complex structure (see Lemma \ref{R0 kernel lemma}):
 \begin{equation*}
  \ R^{\pm}_0(\mu^4,n-m)=\frac{1}{4\mu^3}\left(\frac{\pm ie^{\mp i\theta(\mu)|n-m|}}{\sqrt{1-\frac{\mu^2}{4}}}-\frac{e^{b(\mu)|n-m|}}{\sqrt{1+\frac{\mu^2}{4}}}\right),\qquad b(\mu)=-2\ln \Big(\sqrt{1+\frac{\mu^2}{4}}+\frac{\mu}{2}\Big),
  \end{equation*}
  where $\theta(\mu)\in(-\pi,0)$ satisfies $\cos\theta(\mu)=1-\frac{\mu^2}{2}$. As shown in Lemma \ref{Puiseux expan of R0}, these kernels admit the formal asymptotic expansions
\begin{align*}
R^{\pm}_0(\mu^4,n-m)\thicksim \sum\limits_{j=-3}^{+\infty}\mu^{j}G^{\pm}_j(n,m),\qquad
(JR^{\pm}_0\big((2-\mu)^4\big)J)(n,m)\thicksim\sum\limits_{j=-1}^{+\infty}\mu^{\frac{j}{2}}\widetilde{G}^{\pm}_j(n,m),\quad \mu\rightarrow0,
 \end{align*}
where $J$ is a unitary operator on $\ell^2(\Z)$ given by
\begin{equation}\label{J}
(J\phi)(n)=(-1)^{n}\phi(n), \quad n\in\Z,\ \phi\in\ell^2(\Z).
\end{equation}
 
Given the above expansions of $R_0^\pm(\mu^4)$, we next analyze $\bigl(M^\pm(\mu)\bigr)^{-1}$ by an iterative reduction procedure. In this respect, Jensen and Kato initiated their seminal work in \cite{JK79} for Schr\"odinger operator $-\Delta_{\R^3}+V$ on $\R^3$, while the unified reduction scheme adopted here was developed by Jensen and Nenciu \cite{JN01}; see also \cite{SWY22} and references therein for a higher-order setting.

When considering the discrete bi-Laplacian $\Delta^2$ on the lattice $\Z$, we face two distinct difficulties.
First, in contrast to the discrete Laplacian $-\Delta$ on $\Z$, the threshold $0$ now is a {\bf degenerate critical value} of the symbol $\mcaM(x)=(2-2{\rm cos}x)^2$~(i.e., $\mcaM(0)=\mcaM^{'}(0)=\mcaM^{''}(0)=0$). This degeneracy makes the successive inversion of $M^\pm(\mu)$ near $\mu=0$ more involved. Second, unlike the continuous analogue \cite{SWY22}, the discrete model has an additional threshold at $16$, corresponding to $\mu=2$.

To present these expansions, we introduce some spaces and operators. Let
$$\left<f,g\right>:=\sum\limits_{m\in\Z}^{}f(m)\overline{g(m)},\qquad f,g\in\ell^2(\Z).$$
Set
\begin{equation}\label{P}
P:=\left\|V\right\|^{-1}_{\ell^1}\left<\cdot,v\right>v,\qquad \widetilde{P}:=JPJ^{-1}=\left\|V\right\|^{-1}_{\ell^1}\left<\cdot,\tilde{v}\right>\tilde{v},\qquad \tilde{v}=Jv,
\end{equation}
be the orthogonal projections onto the span of $v$, $\tilde{v}$ in $\ell^2(\Z)$, respectively.
Let 
$$G_0(n,m)=\frac{1}{12}\left(|n-m|^3-|n-m|\right),\qquad\widetilde{G}_{0}(n,m)=\frac{1}{32\sqrt2}\big(2\sqrt2|n-m|-{(2\sqrt2-3)}^{|n-m|}\big).$$ In fact, $G_0,J\widetilde{G}_0J$ are  the fundamental solutions of $\Delta^2$ and $\Delta^2-16$, respectively, namely, $$\Delta^2G_0=\delta_0,\qquad(\Delta^2-16)J\widetilde{G}_0J=\delta_0.$$ 
Let 
$T_0=U+vG_0v$ and $\widetilde{T}_0=U+\tilde{v}\widetilde{G}_0\tilde{v}.$ We define the operator families $\{S_j\}$ and $\{\widetilde{S}_k\}$ as follows.
\begin{definition}\label{definition of Sj}
Let $Q=I-P$~{\rm(}where $I$ is the identity operator{\rm)} and $v_k(n)=n^kv(n),k\in\N$. For any $j\in\{0,1,2,3\}$, define $S_j$ as the orthogonal projection onto the space $S_j\ell^2(\Z)$ with
\begin{itemize}
\item[{\rm(i)}]$S_0\ell^2(\Z):=\big\{f\in\ell^2(\Z):\left<f,v_{k}\right>=0,\ k=0,1\big\}=\left({\rm span}\{v,v_1\}\right)^{\bot}$,
\item[{\rm(ii)}]$S_1\ell^2(\Z):=\big\{f\in\ell^2(\Z):\left<f,v_{k}\right>=0,\ k=0,1,\ S_0T_0f=0\big\}$,
\item[{\rm(iii)}] $S_2\ell^2(\Z):=\big\{f\in\ell^2(\Z):\left<f,v_{k}\right>=0,\ k=0,1,2,\ QT_0f=0\big\}$,
\item[{\rm(iv)}] $S_3\ell^2(\Z):=\big\{f\in\ell^2(\Z):\left<f,v_{k}\right>=0,\ k=0,1,2,3,\ T_0f=0\big\}$.
\end{itemize}
\end{definition}
\begin{definition}\label{definition of Sjtilde}
Let $\widetilde{Q}=I-\widetilde{P}$, $\tilde{v}_k=Jv_k,k\in\N$ and 
$$\widetilde{G}_{2}(n,m)=-\frac{1}{24}|n-m|^3+\frac{5}{48}|n-m|-\big(2\sqrt2-3\big)^{|n-m|}\Big(\frac{1}{64}|n-m|+\frac{7\sqrt{2}}{256}\Big).$$
For any $j\in\{0,1,2\}$, define $\widetilde{S}_j$ as the orthogonal projection onto the space $\widetilde{S}_j\ell^2(\Z)$ with
\begin{itemize}
\item[{\rm(i)}]$\widetilde{S}_0\ell^2(\Z):=\big\{f\in\ell^2(\Z):\left<f,\tilde{v}\right>=0,\ \widetilde{Q}\widetilde{T}_0f=0\big\}$,
\item[{\rm(ii)}]$\widetilde{S}_1\ell^2(\Z):=\big\{f\in\ell^2(\Z):\left<f,\tilde{v}_k\right>=0,\ k=0,1,\ \widetilde{T}_0f=0\big\}$,
\item[{\rm(iii)}]$\widetilde{S}_2\ell^2(\Z):=\big\{f\in\ell^2(\Z):\left<f,\tilde{v}_k\right>=0,\ k=0,1,\ \widetilde{T}_0f=0,\ \widetilde{S}_1\tilde{v}\widetilde{G}_2\tilde{v}\widetilde{S}_1f=0\big\}$.
\end{itemize}
\end{definition}
\begin{remark} {\rm (i) By definition,
$Q\ell^2(\Z)=\left({\rm span}\{v\}\right)^{\bot}$ and $\widetilde{Q}\ell^2(\Z)=\left({\rm span}\{\tilde{v}\}\right)^{\bot}.$
Moreover, the spaces defined above satisfy the following inclusion relations:
\begin{equation*}
S_3\ell^2(\Z)\subseteq S_2\ell^2(\Z)\subseteq S_1\ell^2(\Z)\subseteq S_0\ell^2(\Z)\subseteq Q\ell^2(\Z),\qquad\widetilde{S}_2\ell^2(\Z)\subseteq\widetilde{S}_1\ell^2(\Z)\subseteq \widetilde{S}_0\ell^2(\Z) \subseteq \widetilde{Q}\ell^2(\Z).
\end{equation*}
Note that, for $\beta>7$, $UvG_0v$ is compact. Hence, $T_0=U(I+UvG_0v)$ is Fredholm of index zero. Since $P_0=I-S_0$ has finite rank, it follows that $S_0T_0S_0$
is a finite-rank perturbation of $T_0$, and is therefore also Fredholm of index zero. In particular, $\dim S_1\ell^2<\infty$. By the inclusion relations above, it follows that $\dim S_j\ell^2<\infty$ for $j=2,3$. Similarly, $\dim \widetilde{S}_k\ell^2<\infty$ for $k=1,2$.

(ii) As will be shown in Lemma \ref{lemma of kernel and operaor}, $S_3\ell^2(\Z)=\{0\}=\widetilde{S}_2\ell^2(\Z)$. These facts terminate the iterative reduction after finitely many steps and hence complete the inversion procedure for $\bigl(M^\pm(\mu)\bigr)^{-1}$. For details, see Section \ref{proof of Asy theorem}.
}
\end{remark}
We point out that the projection spaces introduced in Definitions \ref{definition of Sj} and \ref{definition of Sjtilde} also characterize the resonance types of $H$ at the thresholds $0$ and $16$, respectively, as stated in the following theorem.
\begin{theorem}\label{relation space and resonance types}
 Let $H=\Delta^2+V$ with $|V(n)|\lesssim \left<n\right>^{-\beta}$ for $\beta>9$. Then we have
\begin{itemize}
\item[{\rm(i)}] $0$ is a regular point of $H$ if and only if $S_1\ell^2(\Z)=\{0\}$.
\item[{\rm(ii)}] $0$ is a first kind resonance of $H$ if and only if $S_1\ell^2(\Z)\neq\{0\}$ and $S_2\ell^2(\Z)=\{0\}$.
\item[{\rm(iii)}] $0$ is a second kind resonance of $H$ if and only if $S_2\ell^2(\Z)\neq\{0\}$.

\item[{\rm(iv)}] $16$ is a regular point of $H$ if and only if $\widetilde{S}_0\ell^2(\Z)=\{0\}$.
\item[{\rm(v)}] $16$ is a resonance of $H$ if and only if $\widetilde{S}_0\ell^2(\Z)\neq\{0\}$ and $\widetilde{S}_{1}\ell^2(\Z)=\{0\}$.
\item[{\rm(vi)}] $16$ is an eigenvalue of $H$ if and only if $\widetilde{S}_{1}\ell^2(\Z)\neq\{0\}$.
\end{itemize}
\end{theorem}
We say that an integral operator $K$ is absolutely bounded on $\ell^2(\Z)$ if the integral operator $|K|$ with kernel $|K(n,m)|$ is also bounded on $\ell^2(\mathbb Z)$. The expansions of $\bigl(M^\pm(\mu)\bigr)^{-1}$ are given as follows.
\begin{theorem}\label{Asymptotic expansion theorem}
Let $H=\Delta^2+V$ with $|V(n)|\lesssim \left<n\right>^{-\beta}$ for some $\beta>0$. 
Then, given a sufficiently small $0<\mu_0\ll1$, the following asymptotic expansions hold in $\B(\ell^2(\Z))$ for any $0<\mu<\mu_0$.
\begin{itemize}
\item[{\rm(i)}] If $0$ is a regular point of $H$ and $\beta>15$, then
\begin{align}\label{asy expan of regular 0}
\left(M^{\pm}(\mu)\right)^{-1}=S_0A_{0}S_0+\mu QA^{\pm,0}_{11}Q+\mu^2(QA^{\pm,0}_{21}Q+S_0A^{\pm,0}_{22}+A^{\pm,0}_{23}S_0)+\mu^3A^{\pm,0}_3+\Gamma^{0}_{4}(\mu),
\end{align}
\item [{\rm(ii)}] If $0$ is a first kind resonance of $H$ and $\beta>19$, then
\begin{align}\label{asy expan of 1st 0}
\begin{split}
\left(M^{\pm}(\mu)\right)^{-1}&=\mu^{-1}S_{1}A^{\pm}_{-1}S_{1}+\big(S_0A^{\pm,1}_{01}Q+QA^{\pm,1}_{02}S_0\big)+\mu\big(S_0A^{\pm,1}_{11}+A^{\pm,1}_{12}S_0+QA^{\pm,1}_{13}Q\big)\\
&\quad+\mu^2\big(QA^{\pm,1}_{21}+A^{\pm,1}_{22}Q\big)+\mu^3A^{\pm,1}_{3}+\Gamma^{1}_{4}(\mu),
\end{split}
\end{align}
\item [{\rm(iii)}] If $0$ is a second kind resonance of $H$ and $\beta>27$, then
\begin{align}\label{asy expan 2nd 0}
\left(M^{\pm}(\mu)\right)^{-1}&=\frac{S_{2}A^{\pm}_{-3}S_{2}}{\mu^3}+\frac{S_2A^{\pm}_{-2,1}S_0+S_0A^{\pm}_{-2,2}S_2}{\mu^2}
+\frac{S_2A^{\pm}_{-1,1}Q+QA^{\pm}_{-1,2}S_2+S_0A^{\pm}_{-1,3}S_0}{\mu}\notag\\
 &\quad+\big(S_2A^{\pm,2}_{01}+A^{\pm,2}_{02}S_2+QA^{\pm,2}_{03}S_0+S_0A^{\pm,2}_{04}Q\big)+\mu\big(S_0A^{\pm,2}_{11}+A^{\pm,2}_{12}S_0+QA^{\pm,2}_{13}Q\big)\notag\\
&\quad+\mu^2\big(QA^{\pm,2}_{21}+A^{\pm,2}_{22}Q\big)+\mu^3A^{\pm,2}_{3}+\Gamma^{2}_{4}(\mu),
\end{align}
\item[{\rm(iv)}] If $16$ is a regular point of $H$ and $\beta>7$, then
\begin{align}\label{asy expan regular 2}
\left(M^{\pm}(2-\mu)\right)^{-1}=\widetilde{Q}B_{0}\widetilde{Q}+
\mu^{\frac{1}{2}}B^{\pm,0}_1 +\Gamma^{0}_{1}(\mu),
\end{align}
\item [{\rm(v)}] If $16$ is a resonance of $H$ and $\beta>11$, then
\begin{align}\label{asy expan resonance 2}
\left(M^{\pm}(2-\mu)\right)^{-1}=\mu^{-\frac{1}{2}}\widetilde{S}_0B^{\pm}_{-1}\widetilde{S}_0+\big(\widetilde{S}_0B^{\pm,1}_{01}+B^{\pm,1}_{02}\widetilde{S}_0+\widetilde{Q}B^{\pm,1}_{03}\widetilde{Q}\big)+
\mu^{\frac{1}{2}}B^{\pm,1}_1 +\Gamma^{1}_{1}(\mu),
\end{align}
\item [{\rm(vi)}] If $16$ is an eigenvalue of $H$ and $\beta>15$, then
\begin{align}\label{asy expan eigenvalue 2}
\left(M^{\pm}(2-\mu)\right)^{-1}&=\mu^{-1}\widetilde{S}_1B_{-2}\widetilde{S}_1+\mu^{-\frac{1}{2}}\big(\widetilde{S}_0B^{\pm}_{-1,1}\widetilde{Q}+\widetilde{Q}B^{\pm}_{-1,2}\widetilde{S}_0\big)\notag\\
&\quad +\big(\widetilde{Q}B^{\pm,2}_{01}+B^{\pm,2}_{02}\widetilde{Q}\big)+
\mu^{\frac{1}{2}}B^{\pm,2}_1 +\Gamma^{2}_{1}(\mu),
\end{align}
\end{itemize}
where
$A_{0},A^{\pm}_{-1},A^{\pm}_{-3},A^{\pm}_{j,k},A^{\pm,i}_{jk},A^{\pm,i}_3,B_{0},B^{\pm}_{-1},B_{-2},B^{\pm}_{j,k},B^{\pm,i}_{jk},B^{\pm,i}_{1}$ are $\mu$-independent bounded operators on $\ell^2(\Z)$, Moreover, all the operators appearing on the right-hand
sides of \eqref{asy expan of regular 0}--\eqref{asy expan eigenvalue 2}
are absolutely bounded. The remainder terms $\Gamma^{i}_\ell(\mu)$ satisfy 
\begin{equation}\label{estimate of Gamma}
\|\Gamma^{i}_{\ell}(\mu)\|_{\ell^2\rightarrow\ell^2}+\mu\big\|\partial_{\mu}(\Gamma^{i}_{\ell}(\mu))\big\|_{\ell^2\rightarrow\ell^2}\lesssim\mu^{\ell},\qquad 0<\mu<\mu_0.
\end{equation}
\end{theorem}
The proofs of Theorems \ref{Asymptotic expansion theorem} and \ref{relation space and resonance types} are given in Sections \ref{proof of Asy theorem} and \ref{proof of relation space and resonance types}, respectively.

\subsubsection{Treatment of oscillatory integral}
Equipped with the two tools developed above, the final step is to estimate the oscillatory integral in \eqref{kernel of eitHPacH} by the van der Corput lemma \cite[pp.~332--334]{Ste93}. To this end, we decompose \eqref{kernel of eitHPacH} into three parts: 
\begin{equation*}
e^{-itH}P_{ac}(H)=\frac{2}{\pi i}\Big(\int_{0}^{\mu_0}+\int_{\mu_0}^{2-\mu_0}+\int_{2-\mu_0}^{2}\Big)e^{-it\mu^4}\mu^3\big(R^{+}_V(\mu^4)-R^{-}_V(\mu^4)\big)d\mu,
\end{equation*}
where $\mu_0>0$ is a sufficiently small fixed constant. Substituting \eqref{reso identity 1} into the first and third integrals and the following resolvent identity into the middle integral, 
\begin{equation*}
R^{\pm}_V(\mu^4)=R^{\pm}_{0}(\mu^4)-R^{\pm}_{0}(\mu^4)VR^{\pm}_{0}(\mu^4)+R^{\pm}_{0}(\mu^4)VR^{\pm}_V(\mu^4)VR^{\pm}_{0}(\mu^4),
\end{equation*}
we obtain
\begin{equation*}
(e^{-itH}P_{ac}(H))(n,m)=-\frac{2}{\pi i}\sum\limits_{j=0}^{3}(K^{+}_{j}-K^{-}_{j})(t,n,m),
\end{equation*}
where
\begin{align}\label{kernels of Ki}
\begin{split}
K^{\pm}_{0}(t,n,m)&=\int_{0}^{2}e^{-it\mu^4}\mu^3R^{\mp}_0(\mu^4,n,m)d\mu,\\
K^{\pm}_{1}(t,n,m)&=\int_{0}^{\mu_0}e^{-it\mu^4}\mu^3\big[R^{\pm}_0(\mu^4)v\big(M^{\pm}(\mu)\big)^{-1}vR^{\pm}_0(\mu^4)\big](n,m)d\mu,\\
K^{\pm}_{2}(t,n,m)&=\int_{\mu_0}^{2-\mu_0}e^{-it\mu^4}\mu^3\big[R^{\pm}_0(\mu^4)VR^{\pm}_0(\mu^4)-R^{\pm}_0(\mu^4)VR^{\pm}_V(\mu^4)VR^{\pm}_0(\mu^4)\big](n,m)d\mu,\\
K^{\pm}_{3}(t,n,m)&=\int_{2-\mu_0}^{2}e^{-it\mu^4}\mu^3\big[R^{\pm}_0(\mu^4)v\big(M^{\pm}(\mu)\big)^{-1}vR^{\pm}_0(\mu^4)\big](n,m)d\mu.
\end{split}
\end{align}
Thus, it suffices to establish the decay estimate \eqref{eitH decay-estimate} for each $K_j^+-K_j^-$, which will be treated in Sections \ref{Sec of decay for free} and \ref{sec of proof}.
\subsection{Organizations of the paper}
The remainder of this paper is organized as follows. Section \ref{sec of LAP} contains some preliminaries, including basic properties of the free resolvent and the limiting absorption principle (Theorem \ref{LAP-theorem}). In Section \ref{Sec of decay for free}, we establish the decay estimate in the free case and prove its sharpness. Section \ref{sec of proof} is devoted to estimating the kernels $(K_j^+-K_j^-)(t,n,m)$ defined in \eqref{kernels of Ki} for $j=1,2,3$. Sections \ref{proof of Asy theorem} and \ref{proof of relation space and resonance types} contain the proofs of Theorems \ref{Asymptotic expansion theorem} and \ref{relation space and resonance types}, respectively. Finally, Appendix \ref{section of Appendix} provides a brief review of commutator estimates and Mourre theory.
\section{Limiting absorption principle}\label{sec of LAP}
\subsection{Free resolvent}\label{Subsec of free resol}
This subsection presents some basic properties of the  bi-Laplacian $\Delta^2$ on $\Z$. 
Recalling the definition of $\Delta$ in \eqref{definition of Laplacian}, the bi-Laplacian $\Delta^2$ is given by 
\begin{equation*}
(\Delta^2\phi)(n)=(\Delta(\Delta\phi))(n)=\phi(n+2)-4\phi(n+1)+6\phi(n)-4\phi(n-1)+\phi(n-2).
\end{equation*}
Consider the Fourier transform $\mcaF$: $\ell^2(\Z)\rightarrow L^2(\T)$ defined by
\begin{equation}\label{fourier transform}
(\mcaF\phi)(x):=(2\pi)^{-{\frac{1}{2}}}\sum_{n\in\Z}^{}e^{-inx}\phi(n), \quad x\in\T=\R/(2\pi\Z)\simeq(-\pi,\pi],\quad  \phi\in\ell^2(\Z).
\end{equation}
Under the Fourier transform, we have
\begin{equation}\label{unitary equivalent}
(\mcaF\Delta^2\phi)(x)=\mcaM(x)(\mcaF\phi)(x),\qquad \mcaM(x)=(2-2 \cos x)^2.
\end{equation}
It follows that the spectrum of $\Delta^2$ is purely absolutely
continuous and is given by $[0,16]$.
Let
  $$R_0(z):=(\Delta^2-z)^{-1},\qquad z\in\C\setminus[0,16],$$
  be the resolvent of $\Delta^2$, and define its boundary values on $(0,16)$ by
  \begin{align*}
  R^{\pm}_{0}(\lambda):=\lim\limits_{\varepsilon>0,\ \varepsilon\rightarrow0 }R_{0}(\lambda\pm i\varepsilon ),\qquad\lambda\in(0,16).
  \end{align*}
   Denote by $\B(s,s')$ the space of all bounded linear operators from $\ell^{2,s}(\Z)$ to $\ell^{2,s'}(\Z)$. The existence of $R^{\pm}_0(\lambda)$ in  $\B(s,-s)$ for $s>\frac{1}{2}$ follows from the resolvent decomposition
\begin{equation}\label{unity partition}
 R_0(z)=\frac{1}{2\sqrt z}\left(R_{-\Delta}(\sqrt z)-R_{-\Delta}(-\sqrt z)\right),\quad\sqrt{z}=\sqrt{|z|}e^{i\frac{\arg z}{2}},\quad  0<\arg z<2\pi,
 \end{equation}
together with the limiting absorption principle for $-\Delta$~(see, e.g., \cite[Lemma 3.1]{KKK06})
  $$R^{\pm}_{-\Delta}(\mu):=\lim\limits_{\varepsilon>0,\ \varepsilon\rightarrow0}R_{-\Delta}(\mu\pm i\varepsilon )\quad {\rm in\ the\ operator\ norm\ of }\ \B(s,-s)\ {\rm for}\ s>\frac{1}{2},\quad \mu\in(0,4),$$
where $R_{-\Delta}(\omega)=(-\Delta-\omega)^{-1}$ is the resolvent of $-\Delta$.

Moreover, we recall that the kernel of $R_{-\Delta}(\omega)$ is given as follows.
\begin{lemma}\label{kernel of lapla}
{\rm(\cite[Lemma 2.1]{KKK06})} For $\omega\in\C\setminus[0,4]$, the kernel of the resolvent $R_{-\Delta}(\omega)$ is given by
\begin{equation}\label{kernel of lapl resolvent}
R_{-\Delta}(\omega,n-m)=\frac{-ie^{-i\Theta(\omega)|n-m|}}{2\sin\Theta(\omega)},\quad n,m\in\Z,
\end{equation}
where $\Theta(\omega)$ is the solution of the equation
\begin{equation}\label{map}
2-2\cos\Theta=\omega
\end{equation}
in the domain $\mcaD:=\left\{\Theta=a+ib~|-\pi\leq a\leq\pi,b<0\right\}$.
\end{lemma}
\begin{remark}
{\rm 

(1) Let
$$\C^{\pm}=\{\omega=x\pm iy~|~y>0\},\qquad \mcaD_{\mp}=\{\Theta(\omega)=a+ ib\in\mcaD~|\pm a<0\}$$ and introduce the following oriented lines and line segments
$\ell_i$, $\ell_i'$, and $\widetilde{\ell}_i$:
 \begin{align*}
 &\ell_1=\{x~|~x:-\infty\rightarrow0\},\quad\ell_2=\{x~|~x:0\rightarrow4\},\quad\ell_3=\{x~|~x:4\rightarrow+\infty\},\\
 &\ell'_1=\{ib~|~b:-\infty\rightarrow0\}, \quad\ell'_2=\{a~|~a:0\rightarrow\pi\}, \quad\tilde{\ell}_2=\{a~|~a:-\pi\rightarrow0\},\\ &\ell'_{3}=\{\pi+ib~|~b:0\rightarrow-\infty\},\quad\tilde{\ell}_3=\{-\pi+ib~|~b:-\infty\rightarrow0\}.
 \end{align*}
Let $\ell_i^{-}$ denote the line $\ell_i$ with the opposite orientation. Then the mapping defined in \eqref{map}, $$\Theta(\omega):\ \C\setminus[0,4]\longrightarrow \mcaD~(\omega\mapsto\Theta(\omega))$$ gives the following correspondence (see Figure \ref{myplot} below).
\begin{figure}[htbp!]\label{figure}
\centering
\captionsetup{labelformat=empty}
\begin{minipage}{0.4\textwidth}
\centering
\begin{tikzpicture}[>=stealth]
\draw[->] (-2.5,0)--(2.5,0)node[below left]{$x$};
\draw[->]  (0,-2.5)--(0,2.5)node[below right]{$y$};
\draw[green,thick, fill=none] (0,0)circle (1.5pt)node[below left, black]{0};
\draw[blue, fill=none] (1,0)circle (1.5pt)node[below, black]{4};
\draw[red,thick](0.06,0)--(0.94,0);
\draw[->,red,thick](0.06,0)--(0.6,0);
\draw[green,thick] (-2.5,0)--(-0.06,0);
\draw[->,green,thick] (-2.5,0)--(-1,0);
\draw[blue,thick] (1.06,0)--(2.5,0);
\draw[->,blue,thick] (1.06,0)--(1.7,0);
\fill[pattern=north west lines, pattern color=magenta, opacity=0.5] (-2.5,0.06) rectangle (2.5,2.5);
\node at (-1.25,1.25){$\C^{+}$};
\fill[pattern=north west lines, pattern color=cyan, opacity=0.9] (-2.5,-2.5) rectangle (2.5,-0.06);
\node at (-1.25,-1.25){$\C^{-}$};
\node at (-1.3,0.2){$\textcolor{green}{\ell_1}$};
\node at (0.5,-0.3){$\textcolor{red}{\ell_2}$};
\node at (1.8,-0.3){$\textcolor{blue}{\ell_3}$};
\node[circle, draw, inner sep=1pt] at (2,2) {$\omega$};
\end{tikzpicture}
\end{minipage}
\begin{tikzpicture}[overlay, remember picture]
        \draw[->, thick] (0.2,1)--(2,1);
        \node[below, font=\bfseries] at (1,0.9){$\Theta(\C^{\pm})=\mcaD_{\mp}$};
        \node[below, font=\bfseries] at (1,0.4){$\Theta(\ell_i)=\ell'_{i}$};
        \node[below, font=\bfseries] at (1,-0.1){$\Theta(\ell^{-}_j)=\tilde{\ell}_{j}$};
        \node[below, font=\bfseries] at (1,-0.7){($i=1,2,3,\ j=2,3$)};
        \node[above, font=\bfseries] at (1,1.06) {${\rm \ \ cos}\Theta(\omega)=1-\frac{\omega}{2}$}; 
    \end{tikzpicture}
\hspace{0.1\textwidth}
\begin{minipage}{0.4\textwidth}
\centering
\begin{tikzpicture}[>=stealth]
\draw[->] (-2.5,0)--(2.5,0) node[below left]{$a$};
\draw[->]  (0,-2.5)--(0,2.5)node[below right]{$b$};
\draw[green,thick, fill=none] (0,0)circle (1.5pt)node[below left, black]{0};
\draw[blue, thick,fill=none] (1,0)circle (1.5pt)node[below right, black]{$\pi$};
\draw[blue, thick,fill=none] (-1,0)circle (1.5pt)node[below left, black]{$-\pi$};
\draw[red,thick](0.06,0)--(0.94,0);
\draw[->,red,thick](0.06,0)--(0.56,0);
\draw[red,thick](-0.94,0)--(-0.06,0);
\draw[->,red,thick](-0.94,0)--(-0.4,0);
\draw[green,thick](0,-2.5)--(0,-0.06);
\draw[->,green,thick](0,-2.5)--(0,-1);
\draw[blue,thick](-1,-2.5)--(-1,-0.06);
\draw[->,blue,thick](-1,-2.5)--(-1,-1);
\draw[blue,thick](1,-2.5)--(1,-0.06);
\draw[-<,blue,thick](1,-2.5)--(1,-1);
\fill[pattern=north west lines, pattern color=magenta, opacity=0.5] (-0.95,-2.5) rectangle (-0.06,-0.06);
\node at (-0.5,-1.25){$\mcaD_-$};
\fill[pattern=north west lines, pattern color=cyan, opacity=0.9] (0.05,-2.5) rectangle (0.95,-0.06);
\node at (0.5,-1.25){$\mcaD_+$};
\node at (0,-2.75){$\textcolor{green}{\ell'_1}$};
\node at (-1,-2.75){$\textcolor{blue}{\tilde{\ell}_{3}}$};
\node at (-0.6,0.3){$\textcolor{red}{\tilde{\ell}_{2}}$};
\node at (0.6,0.3){$\textcolor{red}{\ell'_{2}}$};
\node at (1,-2.75){$\textcolor{blue}{\ell'_{3}}$};
\node[circle, draw, inner sep=1pt] at (2,2) {$\Theta$};
\end{tikzpicture}
\end{minipage}
\caption{Figure 1: The map $\Theta(\omega)$ from $\C\setminus[0,4]$ to $\mcaD$.}
\label{myplot}
\end{figure}

(2) Note that if $\omega\in\C^{\pm}\cup(-\infty,0)$, the solution of
\eqref{map} is uniquely determined. Consequently, we have the following.
\begin{itemize}
\item If $\lambda\in (0,4)$, the kernel of $R^{\pm}_{-\Delta}(\lambda)$ is given by
\begin{equation}\label{kernel of lapa boundary}
R^{\pm}_{-\Delta}(\lambda,n-m)=\frac{\mp ie^{\mp i\Theta(\lambda)|n-m|}}{2\sin\Theta(\lambda)},
\end{equation}
where $\Theta(\lambda)\in(-\pi,0)$ satisfies the equation $2-2\cos\Theta=\lambda$.
\vskip0.1cm
\item If $\lambda\in(-\infty,0)$, then
\begin{equation}\label{expre of sin theta}
    R_{-\Delta}(\lambda,n-m)=\frac{\Big(1-\frac{\lambda}{2}-\frac{1}{2}\sqrt{\lambda^2-4\lambda}\Big)^{|n-m|}}{\sqrt{\lambda^2-4\lambda}}.
    \end{equation}
\end{itemize}
}
\end{remark}
As a consequence, we obtain the following explicit kernel representation for $R^{\pm}_0(\mu^4)$.
\begin{lemma}\label{R0 kernel lemma}
{ For $\mu\in(0,2)$, the kernel of $R^{\pm}_0(\mu^4)$ is given by
\begin{equation}\label{kernel of R0 boundary}
R^{\pm}_{0}(\mu^4,n-m)
=\frac{1}{4\mu^3}\Big(\pm ia_1(\mu)e^{\mp i\theta(\mu)|n-m|}+a_2(\mu)e^{b(\mu)|n-m|}\Big),
\end{equation}
where $\theta(\mu)\in(-\pi,0)$ satisfies $2-2\cos\theta(\mu)=\mu^2$ and
\begin{equation}\label{exp of a1 a2 bu}
a_1(\mu)=\frac{1}{\sqrt{1-\frac{\mu^2}{4}}},\quad a_2(\mu)=\frac{-1}{\sqrt{1+\frac{\mu^2}{4}}},\quad b(\mu)=-2\ln \Big(\sqrt{1+\frac{\mu^2}{4}}+\frac{\mu}{2}\Big).
\end{equation}
}
\end{lemma}
  \begin{proof}
 For any $\mu\in(0,2)$ and $\varepsilon>0$, substituting $z=\mu^4\pm i\varepsilon$ into \eqref{unity partition} and letting $\varepsilon\rightarrow0$ yields
\begin{equation}\label{R0mu4 and Rdeltamu2}
R^{\pm}_{0}(\mu^4)=\frac{1}{2\mu^2}\left(R^{\pm}_{-\Delta}(\mu^2)-R_{-\Delta}(-\mu^2)\right).
\end{equation}
Combining \eqref{kernel of lapa boundary} and \eqref{expre of sin theta}, we then obtain
\eqref{kernel of R0 boundary}.
\end{proof}
\begin{remark}\label{remark of R0 kernel lema}
{\rm Recall the resolvent kernel of the bi‑Laplacian $\frac{d^4}{dx^4}$ on $\R$  is given by 
 (cf. \cite{SWY22})
$$\mcaR^{\pm}_0(\mu^4,x-y)=\frac{1}{4\mu^3}\big(\pm ie^{\pm i\mu|x-y|}-e^{-\mu|x-y|}\big),\qquad x,y\in \R.$$
Unlike the continuous kernel, which has only a $\mu^{-3}$ singularity
at $\mu=0$, the discrete kernel has an additional
$(2-\mu)^{-\frac12}$ singularity at $\mu=2$. Moreover, its more
complicated structure makes the asymptotic analysis considerably more involved (see Section \ref{proof of Asy theorem}).
}
\end{remark}
\subsection{Perturbed case}\label{subsection of asympotic}
As shown in the previous subsection, the limiting absorption principle
(LAP) has been established for the free operator. We now turn to the perturbed operator.
\begin{theorem}\label{LAP-theorem}
{ Let $H=\Delta^2+V$ with $|V(n)|\lesssim \left<n\right>^{-\beta}$ for some $\beta>2$ and $\mcaI=(0,16)$. Denote by $[\beta]$ the greatest integer not exceeding $\beta$. Then the following statements hold.}
\begin{itemize}
\item [\rm(i)] The point spectrum $\sigma_p(H)\cap\mcaI$ is discrete, each eigenvalue has finite multiplicity and the singular continuous spectrum $\sigma_{sc}(H)=\varnothing$.
\vskip0.1cm
\item [\rm(ii)] Let $j\in\left\{0,\cdots,[\beta]-2\right\}$ and $s>j+\frac{1}{2}$. Then the following norm limits
\begin{equation*}
   \frac{d^{j}}{d\lambda^j}(R^{\pm}_V(\lambda))=\lim\limits_{\varepsilon>0,\ \varepsilon\rightarrow0}R^{(j)}_V(\lambda\pm i\varepsilon) \quad {\rm in} \quad \B(s,-s)
    \end{equation*}
exist and are norm-continuous in $\mcaI\setminus\sigma_p(H)$, where $R_{V}(z)=(H-z)^{-1}$ is the resolvent of $H$ and $R^{(j)}_{V}(z)$ denotes the j-th derivative of $R_{V}(z)$.
\end{itemize}
\end{theorem}
We remark that the proof of this LAP is based on the commutator estimates of Jensen, Mourre, and Perry \cite{JMP84} together with Mourre theory. The upper bound of $j$ is closely related to the regularity of $H$ with respect to a suitable conjugate operator. We refer to Appendix \ref{section of Appendix} for the relevant abstract results.

We now outline the main ideas of the proof. Following the approach of \cite{JMP84}, we first choose a suitable conjugate operator $A$ and establish estimates for the derivatives of the resolvent $R_V(z)$ in the Besov spaces associated with $A$, as in Theorem \ref{Resol-Smoo}. We then transfer these estimates to the weighted spaces $\ell^{2,s}(\Z)$, thereby obtaining the desired limiting absorption principle.

We begin by introducing the conjugate operator used for our problem. Let $\mcaN$ be the position operator defined by
$$(\mcaN\phi)(n):=n\phi(n),\qquad \phi\in\mcaD(\mcaN)=\big\{\phi\in\ell^2(\Z):\sum\limits_{n\in\Z}|n|^2|\phi(n)|^2<\infty\big\},$$
and let $\mcaP$ be the difference operator, with adjoint $\mcaP^*$, given by
$$(\mcaP\phi)(n):=\phi(n+1)-\phi(n),\qquad (\mcaP^{*}\phi)(n):=\phi(n-1)-\phi(n).$$
We then define the conjugate operator $A$ by
\begin{equation}\label{A}
A=-i\big(\mcaN\mcaP-\mcaP^{*}\mcaN\big).
\end{equation}
\begin{lemma}\label{LAP lemma}
{Let $H=\Delta^2+V$ with $|V(n)|\lesssim \left<n\right>^{-\beta}$ for some $\beta>2$ and let $A$ be defined as in \eqref{A}. Let  $\lambda\in\mcaI=(0,16)$. Then the following statements hold.

{\rm (1) (}Mourre estimate{\rm)} There exist $\alpha>0$, $\delta>0$, and a compact operator $K$ on $\ell^2(\Z)$ such that
\begin{equation}\label{mourre estim}
E_{H}(\mcaJ)iad^{1}_{A}(H)E_{H}(\mcaJ)\geq\alpha E_{H}(\mcaJ)+K,\qquad \mcaJ=(\lambda-\delta,\lambda+\delta),
\end{equation}
where $E_{H}(\mcaJ)$ represents the spectral projection of $H$ onto  $\mcaJ$ and $ad^{1}_{A}(H)$ is defined in \eqref{adk}.
\vskip0.1cm
{\rm(2)} For any relatively compact interval ${\rm I}\subseteq\mcaJ\setminus\sigma_{p}(H)$, $j\in\left\{0,\cdots,[\beta]-2\right\}$ and $s>j+\frac{1}{2}$, then
\begin{itemize}
\item [{\rm (i)}]\begin{equation}\label{jth esti of Lem}
\sup\limits_{\operatorname{Re}z\in {\rm I},\ \operatorname{Im}z\neq0}\left\|\left<A\right>^{-s}R^{(j)}_V(z)\left<A\right>^{-s}\right\|<\infty;
\end{equation}
\item[{\rm (ii)}]  $\left<A\right>^{-s}R^{(j)}_V(z)\left<A\right>^{-s}$ is H\"{o}lder continuous on $\widetilde{\rm I}:=\{z:\operatorname{Re} z\in {\rm I},\ 0<|\operatorname{Im} z|\leq1\}$ with the exponent $\delta(s,j+1)$ defined in \eqref{delta};
    \vskip0.1cm
\item [{\rm (iii)}] let $\mu\in {\rm I}$, the norm limits
$$\lim\limits_{\varepsilon>0,\ \varepsilon\rightarrow0}\left<A\right>^{-s}R^{(j)}_V(\mu\pm i\varepsilon)\left<A\right>^{-s}$$
exist and equal
$$\frac{d^{j}}{d\mu^j}(\left<A\right>^{-s}R^{\pm}_V(\mu)\left<A\right>^{-s}),$$
where
$$\left<A\right>^{-s}R^{\pm}_V(\mu)\left<A\right>^{-s}:=\lim\limits_{\varepsilon>0,\ \varepsilon\rightarrow0}\left<A\right>^{-s}R_V(\mu\pm i\varepsilon)\left<A\right>^{-s}.$$
The norm limits are H\"{o}lder continuous with exponent $\delta(s,j+1)$ given by \eqref{delta}.
\end{itemize}
}
\end{lemma}
\begin{proof}
(1) For notational simplicity,  set $H_0=\Delta^2$ and write $[\cdot,A]=ad^{1}_{A}(\cdot)$. Fix $\lambda\in\mcaI=(0,16)$. By the commutator identity established in Lemma \ref{regularity of H},
$$[H_0,iA]=2H_0(4-\sqrt {H_{0}})=:g(H_0),\qquad g(x)=2x(4-\sqrt x),\qquad 0\leq x\leq16.$$
Since $g(x)>0$ for $x\in(0,16)$, let
$\delta_0=\frac{1}{2}\min(\lambda, 16-\lambda)$ and 
$$\alpha=\min\limits_{x\in \mcaJ_1}g(x)>0,\qquad \mcaJ_1=[\lambda-\delta_0,\lambda+\delta_0].$$
Using functional calculus, we obtain
 \begin{equation}\label{J1}
 E_{H_0}(\mcaJ_1)[H_0,iA]E_{H_0}(\mcaJ_1)\geq \alpha E_{H_0}(\mcaJ_1).
 \end{equation}
Choose $0<\delta<\delta_0$ and $\mcaJ=(\lambda-\delta,\lambda+\delta)$. Multiplying both sides of \eqref{J1} by $E_{H_0}(\mcaJ)$, we derive
\begin{equation}\label{H0}
E_{H_0}(\mcaJ)[H_0,iA]E_{H_0}(\mcaJ)\geq\alpha E_{H_0}(\mcaJ),\qquad \mcaJ=(\lambda-\delta,\lambda+\delta).
\end{equation}
Under the assumption on $V$, all $V$, $[V,iA]$, and $E_{H}(\mcaJ)-E_{H_0}(\mcaJ)$ are compact on $\ell^2(\Z)$. Hence,  
\begin{align*}
&E_{H}(\mcaJ)[H,iA]E_{H}(\mcaJ)=E_{H_0}(\mcaJ)[H_0,iA]E_{H_0}(\mcaJ)+\\
&\underbrace{E_{H_0}(\mcaJ)[H_0,iA](E_{H}(\mcaJ)-E_{H_0}(\mcaJ))+(E_{H}(\mcaJ)-E_{H_0}(\mcaJ))[H_0,iA]E_{H}(\mcaJ)+E_{H}(\mcaJ)[V,iA]E_{H}(\mcaJ)}_{K_1}\\
&\geq \alpha E_{H_0}(\mcaJ)+K_1=\alpha E_{H}(\mcaJ)+\underbrace{\alpha( E_{H_0}(\mcaJ)-E_{H}(\mcaJ))+K_1}_{=:K\ \text{is compact}}.
\end{align*}
The desired \eqref{mourre estim} is proved. 
\vskip0.2cm
(2) Since $H$ is bounded and self-adjoint and $H\in C^{[\beta]}(A)$ by Lemma \ref{regularity of H}, combining these facts with {\rm (1)} and applying Theorem \ref{Resol-Smoo} to $H$ with $n=j+1$ yields the assertions in {\rm (2)} for all $0\leq j\leq[\beta]-2$.
\end{proof}
Therefore, the spectral assertions (i) in Theorem \ref{LAP-theorem} follow from the Mourre estimate \eqref{mourre estim} and the abstract Mourre theory. To establish the resolvent estimates and the regularity of the boundary values stated in Theorem \ref{LAP-theorem} (ii), it remains to transfer the above estimates in the spaces associated with $A$ to the weighted spaces $\ell^{2,s}(\Z)$. This is achieved by the following proposition.
\begin{proposition}\label{LAP1}
 Let $H=\Delta^2+V$ with $|V(n)|\lesssim \left<n\right>^{-\beta}$ for some $\beta>2$ and $\mcaI=(0,16)$. Given $\lambda\in \mcaI$, let $\mcaJ$ be as in \eqref{mourre estim}. For any relatively compact interval ${\rm I}\subseteq\mcaJ\setminus\sigma_{p}(H)$, define $\widetilde{\rm I}=\{z:\operatorname{Re} z\in {\rm I},\ 0<|\operatorname{Im} z|\leq1\}$. Then, for any $j\in\left\{0,\cdots,[\beta]-2\right\}$ and $s>j+\frac{1}{2}$, we have
\begin{itemize}
{
\item [{\rm (i)}] \begin{equation}\label{jth reso estim}
\sup\limits_{z\in\widetilde{\rm I}}\left\|R^{(j)}_V(z)\right\|_{\B(s,-s)}<\infty;
    \end{equation}
\item[{\rm (ii)}] $R^{(j)}_V(z)$ is uniformly continuous on $\widetilde{\rm I}$ in the norm topology of $\B(s,-s)$;
\vskip0.1cm
\item[{\rm (iii)}] for $\mu\in {\rm I}$, the norm limits
\begin{equation*}
     \frac{d^{j}}{d\mu^j}(R^{\pm}_V(\mu))=\lim\limits_{\varepsilon>0,\ \varepsilon\rightarrow0}R^{(j)}_V(\mu\pm i\varepsilon)
    \end{equation*}
     exist in $\B(s,-s)$ and are uniformly norm continuous on ${\rm I}$.
}
\end{itemize}
\end{proposition}

\begin{proof}[Proof of Proposition \ref{LAP1}]
Fix $\lambda\in\mcaI$ and a relatively compact interval ${\rm I}\subseteq\mcaJ\setminus\sigma_{p}(H)$. Let $j\in\left\{0,\cdots,[\beta]-2\right\}$.  Since
\begin{equation}\label{increase of weighted space}
    \|T\|_{\B(s_2,-s_2)}\leq \|T\|_{\B(s_1,-s_1)},\qquad s_1<s_2, 
\end{equation}
it suffices throughout the proof to consider $j+\frac{1}{2}<s\leq [\beta]$. Differentiating the following resolvent identity $j$ times with respect to $z$
 \begin{equation*}
 R_V(z)=R_V(-i)+(z+i)(R_V(-i))^2+(z+i)^2R_V(-i)R_V(z)R_V(-i)
 \end{equation*}
 yields
 \begin{equation*}
 R^{(j)}_V(z)=(R_V(-i))^{(j)}+(z+i)^{(j)}(R_V(-i))^2+\sum\limits_{j_1+j_2=j}C^{j_1}_j((z+i)^2)^{(j_1)}R_V(-i)R^{(j_2)}_V(z)R_V(-i).
 \end{equation*}
 \vskip0.1cm
 \underline{$\bm{({\rm i})}$}~To obtain \eqref{jth reso estim}, it suffices to prove the uniform estimate for any $0\leq j_2\leq j$:
 \begin{equation}\label{R(-i)R(z)R(-i)}
\sup\limits_{z\in\widetilde{\rm I}}\left\|R_V(-i)R^{(j_2)}_V(z)R_V(-i)\right\|_{\B(s,-s)}<\infty, \qquad j_2+\frac{1}{2}<s\leq [\beta].
 \end{equation} 
For this purpose, we write
 \begin{equation*}
 \left<\mcaN\right>^{-s}R_V(-i)R^{(j_2)}_V(z)R_V(-i)\left<\mcaN\right>^{-s}=\underbrace{\left<\mcaN\right>^{-s}R_V(-i)\left<A\right>^{s}}\underbrace{\left<A\right>^{-s}R^{(j_2)}_V(z)\left<A\right>^{-s}}\underbrace{\left<A\right>^{s}R_V(-i)\left<\mcaN\right>^{-s}}.
 \end{equation*}
By \eqref{jth esti of Lem} and the duality, it therefore remains to prove
 \begin{equation}\label{ARVN}
 \left<A\right>^{s}R_V(\pm i)\left<\mcaN\right>^{-s}\in\B(0,0),\qquad 0\leq s\leq[\beta].
 \end{equation}
It is obvious that \eqref{ARVN} holds for $s=0$. We next show that
 \begin{equation*}
\left<A\right>^{[\beta]}R_V(\pm i)\left<\mcaN\right>^{-[\beta]}\in\B(0,0),
 \end{equation*}
from which \eqref{ARVN} follows by complex interpolation (see
\cite[Lemma 8.2]{PSS81}). It further reduces to verify
 \begin{equation}\label{ell}
 A^{k}R_V(\pm i)\left<\mcaN\right>^{-[\beta]}\in\B(0,0),\qquad\forall\ 1\leq k\leq[\beta],\quad k\in \N^+.
 \end{equation}
 Indeed, repeatedly using 
 \begin{equation*}
AR_V(\pm i)=-ad^1_{A}(R_V(\pm i))+R_V(\pm i)A= R_V(\pm i)\big(ad^{1}_{A}(H)R_V(\pm i)+A\big)
 \end{equation*}
together with
$$Aad^q_A(H)=ad^q_A(H)A-ad^{q+1}_A(H),$$
we can move all powers of $A$ to the right and write
$$A^{k}R_V(\pm i)\left<\mcaN\right>^{-[\beta]}=\sum\limits_{\ell=0}^{k}T^{\pm}_{k,\ell}A^{\ell}\left<\mcaN\right>^{-[\beta]},\qquad 0\leq \ell\leq k,$$
where each $T^\pm_{k,\ell}$ is a finite product of operators of the form
$ad_A^q(H)$ and $R_V(\pm i)$ with $1\le q\leq k$. Since $H\in C^{[\beta]}(A)$ and
$k\leq[\beta]$, all these factors are bounded, and hence $T^\pm_{k,\ell}\in\B(0,0)$. Moreover, since $k\leq[\beta]$,
 we have $A^{\ell}\left<\mcaN\right>^{-[\beta]}\in\B(0,0)$ for any $0\leq\ell\leq k$. This establishes the \eqref{ell} and consequently proves the desired statement (i).
\vskip0.1cm
Combining with (i), the assertions \underline{$\bm{({\rm ii})}$ and $\bm{({\rm iii})}$} follow in the same way from Lemma \ref{LAP lemma} (2)(ii)--(iii) and the factorization above, which preserve the uniform continuity and boundary-value properties. This completes the proof. 
\end{proof}
Assume in addition that $H$ has no eigenvalues in $\mcaI$. Then, as a consequence of Theorem \ref{LAP-theorem}, for every $\mu\in(0,2)$, the boundary values $R^{\pm}_V(\mu^4)$ exist in $\B(s,-s)$ for every $s>\frac12$. This leads to the following consequence concerning the invertibility of $M^{\pm}(\mu)$, which will play a crucial role in the analysis of the asymptotic behavior of $R^{\pm}_V(\mu^4)$ near $\mu=0$ and $\mu=2$.
\begin{corollary}\label{corollary inverse of M}
{ Let $H,V$ be as in Theorem \ref{LAP-theorem}. Assume that $H$ has no eigenvalues in $\mcaI$. Then for any $\mu\in(0,2)$, $M^{\pm}(\mu)$ is invertible on $\ell^2(\Z)$ and satisfies the relation below in $\B(s,-s)$ for $s>\frac{1}{2}$:
\begin{equation}\label{inverti relation}
R^{\pm}_V(\mu^4)=R^{\pm}_{0}(\mu^4)-R^{\pm}_0(\mu^4)v\left(M^{\pm}(\mu)\right)^{-1}vR^{\pm}_0(\mu^4).
\end{equation}
}
\end{corollary}
\begin{proof}
First, for any $z\in {\rm(}\C\setminus[0,16]{\rm)}\cap\rho(H)$~(the resolvent set of $H$), define $$f(z)=U+vR_0(z)v,\qquad g(z)=U-UvR_V(z)vU.$$ 
Using $U^2=I$ and  the resolvent identity 
\begin{equation}\label{R0-RV}
R_0(z)-R_V(z)=R_0(z)VR_V(z)=R_V(z)VR_0(z),
\end{equation}
one readily verifies that
\begin{equation}\label{fg}
f(z)g(z)=I=g(z)f(z).
\end{equation}
Fix $\mu\in(0,2)$ and choose $\frac12<s_0\leq\frac{\beta}{2}$. By Theorem \ref{LAP-theorem},
\begin{equation*}
\| f(\mu^4\pm i\varepsilon)-M^{\pm}(\mu)\|_{\B(0,0)}\leq ||v(\cdot)\left<\cdot\right>^{s_0}||^2_{\ell^{\infty}}\|R_0(\mu^4\pm i\varepsilon)-R^{\pm}_0(\mu^4)\|_{\B(s_0,-s_0)}\rightarrow 0 ,\quad \varepsilon\rightarrow0,
\end{equation*}
and similarly, 
\begin{equation*}
\lim\limits_{\varepsilon>0,\ \varepsilon\rightarrow0}g(\mu^4\pm i\varepsilon)=U-UvR^{\pm}_V(\mu^4)vU\quad {\rm in}\quad \B(0,0).
\end{equation*}
Using \eqref{jth reso estim}  and passing to the limit in \eqref{fg}, we obtain that $M^{\pm}(\mu)$ is invertible on $\ell^2(\mathbb Z)$ and
 $$(M^{\pm}(\mu))^{-1}=U-UvR^{\pm}_V(\mu^4)vU.
 $$
Furthermore, using $V=vUv$ together with \eqref{R0-RV}, we obtain
\begin{equation}
R_{V}(z)=R_0(z)-R_0(z)vg(z)vR_0(z).
\end{equation}
Passing to the boundary values and using \eqref{jth reso estim}, we conclude that \eqref{inverti relation} holds in $\B(s,-s)$ for $1/2<s\leq\beta/2$. Combining this with \eqref{increase of weighted space}, we obtain the desired result for all $s>1/2$.
\end{proof}
By the same argument, applied to the resolvent identity
$$R_V(z)=R_0(z)-R_0(z)VR_0(z)+R_0(z)VR_V(z)VR_0(z),$$
we also obtain, in $\B(s,-s)$ for $\frac12<s\leq[\beta]$,
$$R^{\pm}_V(\mu^4)=R^{\pm}_{0}(\mu^4)-R^{\pm}_{0}(\mu^4)VR^{\pm}_{0}(\mu^4)+R^{\pm}_{0}(\mu^4)VR^{\pm}_V(\mu^4)VR^{\pm}_{0}(\mu^4).$$
This identity will play a key role in the estimates for the kernel $K_2^{\pm}(t,n,m)$.

\section{Sharp decay estimates for the free case}\label{Sec of decay for free}
In this section, we consider the free case $V\equiv0$ and establish the corresponding decay estimates in our main results, together with their sharpness.
\begin{theorem}\label{D-E for free case}
{ For $t\neq0$, the following sharp decay estimates hold:
\begin{equation}\label{eitH0 decay estimate}
\|e^{it\Delta}\|_{\ell^1(\Z )\rightarrow\ell^{\infty}(\Z )}\lesssim |t|^{-\frac{1}{3}} ,\qquad\big\|e^{-it\Delta^2}\big\|_{\ell^1(\Z )\rightarrow\ell^{\infty}(\Z )}\lesssim|t|^{-\frac{1}{4}}.
\end{equation}
}
\end{theorem}
It is known that the sharp decay estimate for $e^{it\Delta}$ can be derived directly from the Fourier transform. Its kernel is given by
  \begin{equation*}
\big(e^{it\Delta}\big)(n,m)=(2\pi)^{-1}\int_{-\pi}^{\pi}e^{-it(2-2\cos x)+i(n-m)x}dx.
 \end{equation*}
The same Fourier approach may also be used to establish the decay estimate for $e^{-it\Delta^2}$. However, for later use in the perturbed case, we derive the free estimate \eqref{eitH0 decay estimate} in Theorem \ref{D-E for free case} via Stone's formula. This approach naturally extends to the perturbed resolvent and will provide the framework for the subsequent analysis. To this end, we first establish the following oscillatory integral estimate.
\begin{lemma}\label{Von der}
{ For $t\neq0$, the following estimate holds:}
\begin{equation}\label{von der esti}
\sup\limits_{s\in\R}\left|\int_{-\pi}^{0}e^{-it\left[(2\pm2\cos x)^2-sx\right]}dx\right|\lesssim |t|^{-\frac{1}{4}}.
\end{equation}
\end{lemma}
\begin{proof}
This estimate is a direct application of the Van der Corput lemma (see, e.g., \cite[pp.~332--334]{Ste93}). By the change of variables $x=-\pi-y$, we obtain
$$\int_{-\pi}^{0}e^{-it[(2+2\cos x)^2-sx]}dx=e^{-its\pi}\int_{-\pi}^{0}e^{-it[(2-2\cos y)^2+sy]}dy.$$
Hence it suffices to consider the ``$-$" case. For any $s\in\R$, let $$\Phi_{s}(x):=(2-2\cos x)^2-sx,\quad x\in[-\pi,0].$$
First, a direct calculation gives
$$\Phi'_{s}(x)=8(1-\cos x)\sin x-s,\qquad \Phi''_{s}(x)=8(1-\cos x)(2\cos x+1).$$
Since 
$$\Phi''_{s}(x)=0,\quad x\in[-\pi,0]\Longleftrightarrow x=-\frac{2\pi}{3}\ \ {\rm or}\ \  x=0,$$
the function $\Phi'_s$ is decreasing on
$\left[-\pi,-\frac{2\pi}{3}\right]$ and increasing on
$\left[-\frac{2\pi}{3},0\right]$. Moreover, 
$$\Phi'_s(-\pi)=\Phi'_s(0)=-s, \qquad \Phi'_s\Big(-\frac{2\pi}{3}\Big)=-6\sqrt3-s.$$
Consequently, for any $s\in\R$, the equation $\Phi'_{s}(x)=0$ has at most two solutions on $[-\pi,0]$. By the Van der Corput Lemma, slower decay rates of the oscillatory integral \eqref{von der esti} for the ``$-$" case occur in the cases of $s=0$ and $s=-6\sqrt 3$. For the other values of $s$, the rate is either $|t|^{-1}$ or $|t|^{-\frac{1}{2}}$.

For $s=0$, we have $$\Phi'_{0}(x)=0,\quad x\in[-\pi,0]\Longleftrightarrow x=0\ \ {\rm or}\ \ x=-\pi.$$
At $x=-\pi$, $\Phi''_{0}(-\pi)=-16\neq0$, whereas at $x=0$,
$$\Phi''_{0}(0)=\Phi^{(3)}_{0}(0)=0,\qquad\Phi^{(4)}_{0}(0)=24\neq0.$$
By the
Van der Corput lemma on suitable subintervals of $[-\pi,0]$, we obtain
the decay rate $|t|^{-1/4}$.

For $s=-6\sqrt 3$, the equation $\Phi'_{-6\sqrt3}(x)=0$ has the unique
solution
$ x=-\frac{2\pi}{3}.$ At this point,
$$\Phi''_{-6\sqrt 3}\Big(-\frac{2\pi}{3}\Big)=0,\qquad\Phi^{(3)}_{-6\sqrt 3}\Big(-\frac{2\pi}{3}\Big)\neq0.$$
This implies that the decay rate is $|t|^{-\frac{1}{3}}$. In summary, we derive the desired estimate of $|t|^{-\frac{1}{4}}$.
\end{proof}
From the proof of Lemma \ref{Von der}, we immediately obtain the following consequence, which will be used in the estimates for the kernels $K_j^\pm(t,n,m)$ defined in \eqref{kernels of Ki}.
\begin{corollary}\label{corollary} 
The following statements hold for every interval $[a,b]\subseteq[-\pi,0]$.

 {\rm(i)}~The estimate \eqref{von der esti} still holds on $[a,b]$.

 {\rm(ii)}~Suppose that $\phi(x)$ is a continuously differentiable function on $(a,b)$ with $\phi'(x)\in L^{1}((a,b))$. Moreover, $\lim\limits_{x\rightarrow a^+}\phi(x)$ and $\lim\limits_{x\rightarrow b^-}\phi(x)$ exist. Then
\begin{equation}\label{von der esti complex}
\sup\limits_{s\in\R}\left|\int_{a}^{b}e^{-it\left[(2\pm2\cos x)^2-sx\right]}\phi(x)dx\right|\lesssim |t|^{-\frac{1}{4}}\Big(\big|\lim\limits_{x\rightarrow b^-}\phi(x)\big|+\int_{a}^{b}|\phi'(x)|dx\Big).
\end{equation}
\end{corollary}
We now return to the proof of Theorem \ref{D-E for free case}.
\begin{proof}[Proof of Theorem \ref{D-E for free case}]
First, using Stone's formula, the kernel of $e^{-it\Delta^2}$ is given by
\begin{equation*}
\big(e^{-it\Delta^2}\big)(n,m)=-\frac{2}{\pi i}(K^{+}_{0}-K^{-}_{0})(t,n,m),\qquad K^{\pm}_{0}(t,n,m)=\int_{0}^{2}e^{-it\mu^4}\mu^3R^{\mp}_0(\mu^4,n,m)d\mu.
\end{equation*}
 Although it would suffice to estimate $K_0^+-K_0^-$, doing so essentially recovers the Fourier representation of the free evolution. We therefore estimate $K_0^\pm$ separately. This is slightly more involved, but better suited to the perturbed analysis. To this end, by \eqref{kernel of R0 boundary}, we obtain
\begin{align}\label{decompose of K0+}
\begin{split}
4K^{\pm}_0(t,n,m)&=\mp i\int_{0}^{2}e^{-it\mu^4} e^{\pm i\theta(\mu)|n-m|}a_1(\mu)d\mu+\int_{0}^{2}e^{-it\mu^4} e^{b(\mu)|n-m|}a_2(\mu)d\mu\\
&=:(K^{\pm}_{0,1}+K_{0,2})(t,n,m).
\end{split}
\end{align}
We first estimate $K_{0,1}^\pm(t,n,m)$. Using the change of variables
$\theta=\theta(\mu)$ determined by
\begin{equation}\label{varible substi1}
 \cos\theta=1-\frac{\mu^2}{2} \Longrightarrow\  \frac{d\mu}{d\theta}=-(a_1(\mu))^{-1},\qquad \theta\rightarrow0\ {\rm as}\ \mu\rightarrow0\ {\rm and}\ \theta\rightarrow-\pi\ {\rm as}\ \mu\rightarrow2,
\end{equation}
we have
 \begin{equation*}
 K^{\pm}_{0,1}(t,n,m)=\mp i\int_{-\pi}^{0}e^{-it\big[(2-2\cos\theta)^2\mp\theta\frac{|n-m|}{t}\big]}d\theta,\qquad t\neq0.
\end{equation*}
Hence, by Lemma \ref{Von der},
\begin{equation}\label{esti for K0,1}
\sup\limits_{n,m\in\Z}|K^{\pm}_{0,1}(t,n,m)|\leq\sup\limits_{s\in\R}\left|\int_{-\pi}^{0}e^{-it\left[(2-2\cos\theta)^2-s\theta\right]}d\theta\right|\lesssim |t|^{-\frac{1}{4}},\qquad t\neq0.
\end{equation}
We next estimate $K_{0,2}(t,n,m)$. Using the Van der Corput Lemma directly, we obtain
\begin{equation*}
|K_{0,2}(t,n,m)|\lesssim|t|^{-\frac{1}{4}}\Big(\Big|\frac{1}{\sqrt{2}}(3-2\sqrt{2})^{|n-m|}\Big|+\big\|\partial_{\mu}\big(e^{b(\mu)|n-m|}a_2(\mu)\big)\big\|_{L^{1}((0,2))}\Big).
\end{equation*}
Since $b(\mu)<0$, $b'(\mu)=a_2(\mu)<0$, and $a_2(\mu),a_2'(\mu)$ are uniformly bounded on $(0,2)$, we have
\begin{equation}\label{integral ebnm}
\big\|\partial_{\mu}\big(a_2(\mu)e^{b(\mu)|n-m|}\big)\big\|_{L^{1}((0,2))}\lesssim \int_{0}^{2}-b'(\mu)|n-m|e^{b(\mu)|n-m|}d\mu+\int_{0}^{2}e^{b(\mu)|n-m|}d\mu\leq4,
\end{equation}
uniformly in $n,m\in\Z$. Therefore, $K_{0,2}(t,n,m)$ satisfies the bound $|t|^{-1/4}$ uniformly in $n,m\in\Z$. Combining this with \eqref{esti for K0,1} and \eqref{decompose of K0+}, we obtain \eqref{eitH0 decay estimate}.
\end{proof}
Finally, we show that the exponent $\frac14$ in \eqref{eitH0 decay estimate} is sharp. More precisely, $\frac14$ is the supremum of all $\alpha>0$ for which there exists a constant $C_\alpha>0$ such that
\begin{equation}\label{sharness meaning}
\big\|e^{-it\Delta^2}\phi\big\|_{\ell^\infty}
\leq
C_\alpha |t|^{-\alpha}\|\phi\|_{\ell^1},
\qquad t\neq0,
\end{equation}
for every $\phi\in\ell^1(\Z)$.
\begin{proposition}\label{theorem of strichartz estimate}
Consider the inhomogeneous free discrete fourth-order Schr\"odinger equation on $\Z$,
$$i(\partial_tu)(t,n)-(\Delta^2u)(t,n)+F(t,n)=0,$$
with the initial value $\{u(0,n)\}_{n\in\Z}\in\ell^2(\Z)$. Then the following statements hold.

{\rm (i)} For any pairs $(q,r)$ and $(\widetilde q,\widetilde r)$ satisfying
$$q,r,\widetilde q,\widetilde r\geq2, \qquad \frac1q+\frac1{4r}\leq\frac18, \qquad \frac1{\widetilde q}+\frac1{4\widetilde r}\leq\frac18,$$
we have the Strichartz estimate
\begin{equation}\label{strichartz estimate}
\|\{u(t,n)\}\|_{L^{q}_t\ell^{r}}\leq C\big(\|\{u(0,n)\}\|_{\ell^2}+\|\{F(t,n)\}\|_{L^{\tilde{q}^{'}}_t\ell^{\tilde{r}^{'}}}\big),
\end{equation}
where $\widetilde q'$ and $\widetilde r'$ denote the dual exponents of $\widetilde q$ and $\widetilde r$, respectively. Here
$$\|\{u(t,n)\}\|_{L^{q}_t\ell^{r}}=\Big(\int_{\R}^{}\Big(\sum\limits_{n\in\Z}|u(t,n)|^{r}\Big)^{\frac{q}{r}}dt\Big)^{\frac{1}{q}}.$$

{\rm (ii)} The exponent $\frac14$ in the decay estimate for $e^{-it\Delta^2}$ in \eqref{eitH0 decay estimate} is sharp.
\end{proposition}
\begin{proof}
\underline{\textbf{(i)}}~From the energy identity $\|e^{-it\Delta^2}\varphi\|_{\ell^2}=\|\varphi\|_{\ell^2}$ and the decay estimate \eqref{eitH0 decay estimate}, the Strichartz estimates \eqref{strichartz estimate} follow directly from \cite[Theorem 1.2]{KT98} by Keel and Tao.

\underline{\textbf{(ii)}} We first derive a necessary condition for the homogeneous Strichartz estimate by a Knapp-type argument. Let $2\leq q,r<\infty$. By duality,
\begin{equation}\label{equivalent relation}
\|\{e^{-it\Delta^2}u(0,n)\}\|_{L^{q}_{t}\ell^{r}}\leq C\|u(0,\cdot)\|_{\ell^2}\Longleftrightarrow\left\|\int_{\R}^{}e^{it\Delta^2}F(t,\cdot)dt\right\|_{\ell^2}\leq C\|\{F(t,n)\}\|_{L^{q'}_{t}\ell^{r'}}.
\end{equation}
By the Fourier transform \eqref{fourier transform} on $\Z$, 
$$\mcaF\Big(\int_{\R}^{}e^{it\Delta^2}F(t,\cdot)dt\Big)(x)=(2\pi)^{\frac{1}{2}}\hat{f}_{time}(-(2-2{\rm cos}x)^2,x),$$
where $$\hat{f}_{time}(s,x)=(2\pi)^{-\frac{1}{2}}\int_{\R}^{}f(t,x)e^{-ist}dt,\qquad f(t,x)=(\mcaF F(t,\cdot))(x).$$
Thus, by Plancherel's theorem, \eqref{equivalent relation} implies
\begin{equation}\label{equivent}
\Big(\int_{-\pi}^{\pi}\big|\hat{f}_{time}(-(2-2{\rm cos}x)^2,x)\big|^2dx\Big)^{\frac{1}{2}}\leq C'\|\{F(t,n)\}\|_{L^{q'}_{t}\ell^{r'}}.
\end{equation}
For any $0<\varepsilon\ll1$, choose
$$\hat{f}_{time}(s,x)=\chi(\varepsilon^{-4}s)\chi(\varepsilon^{-1}x),$$
where $\chi$ is the characteristic function of $(-1,1)$. Then
$$F(t,n)=C''\frac{{\rm sin}\left(\varepsilon^{4}t\right)}{t}\frac{{\rm sin}(\varepsilon n)}{n}.$$
Since $(2-2{\rm cos}x)^2=O(x^4)$ as $  x\rightarrow0$, the left-hand side of \eqref{equivent} satisfies
$$\Big(\int_{-\pi}^{\pi}\big|\hat{f}_{time}(-(2-2{\rm cos}x)^2,x)\big|^2dx\Big)^{\frac{1}{2}}\gtrsim\varepsilon^{\frac{1}{2}}.$$
On the other hand,
\begin{align*}
\sum\limits_{n\in\Z}^{}\frac{|{\rm sin}(\varepsilon n)|^{r'}}{|n|^{r'}}&=\big(\sum\limits_{|n|\leq\frac{1}{\varepsilon}}^{}+\sum\limits_{|n|>\frac{1}{\varepsilon}}^{}\big)\frac{\sin\varepsilon n)|^{r'}}{|n|^{r'}}\leq C'''\varepsilon^{r'-1}+\sum\limits_{|n|>\frac{1}{\varepsilon}}^{}\frac{1}{|n|^{r'}}\lesssim \varepsilon^{r'-1}.
\end{align*}
Therefore, $$\|\{F(t,n)\}\|_{L^{q'}_{t}\ell^{r'}}\lesssim\varepsilon^{\frac{1}{r}+\frac{4}{q}}.$$
Substituting these estimates into \eqref{equivent} yields the necessary condition $\frac{1}{2}\geq\frac{1}{r}+\frac{4}{q}$.

We now prove the sharpness of the exponent $\frac14$ in
\eqref{eitH0 decay estimate}. Suppose, to the contrary, that for some
$\alpha>\frac14$. By \cite[Theorem 1.2]{KT98}, the corresponding Strichartz estimates
would then hold whenever $\frac{1}{q}+\frac{\alpha}{r}\leq\frac{\alpha}{2}$. Since $\alpha>\frac{1}{4}$, then there exists $q,r\geq2$ satisfying
\begin{equation*}
\frac{1}{q}+\frac{\alpha}{r}\leq\frac{\alpha}{2}\ \ {\rm and}\ \ \frac{1}{q}+\frac{1}{4r}>\frac{1}{8}.
\end{equation*}
This contradicts the necessary condition established above. Therefore, the exponent $\frac14$ is sharp.
\end{proof}
\section{Proof of main Theorem \ref{main-theorem}}\label{sec of proof}
This section is devoted to proving the decay estimate \eqref{eitH decay-estimate} for $e^{-itH}P_{ac}(H)$. The proof of \eqref{cos-sin decay-estimate} is analogous.
To begin with, we recall the decomposition
\begin{equation}\label{kernel of eitHPacH(4 section)1}
(e^{-itH}P_{ac}(H))(n,m)=-\frac{2}{\pi i}\sum\limits_{j=0}^{3}(K^{+}_{j}-K^{-}_{j})(t,n,m),
\end{equation}
where $K^{\pm}_{j}(t,n,m)(j=0,1,2,3)$ are defined in \eqref{kernels of Ki}.
As shown in Section \ref{Sec of decay for free}, the required estimate for $K^{\pm}_0(t,n,m)$ has already been established. It therefore remains to estimate $(K_j^+-K_j^-)(t,n,m)$ for $j=1,2,3$.
\begin{theorem}\label{theorem of estimate of K1,K3}
 Let $H=\Delta^2+V$ with $|V(n)|\lesssim \left<n\right>^{-\beta}$ for some $\beta>0$ and $K^{\pm}_j(t,n,m)(j=1,2,3)$ be defined as in \eqref{kernels of Ki}. Then the following statements hold.
\begin{itemize}
\item[{\rm(i)}] If
\begin{align}\label{cases and conditions 0}
\beta>\left\{\begin{aligned}&15,\ 0\ is\ a\ regular\ point\ of\ H,\\
&19,\ 0\ is\ a\ first\ kind\ resonance\ of\ H,\\
&27,\ 0\ is\ a\ second\ kind\ resonance\ of\ H,\end{aligned}\right.
\end{align} then
\begin{equation}\label{esti for Kpm1}
\left|K^{\pm}_{1}(t,n,m)\right|\lesssim |t|^{-\frac{1}{4}}, \quad t\neq0,\  \ uniformly\ in\ n,m\in\Z.
\end{equation}
\item[{\rm(ii)}] If $\beta>2$, then
\begin{equation}\label{esti for Kpm2}
\left|K^{\pm}_{2}(t,n,m)\right|\lesssim |t|^{-\frac{1}{4}}, \quad t\neq0,\  \ uniformly\ in\ n,m\in\Z.
\end{equation}
\item[{\rm(iii)}] If
\begin{align}\label{cases and conditions 16}
\beta>\left\{\begin{aligned}&7,\ \  16\ is\ a\ regular\ point\ of\ H,\\
&11,\ 16\ is\ a\ resonance\ of\ H,\\
&15,\ 16\ is\ an\ eigenvalue\ of\ H,\end{aligned}\right.
\end{align}
then
\begin{equation}\label{esti for Kpm3}
\left|(K^{+}_{3}-K^{-}_3)(t,n,m)\right|\lesssim |t|^{-\frac{1}{4}}, \quad t\neq0,\  \ uniformly\ in\ n,m\in\Z.
\end{equation}
\end{itemize}
\end{theorem}
Once Theorem \ref{theorem of estimate of K1,K3} is established, Theorem \ref{main-theorem} follows immediately from it together with Theorem \ref{D-E for free case}. To prove Theorem \ref{theorem of estimate of K1,K3}, we divide the argument into three subsections.
\subsection{The estimates of kernels $K^{\pm}_1(t,n,m)$}\label{subse of K1}
In this subsection, we establish the estimates \eqref{esti for Kpm1} for $K^{\pm}_1(t,n,m)$ under the assumptions in \eqref{cases and conditions 0}. We begin with the following lemma, which plays a key role in handling the singular behavior of $R^{\pm}_0(\mu^4)$ near $\mu=0$.
\begin{lemma}\label{cancelation lemma}
Let $Q$ and $S_j$ $(j=0,1,2)$ be the operators defined in Definition \ref{definition of Sj}.
\vskip0.2cm
\noindent{\rm(1)} Let $b_1(\mu)=-\theta(\mu)a_1(\mu)$ and $ b_2(\mu)=-b(\mu)a_2(\mu)$. Then
$$(R^{\pm}_0(\mu^4)vQf)(n)=\frac{1}{4\mu^3}\sum\limits_{m\in\Z}\mcaB^{\pm}(\mu,n,m)(Qf)(m),\quad Q(vR^{\pm}_0(\mu^4)f\big)=\frac{1}{4\mu^3}Q\Big(\sum\limits_{m\in\Z}\mcaB^{\pm}(\mu,m,\cdot)f(m)\Big),$$
where 
\begin{equation*}
    \mcaB^{\pm}(\mu,n,m)=v_1(m)\int_{0}^{1}({\rm sign}(n-\rho m))\big({\bm {b_1(\mu)}}e^{\mp i\theta(\mu)|n-\rho m|}+{\bm {b_2(\mu)}}e^{b(\mu)|n-\rho m|}\big)d\rho.
\end{equation*}
    
\noindent{\rm(2)} Let $c_1(\mu)=- i(\theta(\mu))^2a_1(\mu), \ c_2(\mu)=(b(\mu))^2a_2(\mu) $ and $c_3(\mu)=\theta(\mu)a_1(\mu)+b(\mu)a_2(\mu)$. Then for $j=0,1$,
$$(R^{\pm}_0(\mu^4)vS_jf)(n)=\frac{1}{4\mu^3}\sum\limits_{m\in\Z}\mcaC^{\pm}(\mu,n,m)(S_jf)(m),\quad S_j\big(vR^{\pm}_0(\mu^4)f\big)=\frac{1}{4\mu^3}S_j\Big(\sum\limits_{m\in\Z}\mcaC^{\pm}(\mu,m,\cdot)f(m)\Big),$$
where 
\begin{equation*}
    \mcaC^{\pm}(\mu,n,m)=v_2(m)\int_{0}^{1}(1-\rho)\big(\pm{\bm{c_1(\mu)}}e^{\mp i\theta(\mu)|n-\rho m|}+{\bm{c_2(\mu)}}e^{b(\mu)|n-\rho m|}\big)d\rho +{\bm{c_3(\mu)}}v(m)|n-m|.
\end{equation*}
    
\noindent{\rm(3)} Let $d_1(\mu)=(\theta(\mu))^3a_1(\mu),\  d_2(\mu)=-(b(\mu))^3a_2(\mu)$ and $d_3(\mu)=2c_3(\mu)$. Then
$$(R^{\pm}_0(\mu^4)vS_2f)(n)=\frac{1}{8\mu^3}\sum\limits_{m\in\Z}\mcaD^{\pm}(\mu,n,m)(S_2f)(m),\ S_2\big(vR^{\pm}_0(\mu^4)f\big)=\frac{1}{8\mu^3}S_2\Big(\sum\limits_{m\in\Z}\mcaD^{\pm}(\mu,m,\cdot)f(m)\Big),$$
where 
\begin{align*}
    \mcaD^{\pm}(\mu,n,m)&=v_3(m)\int_{0}^{1}(1-\rho)^2{\rm sign}(n-\rho m)\big({\bm{d_1(\mu)}}e^{\mp i\theta(\mu)|n-\rho m|}+{\bm{d_2(\mu)}}e^{b(\mu)|n-\rho m|}\big)d\rho\\
    &\quad+{\bm{d_3(\mu)}}v(m)|n-m|.
\end{align*}
\end{lemma}
\begin{remark}\label{remark of cancelation lemma}
{\rm(1) We note that $\theta(\mu)$, $b(\mu)$ and $c_3(\mu)$ exhibit the following behaviors, respectively:
   \begin{equation}
    \theta(\mu)=-\mu-\frac{\mu^3}{24}+O(\mu^5),\quad b(\mu)=-\mu+\frac{\mu^3}{24}+O(\mu^5),\quad c_3(\mu)=-\frac{1}{3}\mu^3+O(\mu^7),\quad \mu\rightarrow0^+.  
   \end{equation}
This shows that, compared with the free resolvent $R_0^\pm(\mu^4)=O(\mu^{-3})$, the operators in this lemma have weaker singularities near $\mu=0$. Here and below, unless otherwise specified, $O(\mu^\alpha)$ refers to the order in $\mu$ of the corresponding integral kernel. More precisely, we have
\begin{align}\label{order of operators in cancel lemma 0}
\begin{split}
R^{\pm}_0(\mu^4)vQ&=O(\mu^{-2}),\quad R^{\pm}_0(\mu^4)vS_j=O(\mu^{-1})~(j=0,1),\quad  R^{\pm}_0(\mu^4)vS_2=O(1);\\
QvR^{\pm}_0(\mu^4)&=O(\mu^{-2}),\quad S_jvR^{\pm}_0(\mu^4)=O(\mu^{-1})~(j=0,1),\quad  S_2vR^{\pm}_0(\mu^4)=O(1).
\end{split}
\end{align}
\vskip0.1cm
\noindent(2)~However, the discrete structure is somewhat more involved than its continuous counterpart \cite{SWY22}. Indeed, comparing \eqref{kernel of R0 boundary} with the one-dimensional continuous kernel
$$
\mcaR^{\pm}_0(\mu^4,x-y)
=\frac{1}{4\mu^3}\big(\pm i e^{\pm i\mu|x-y|}-e^{-\mu|x-y|}\big),
\qquad x,y\in\R,
$$
we see that the continuous analogues of $(\theta(\mu),b(\mu),a_1(\mu),a_2(\mu))$ are $(-\mu,-\mu,1,-1)$. Hence, the corresponding ${b}_j(\mu)$, ${c}_j(\mu)$, and ${d}_j(\mu)$ are polynomials in $\mu$. In particular, $c_3(\mu)\equiv0$. This difference leads to some additional technical difficulties in establishing the decay estimate for $K_1^\pm(t,n,m)$.
}
\end{remark}
The proof of this lemma is deferred to the end of this subsection. We now prove \eqref{esti for Kpm1} case by case.
\subsubsection{{\bf {Regular case}}}
We first consider the case where $0$ is a regular point of $H$. Recall that
\begin{equation}\label{kernel of Kpm1}
K^{\pm}_{1}(t,n,m)=\int_{0}^{\mu_0}e^{-it\mu^4}\mu^3\big[R^{\pm}_0(\mu^4)v\big(M^{\pm}(\mu)\big)^{-1}vR^{\pm}_0(\mu^4)\big](n,m)d\mu.
\end{equation}
Using the expansion \eqref{asy expan of regular 0},
\begin{equation*}
\big(M^{\pm}(\mu)\big)^{-1}=S_0A_{0}S_0+\mu QA^{\pm,0}_{11}Q+\mu^2(QA^{\pm,0}_{21}Q+S_0A^{\pm,0}_{22}+A^{\pm,0}_{23}S_0)+\mu^3A^{\pm,0}_3+\Gamma^{0}_{4}(\mu),
\end{equation*}
we can further decompose $K^{\pm}_{1}(t,n,m)$ as follows:
\begin{align}\label{new expr of K1 regular 0}
K^{\pm}_1(t,n,m)=\Big(\sum\limits_{A\in\mcaA_0}K^{\pm}_A+K^{\pm}_r\Big)(t,n,m),
\end{align}
where $\mcaA_0=\{S_0A_{0}S_0,\mu QA^{\pm,0}_{11}Q,\mu^2QA^{\pm,0}_{21}Q,\mu^2S_0A^{\pm,0}_{22},\mu^2A^{\pm,0}_{23}S_0,\mu^3A^{\pm,0}_3\}$ and
\begin{align}\label{set A0 and kernel in regular case}
\begin{split}
K^{\pm}_A(t,n,m)&=\int_{0}^{\mu_0}e^{-it\mu^4}\mu^3\big[R^{\pm}_0(\mu^4)vAvR^{\pm}_0(\mu^4)\big](n,m)d\mu,\qquad A\in\mcaA_0,\\
K^{\pm}_r(t,n,m)&=\int_{0}^{\mu_0}e^{-it\mu^4}\mu^3\big[R^{\pm}_0(\mu^4)v\Gamma^0_{4}(\mu)vR^{\pm}_0(\mu^4)\big](n,m)d\mu.
\end{split}
\end{align}
Thus, to prove \eqref{esti for Kpm1}, it suffices to establish the corresponding bounds for $\{K^{\pm}_{A}:A\in\mcaA_0\}$ and for the remainder term $K^{\pm}_{r}$. These estimates are proved in the following two propositions.
\begin{proposition}\label{propo of K 01}
Let $H=\Delta^2+V$ with $|V(n)|\lesssim \left<n\right>^{-\beta}$ for $\beta>15$. Suppose that $0$ is a regular point of $H$. Then for any $A\in\mcaA_0$,
\begin{equation}\label{estimate of KB regular 0}
\sup\limits_{n,m\in\Z}|K^{\pm}_A(t,n,m)|\lesssim |t|^{-\frac{1}{4}},\qquad t\neq0.
\end{equation}
\end{proposition}
\begin{proof}
{\underline{\bf(1)}}~We first consider $A=S_0A_0S_0$. By Lemma \ref{cancelation lemma},
$$K^{\pm}_{01}(t,n,m):=\int_{0}^{\mu_0}e^{-it\mu^4}\mu^3\big[R^{\pm}_0(\mu^4)vS_0A_{0}S_0vR^{\pm}_0(\mu^4)\big](n,m)d\mu=\frac{1}{16}\sum\limits_{j=1}^{4}K^{\pm,j}_{01}(t,n,m),$$
where $N_1=n-\rho_1m_1$, $\widetilde{N}_1=n-m_1$,
$M_2=m-\rho_2m_2$, $\widetilde{M}_2=m-m_2$, and 
\begin{align*}
K^{\pm,1}_{01}(t,n,m)&=\sum\limits_{m_1,m_2\in\Z}\int_{[0,1]^2}\Omega^{\pm}_1(t,N_1,M_2)\prod\limits_{j=1}^{2}(1-\rho_j)d\rho_1d\rho_2(v_2S_0A_0S_0v_2)(m_1,m_2),\\
K^{\pm,2}_{01}(t,n,m)&=\sum\limits_{m_1,m_2\in\Z}\int_{0}^{1}\Omega^{\pm}_2(t,N_1)(1-\rho_1)d\rho_1|\widetilde{M}_2|(v_2S_0A_0S_0v)(m_1,m_2),\\
K^{\pm,3}_{01}(t,n,m)&=\sum\limits_{m_1,m_2\in\Z}\int_{0}^{1}\Omega^{\pm}_2(t,M_2)(1-\rho_2)d\rho_2|\widetilde{N}_1|(vS_0A_0S_0v_2)(m_1,m_2),\\
K^{\pm,4}_{01}(t,n,m)&=\Omega_3(t)\sum\limits_{m_1,m_2\in\Z}|\widetilde{N}_1|\cdot|\widetilde{M}_2|(vS_0A_0S_0v)(m_1,m_2).
\end{align*}
Denote $F^{\pm}(\mu,x)=c_2(\mu)e^{b(\mu)|x|}\pm c_1(\mu)e^{\mp i\theta(\mu)|x|}$. Here
\begin{align}\label{expre of Omega}
\begin{split}
    &\Omega^{\pm}_{1}(t,N_1,M_2)=\int_{0}^{\mu_0}e^{-it\mu^4}\mu^{-3}F^{\pm}(\mu,N_1)F^{\pm}(\mu,M_2)d\mu,\\
\Omega^{\pm}_2(t,N_1)&=\int_{0}^{\mu_0}e^{-it\mu^4}\mu^{-3}c_3(\mu)F^{\pm}(\mu,N_1)d\mu,\qquad
\Omega_3(t)=\int_{0}^{\mu_0}e^{-it\mu^4}\mu^{-3}(c_3(\mu))^2d\mu.
\end{split}
\end{align}
We claim that the desired decay estimates for $K^{\pm}_{01}(t,n,m)$ reduce to establishing 
\begin{equation}\label{estim for Omega123}
\sup\limits_{N_1,M_2\in\R}
\big|\Omega_1^\pm(t,N_1,M_2)\big|
+
\sup_{N_1\in\R}
\big|\Omega_2^\pm(t,N_1)\big|
+
|\Omega_3(t)|
\lesssim |t|^{-\frac14},
\qquad t\neq0.
\end{equation}
Indeed, since $\left<S_0f,v\right>=0$, $S_0v=0$ and $(S_0A_0S_0)(m_1,m_2)=\overline{(S_0A^*_0S_0)(m_2,m_1)}$, we can rewrite 
\begin{align*}
K^{\pm,2}_{01}(t,n,m)&=\big<g_m, S_0A^*_0S_0\overline{f^{\pm}_{t,n}}\big>,\qquad 
K^{\pm,3}_{01}(t,n,m)=\big<S_0A_0S_0{f^{\pm}_{t,m}},g_n\big>,\\
K^{\pm,4}_{01}(t,n,m)&=\Omega_3(t)\big<S_0A_0S_0g_m,g_n\big>,
\end{align*}
where
\begin{equation*}
g_y(x)=(|y-x|-|y|)v(x),\qquad    f^{\pm}_{t,y}(x)=v_2(x)\int_{0}^{1}(1-\rho)\Omega^{\pm}_2(t,y-\rho x)d\rho.
\end{equation*}
Once \eqref{estim for Omega123} is established, the triangle inequality, the Cauchy--Schwarz inequality, and the absolute boundedness of $S_0A_0S_0$ yield
\begin{equation*}
\sup\limits_{n,m\in\Z}|K^{\pm}_{01}(t,n,m)|\lesssim\sum\limits_{j=1}^{4}\sup\limits_{n,m\in\Z}|K^{\pm,j}_{01}(t,n,m)|\lesssim |t|^{-1/4}\left<|S_0A_0S_0|(|v_2|),|v_2|\right>\lesssim |t|^{-1/4},\qquad t\neq0.
\end{equation*}
In view of \eqref{expre of Omega}, to establish \eqref{estim for Omega123}, it suffices to estimate oscillatory integrals of the form
\begin{equation*} 
\mcaI_{1}(t,s,N)=\int_{0}^{\mu_0}e^{-it\mu^4}e^{i\theta(\mu)s}c_{k\ell}(\mu)e^{b(\mu)N}d\mu,\quad c_{k\ell}(\mu)=\mu^{-3}c_k(\mu)c_{\ell}(\mu),\quad k, \ell\in\{1,2,3\},
\end{equation*}
where $s\in\R$ and $N\in[0,\infty)$. Using the change of variables
$\theta=\theta(\mu)$:
\begin{equation}\label{varible substi local}
 \cos\theta=1-\frac{\mu^2}{2} \Longrightarrow\  \frac{d\mu}{d\theta}=-(a_1(\mu))^{-1},
\end{equation}
we obtain 
\begin{equation*}
\mcaI_1(t,s,N)=\int_{\theta(\mu_0)}^{0}e^{-it((2-2\cos \theta)^2-\frac{ s}{t}\theta)} C_{k,\ell}(\mu(\theta))e^{b(\mu(\theta))N}d\theta,\qquad C_{k\ell}(\mu)=(a_1(\mu))^{-1}c_{k\ell}(\mu).
\end{equation*}
Since
$$c_1(\mu)=- i(\theta(\mu))^2a_1(\mu), \qquad c_2(\mu)=(b(\mu))^2a_2(\mu),\qquad c_3(\mu)=-\frac{1}{3}\mu^3+O(\mu^7),\quad \mu\rightarrow0^+,$$
and 
\begin{equation*}
    \theta(\mu)=-\mu-\frac{\mu^3}{24}+O(\mu^5),\qquad b(\mu)=-\mu+\frac{\mu^3}{24}+O(\mu^5),\qquad\mu\rightarrow0^+,
\end{equation*}
the smoothness of the functions involved implies that
$$\frac{d^jC_{k\ell}(\mu(\theta))}{d\theta^j}\in L^{\infty}([\theta(\mu_0),0)),\qquad j=0,1.$$
Moreover, since $b(\mu)<0$, $b'(\mu)<0$, and $\mu'(\theta)<0$, we have, uniformly for $N\geq0$,
\begin{equation}\label{estim for ebutheta}
\|e^{b(\mu(\theta))N}\|_{L^1([\theta(\mu_0),0))}+\int_{\theta(\mu_0)}^{0}b'(\mu(\theta))\mu'(\theta)Ne^{b(\mu(\theta))N}d\theta\lesssim 1+\big|e^{b(\mu)N}|^{\mu=\mu_0}_{\mu=0}\big|\lesssim1.
\end{equation}
Combining these bounds with Corollary \ref{corollary}, we obtain
\begin{equation}\label{estimate of mcaI1}
\sup\limits_{s,N}|\mcaI_1(t,s,N)|\lesssim |t|^{-1/4},\qquad t\neq0.
\end{equation}
Hence, \eqref{estim for Omega123} is established.
\vskip0.2cm
{\underline{\bf(2)}} For the remaining terms
$A\in\mcaA_0\setminus\{S_0A_0S_0\}$, the same argument as above reduces the desired estimates to oscillatory integrals of the form
\begin{equation*}
\mcaI_{2}(t,s,N)=\int_{0}^{\mu_0}e^{-it\mu^4}e^{i\theta(\mu)s}A_{k\ell}(\mu)e^{b(\mu)N}d\mu,
\end{equation*}
where $s\in\R$ and $N\in[0,\infty)$ and
\begin{equation*}
    A_{k\ell}(\mu)=\begin{cases}
    \mu^{-2}b_k(\mu)b_{\ell}(\mu),&A=\mu^{} QA^{\pm,0}_{11}Q \\
    \mu^{-1}b_k(\mu)b_{\ell}(\mu),&A=\mu^2 QA^{\pm,0}_{21}Q \\
    \mu^{-1}a_k(\mu)c_{\ell}(\mu), &A=\mu^2S_0A^{\pm,0}_{22},\ \mu^2A^{\pm,0}_{23}S_0\\
       a_k(\mu)a_{\ell}(\mu),&A=\mu^3A^{\pm,0}_3
    \end{cases}.
\end{equation*}
Since 
\begin{equation}\label{expan of abc}
    b_k(\mu)=O(\mu),\qquad
c_\ell(\mu)=
\begin{cases}
O(\mu^2),&\ell=1,2,\\
O(\mu^3),&\ell=3,
\end{cases}
\qquad
a_k(\mu)=O(1), \qquad\mu\rightarrow0^+,
\end{equation}
After the change of variables \eqref{varible substi local}, these amplitudes have uniformly bounded first derivatives, so the same argument as for $\mcaI_1$ applies to $\mcaI_2$. Therefore, we complete the proof of \eqref{estimate of KB regular 0}.
\end{proof}
We next turn to the remainder term $K^{\pm}_{r}$.
\begin{proposition}\label{propo of K4 0}
Under the assumptions of Proposition \ref{propo of K 01}, let $K^{\pm}_{r}(t,n,m)$ be as in \eqref{set A0 and kernel in regular case}. Then
\begin{equation}\label{estimate of Kr0}
\sup\limits_{n,m\in\Z}\big|K^{\pm}_r(t,n,m)\big|\lesssim |t|^{-\frac{1}{4}},\qquad t\neq0.
\end{equation}
\end{proposition}
\begin{proof}
 By \eqref{set A0 and kernel in regular case} and Lemma \ref{R0 kernel lemma}, it suffices to estimate finitely many integrals of the form
\begin{equation*}
\mcaI^{\pm}_{\varepsilon}(t,n,m)=\int_{0}^{\mu_0}e^{-it\mu^4}\sum\limits_{m_1,m_2\in\Z}e^{\mp i\theta(\mu)(\varepsilon_1|{N}_1|+\varepsilon_2|{M}_2|)}(v{\Gamma}_{\varepsilon}(\mu)v)(m_1,m_2)e^{b(\mu)((1-\varepsilon_1)|{N}_1|+(1-\varepsilon_2)|{M}_2|)}d\mu,
\end{equation*}
where $\varepsilon=(\varepsilon_1,\varepsilon_2)$, $\varepsilon_j\in\{0,1\}$,  ${N}_1=n-m_1$, ${M}_2=m-m_2$, and ${\Gamma}_{\varepsilon}(\mu)=h_{\varepsilon}(\mu)\mu^{-3}\Gamma^0_4(\mu)$ with 
$$ h_{\varepsilon}(\mu)=(a_1(\mu))^{\varepsilon_1+\varepsilon_2}(a_2(\mu))^{2-(\varepsilon_1+\varepsilon_2)}.
$$
For this purpose, we further write
\begin{equation}\label{mcaIe}
    \mcaI^{\pm}_{\varepsilon}(t,n,m)=\int_{0}^{\mu_0}e^{-it\mu^4}e^{\mp i\theta(\mu)(\varepsilon_1|n|+\varepsilon_2|m|)}\Phi^{\pm}_{\varepsilon}(\mu,n,m)d\mu,
\end{equation}
where
\begin{equation*}
\Phi^{\pm}_{\varepsilon}(\mu,n,m)=\left<\Gamma_{\varepsilon}(\mu)\varphi^{\pm,m}_{\mu,\varepsilon_2},\varphi^{\mp,n}_{\mu,\varepsilon_1}\right>,\qquad \varphi^{\pm,y}_{\mu,\lambda}(x)=v(x)e^{\mp i\lambda\theta(\mu)(|y-x|-|y|)}e^{b(\mu)(1-\lambda)|y-x|}.  
\end{equation*}
Applying the change of variables \eqref{varible substi local} to \eqref{mcaIe}, we obtain
\begin{equation*}
    |\mcaI^{\pm}_{\varepsilon}(t,n,m)|\leq \sup\limits_{s\in\R}\Big|\int_{\theta(\mu_0)}^{0}e^{-it((2-2\cos\theta)^2-s\theta)}\widetilde{\Phi}^{\pm}_{\varepsilon}(\mu(\theta),n,m)d\theta\Big|,
\end{equation*}
where $\widetilde{\Phi}^{\pm}_{\varepsilon}(\mu,n,m)=(a_1(\mu))^{-1}\Phi^{\pm}_{\varepsilon}(\mu,n,m)$. 
Since 
\begin{equation*}
\varphi^{\pm,y}_{\mu(\theta),\lambda}(x)=v(x)e^{\mp i\lambda\theta(|y-x|-|y|)}e^{b(\mu(\theta))(1-\lambda)|y-x|},
\end{equation*}
it follows from $b(\mu)<0$,  $||y-x|-|y||\leq |x|$ and \eqref{estim for ebutheta} that
\begin{equation*}
\|\varphi^{\pm,y}_{\mu(\theta),\lambda}\|_{\ell^2(\Z)}+\int_{\theta(\mu_0)}^{0}\|\partial_\theta\varphi^{\pm,y}_{\mu(\theta),\lambda}\|_{\ell^2(\Z)}d\theta\lesssim \|v\|_{\ell^{2,1}(\Z)}+\sum\limits_{x\in\Z}|v(x)|\int_{\theta(\mu_0)}^{0}\big|\partial_{\theta}e^{b(\mu(\theta))(1-\lambda)|y-x|}\big|d\theta\lesssim 1,
\end{equation*}
uniformly in $\lambda,y$. Together with \eqref{estimate of Gamma}, this yields
\begin{equation*}
\sup\limits_{n,m\in\Z}
\big(
\|\widetilde\Phi^{\pm}_\varepsilon(\mu(\theta),n,m)\|_{L^\infty([\theta(\mu_0),0))}
+
\|\partial_\theta
\widetilde\Phi^{\pm}_\varepsilon(\mu(\theta),n,m)\|_{L^1([\theta(\mu_0),0))}
\big)
\lesssim1.
\end{equation*}
Hence, Corollary \ref{corollary} gives $|\mcaI^{\pm}_{\varepsilon}(t,n,m)|\lesssim|t|^{-1/4}$ uniformly in $n,m$, and \eqref{estimate of Kr0} follows. 
\end{proof}
Combining this 
with Proposition \ref{propo of K 01}, we obtain \eqref{esti for Kpm1} in the regular case.
\subsubsection{{\bf {First kind of resonance case}}}
By \eqref{asy expan of 1st 0} and the estimates established in the regular case, it remains only to consider $\{K^{\pm}_A:A\in\mcaA_1\}$, where
$$\mcaA_1=\{\mu^{-1}S_{1}A^{\pm}_{-1}S_{1},\ S_0A^{\pm,1}_{01}Q,\ QA^{\pm,1}_{02}S_0,\ \mu S_0A^{\pm,1}_{11},\ \mu A^{\pm,1}_{12}S_0,\ \mu^2QA^{\pm,1}_{21},\ \mu^2A^{\pm,1}_{22}Q\}.$$ 
\begin{proposition}\label{propo of KB 0 1st}
Let $H=\Delta^2+V$ with $|V(n)|\lesssim \left<n\right>^{-\beta}$ for $\beta>19$. Suppose that $0$ is a first kind resonance of $H$. Then for any $A\in\mcaA_1$, 
\begin{equation}\label{estimate of KB 0 1st}
|K^{\pm}_A(t,n,m)|\lesssim |t|^{-\frac{1}{4}},\quad t\neq0,\quad uniformly\  in\  n,m\in\Z.
\end{equation}
Consequently, \eqref{esti for Kpm1} follows.
\end{proposition}
\begin{proof}
Similarly, it reduces to oscillatory integrals of the form
\begin{equation*}
\mcaI_{3}(t,s,N)=\int_{0}^{\mu_0}e^{-it\mu^4}e^{i\theta(\mu)s}A_{k\ell}(\mu)e^{b(\mu)N}d\mu,
\end{equation*}
where $s\in\R$ and $N\in[0,\infty)$ and
\begin{equation*}
    A_{k\ell}(\mu)=\begin{cases}
    \mu^{-4}c_k(\mu)c_{\ell}(\mu),&A=\mu^{-1}S_{1}A^{\pm}_{-1}S_{1} \\
    \mu^{-3}b_k(\mu)c_{\ell}(\mu),&A=S_0A^{\pm,1}_{01}Q,\ QA^{\pm,1}_{02}S_0 \\
    \mu^{-2}a_k(\mu)c_{\ell}(\mu), &A=\mu S_0A^{\pm,1}_{11},\ \mu A^{\pm,1}_{12}S_0\\
    \mu^{-1}a_k(\mu)b_{\ell}(\mu),&A=\mu^2QA^{\pm,1}_{21},\ \mu^2A^{\pm,1}_{22}Q
    \end{cases}.
\end{equation*}
By \eqref{expan of abc} and the definitions of $a_k(\mu)$, $b_k(\mu)$, and $c_k(\mu)$, each $A_{k\ell}(\mu)=O(1)$ extends to a smooth function at $\mu=0$. Hence, after the change of variables \eqref{varible substi local}, the transformed amplitudes and their first derivatives are uniformly bounded. Therefore, \eqref{estimate of KB 0 1st} follows by the same argument as for $\mcaI_1$.
\end{proof}
\subsubsection{{\bf {Second kind of resonance case}}}
By \eqref{asy expan 2nd 0} and the estimates established in previous two cases, it suffices to consider
$$\mcaA_2=\Big\{\frac{S_{2}A^{\pm}_{-3}S_{2}}{\mu^3},\ \frac{S_2A^{\pm}_{-2,1}S_0}{\mu^2},\ \frac{S_0A^{\pm}_{-2,2}S_2}{\mu^2},\ \frac{ S_2A^{\pm}_{-1,1}Q}{\mu},\ \frac{ QA^{\pm}_{-1,2}S_2}{\mu},\ S_2A^{\pm,2}_{01},\ A^{\pm,2}_{02}S_2\Big\}.$$ 
\begin{proposition}\label{propo of K-3}
 Let $H=\Delta^2+V$ with $|V(n)|\lesssim \left<n\right>^{-\beta}$ for $\beta>27$. Suppose that $0$ is a second kind resonance of $H$. Then for any $A\in\mcaA_2$, 
\begin{equation}\label{estimate of K-3}
|K^{\pm}_A(t,n,m)|\lesssim |t|^{-\frac{1}{4}},\quad t\neq0,\quad uniformly\  in\  n,m\in\Z.
\end{equation}
Therefore, \eqref{esti for Kpm1} holds.
\end{proposition}
\begin{proof} As before, it  reduces to oscillatory integrals of the form
\begin{equation*}
\mcaI_{4}(t,s,N)=\int_{0}^{\mu_0}e^{-it\mu^4}e^{i\theta(\mu)s}A_{k\ell}(\mu)e^{b(\mu)N}d\mu,
\end{equation*}
where $s\in\R$ and $N\in[0,\infty)$ and
\begin{equation*}
    A_{k\ell}(\mu)=\begin{cases}
    \mu^{-6}d_k(\mu)d_{\ell}(\mu),&A=\mu^{-3}S_{2}A^{\pm}_{-3}S_{2} \\
    \mu^{-5}d_k(\mu)c_{\ell}(\mu),&A=\mu^{-2}S_2A^{\pm}_{-2,1}S_0,\ \mu^{-2}S_0A^{\pm}_{-2,2}S_2 \\
    \mu^{-4}d_k(\mu)b_{\ell}(\mu), &A=\mu^{-1} S_2A^{\pm}_{-1,1}Q,\ \mu^{-1}QA^{\pm}_{-1,2}S_2\\
    \mu^{-3}d_k(\mu)a_{\ell}(\mu),&A=S_2A^{\pm,2}_{01},\ A^{\pm,2}_{02}S_2
    \end{cases}.
\end{equation*}
From \eqref{expan of abc} and $d_k(\mu)=O(\mu^3)$, we also obtain $A_{k\ell}(\mu)=O(1)$ as $\mu\rightarrow0^+$, and the same argument as in the first-kind resonance case gives \eqref{estimate of K-3}.
\end{proof}
Combining the regular, first-kind resonance, and second-kind resonance cases, we obtain \eqref{esti for Kpm1}, thereby completing the proof of the estimate for $K^{\pm}_1(t,n,m)$.

We conclude this subsection with the proof of Lemma \ref{cancelation lemma}.
\begin{proof}[{\bf{\underline{Proof of Lemma \ref{cancelation lemma}}}}]
Set $F^{\pm}_1(s)=e^{\pm is} 
$ and $ F_2(s)=e^{-s}$. Then  \eqref{kernel of R0 boundary} can be rewritten as
\begin{equation}\label{new expre of R0 kernel}
R^{\pm}_0(\mu^4,n-m)=\frac{1}{4\mu^3}\Big(\pm ia_1(\mu)F^{\pm}_1(-\theta(\mu)|n-m|)+a_2(\mu)F_2(-b(\mu)|n-m|)\Big).
\end{equation}

{\bf{\underline{(1)}}} Applying \cite[Lemma 2.5~(i)]{SWY22} to $F^{\pm}_1$ and $F_2$, and using $(F^{\pm}_{1})'(s)=\pm iF^{\pm}_{1}(s), \ F'_2(s)=-F_2(s),$ we obtain
\begin{align}\label{F1F2 taylor Q}
\begin{split}
F^{\pm}_1(-\theta(\mu)|n-m|)&=F^{\pm}_1(-\theta(\mu)|n|)\pm i\theta(\mu)m\int_{0}^{1}({\rm sign}(n-\rho m))F^{\pm}_{1}(-\theta(\mu)|n-\rho m|)d\rho,\\
F_2(-b(\mu)|n-m|)&=F_2(-b(\mu)|n|)-b(\mu)m\int_{0}^{1}({\rm sign}(n-\rho m))F_2(-b(\mu)|n-\rho m|)d\rho.
\end{split}
\end{align}
Substituting these two expressions into \eqref{new expre of R0 kernel} and using $\left<Qf,v\right>=0$, we derive
\begin{equation*}
\big(R^{\pm}_0(\mu^4)vQf\big)(n)
=\frac{1}{4\mu^3}\sum\limits_{m\in\Z}\int_{0}^{1}({\rm sign}(n-\rho m))\big(b_1(\mu)e^{\mp i\theta(\mu)|n-\rho m|}+b_2(\mu)e^{b(\mu)|n-\rho m|}\big)d\rho(v_1Qf)(m),
\end{equation*}
where $b_1(\mu)=-\theta(\mu)a_1(\mu)$ and $b_2(\mu)=-b(\mu)a_2(\mu).$
\vskip0.1cm
{\bf{\underline{(2)}}} Let
\begin{equation}\label{F1pmtuta,F2tuta Sj}
\widetilde{F}^{\pm}_1(s)=F^{\pm}_1(s)\mp is,\qquad  \widetilde{F}_2(s)=F_2(s)+s.
\end{equation}
A direct computation shows that $\big(\widetilde{F}^{\pm}_1\big)'(0)=0=\big(\widetilde{F}_2\big)'(0)$. By \cite[Lemma 2.5~(ii)]{SWY22}, we have
\begin{align}\label{F1F2 taylor Sj}
\begin{split}
\widetilde{F}^{\pm}_1(-\theta(\mu)|n-m|)&=\widetilde{F}^{\pm}_1(-\theta(\mu)|n|)+\theta(\mu)m({\rm sign}(n))\big(\widetilde{F}^{\pm}_1\big)'(-\theta(\mu)|n|)\\
&\quad+{(\theta(\mu))}^2m^2\int_{0}^{1}(1-\rho){\big(\widetilde{F}^{\pm}_1\big)}''(-\theta(\mu)|n-\rho m|)d\rho,\\
\widetilde{F}_2(-b(\mu)|n-m|)&=\widetilde{F}_2(-b(\mu)|n|)+b(\mu)m({\rm sign}(n))\big(\widetilde{F}_2\big)'(-b(\mu)|n|)\\
&\quad+{(b(\mu))}^2m^2\int_{0}^{1}(1-\rho){\big(\widetilde{F}_2\big)}''(-b(\mu)|n-\rho m|)d\rho.
\end{split}
\end{align}
For $j=0,1$, combining \eqref{F1pmtuta,F2tuta Sj} and \eqref{new expre of R0 kernel}, we obtain 
\begin{align*}
\big(R^{\pm}_0(\mu^4)vS_jf\big)(n)
&=\frac{1}{4\mu^3}\sum\limits_{m\in\Z}\Big[\pm ia_1(\mu)\widetilde{F}^{\pm}_1(-\theta(\mu)|n-m|)+a_2(\mu)\widetilde{F}_2(-b(\mu)|n-m|)\\
&\quad+\big(\theta(\mu)a_1(\mu)+b(\mu)a_2(\mu)\big)|n-m|\Big](vS_jf)(m).
\end{align*}
Substituting \eqref{F1F2 taylor Sj} into this expression and using $$\big(\widetilde{F}^{\pm}_1\big)''(s)=-{F}^{\pm}_1(s),\quad   \big(\widetilde{F}_2\big)''(s)=F_2(s)\ {\rm and}\  \left<S_jf,v_k\right>=0,\quad k=0,1,$$
we obtain the assertion in part (2).
\vskip0.1cm
{\bf{\underline{(3)}}} Let
$$\widetilde{F}^{\pm}_1(s)=F^{\pm}_1(s)\mp is+\frac{1}{2}s^2,\qquad \widetilde{F}_2(s)=F_2(s)+s-\frac{1}{2}s^2.$$
Then $\big(\widetilde{F}^{\pm}_1\big)'(0)=\big(\widetilde{F}^{\pm}_1\big)''(0)=0=\big(\widetilde{F}_2\big)''(0)=\big(\widetilde{F}_2\big)'(0)$. By \cite[Lemma 2.5~(iii)]{SWY22}, we have
\begin{align*}
\widetilde{F}^{\pm}_1(-\theta(\mu)|n-m|)&=\widetilde{F}^{\pm}_1(-\theta(\mu)|n|)+\theta(\mu)m\operatorname{sign}(n)\big(\widetilde{F}^{\pm}_1\big)'(-\theta(\mu)|n|)+\big(\widetilde{F}^{\pm}_1\big)''(-\theta(\mu)|n|)\notag\\
\times&\frac{{(\theta(\mu))}^2m^2}{2}+\frac{{(\theta(\mu))}^3m^3}{2}\int_{0}^{1}{(1-\rho)}^2{{\rm sign}(n-\rho m)}^3{\big(\widetilde{F}^{\pm}_1\big)}^{(3)}(-\theta(\mu)|n-\rho m|)d\rho,\notag\\
\widetilde{F}_2(-b(\mu)|n-m|)&=\widetilde{F}_2(-b(\mu)|n|)+b(\mu)m{\rm sign}(n)\big(\widetilde{F}_2\big)'(-b(\mu)|n|)+\frac{{(b(\mu))}^2m^2} {2}\big(\widetilde{F}_2\big)''(-b(\mu)|n|)\notag\\
&\quad+\frac{{(b(\mu))}^3m^3}{2}\int_{0}^{1}{(1-\rho)}^2{({\rm sign}(n-\rho m))}^3{\big(\widetilde{F}_2\big)}^{(3)}(-b(\mu)|n-\rho m|)d\rho.
\end{align*}
Using the same method as in {\bf{\underline{(2)}}} and the property $\left<S_2f,v_k\right>=0$ for $k=0,1,2$, part (3) follows.
\vskip0.1cm
The corresponding identities for $QvR_0^\pm(\mu^4)$ and $S_jvR_0^\pm(\mu^4)$ follow analogously by interchanging the roles of $n$ and $m$ in the above expansions and using $Qv=0$ and $S_jv_k=0$, respectively. This completes the proof.
\end{proof}
\subsection{The estimates of kernels $K^{\pm}_2(t,n,m)$}\label{subse of K2}
In this subsection, we prove \eqref{esti for Kpm2}, thereby completing the proof of Theorem \ref{theorem of estimate of K1,K3}~(ii). Recall from \eqref{kernels of Ki} that
\begin{equation*}
K^{\pm}_{2}(t,n,m)=\left(K^{\pm}_{21}-K^{\pm}_{22}\right)(t,n,m),
\end{equation*}
where
\begin{align}\label{kernels of Kpm21,22}
\begin{split}
K^{\pm}_{21}(t,n,m)&=\int_{\mu_0}^{2-\mu_0}e^{-it\mu^4}\mu^3\big(R^{\pm}_{0}(\mu^4)VR^{\pm}_{0}(\mu^4)\big)(n,m)d\mu,\\
K^{\pm}_{22}(t,n,m)&=\int_{\mu_0}^{2-\mu_0}e^{-it\mu^4}\mu^3\big(R^{\pm}_{0}(\mu^4)VR^{\pm}_V(\mu^4)VR^{\pm}_{0}(\mu^4)\big)(n,m)d\mu.
\end{split}
\end{align}
\begin{proof}[Proof of Theorem \ref{theorem of estimate of K1,K3} (ii)]
To derive \eqref{esti for Kpm2}, it suffices to establish
\begin{align}\label{estimate for kpm21 kpm22}
|K^{\pm}_{21}(t,n,m)|+|K^{\pm}_{22}(t,n,m)|\lesssim |t|^{-\frac{1}{4}}, \quad {\rm uniformly\ in}\ n,m\in\Z.
\end{align}
By \eqref{kernel of R0 boundary}, the estimates for $K^{\pm}_{21}$ reduce to estimating 
\begin{equation*}
\mcaI_{5}(t,s,N)=\int_{\mu_0}^{2-\mu_0}e^{-it\mu^4}e^{i\theta(\mu)s}a_{k\ell}(\mu)e^{b(\mu)N}d\mu,\quad (s,N)\in \R\times [0,\infty),\ \ a_{k\ell}(\mu)=\mu^{-3}a_k(\mu)a_{\ell}(\mu),
\end{equation*}
and the estimates for $K^{\pm}_{22}$ reduce to estimating finitely many integrals of the form
\begin{align*}
\mcaI^{\pm,1}_{\varepsilon}(t,n,m)&=\int_{\mu_0}^{2-\mu_0}e^{-it\mu^4}f_{\varepsilon}(\mu)\sum\limits_{m_1,m_2\in\Z}e^{\mp i\theta(\mu)(\varepsilon_1|{N}_1|+\varepsilon_2|{M}_2|)}(VR^{\pm}_V(\mu^4)V)(m_1,m_2)\\
&\quad\times e^{b(\mu)((1-\varepsilon_1)|{N}_1|+(1-\varepsilon_2)|{M}_2|)}d\mu,
\end{align*}
with $\varepsilon=(\varepsilon_1,\varepsilon_2)$, $\varepsilon_j\in\{0,1\}$, ${N}_1=n-m_1$, ${M}_2=m-m_2$, and 
\begin{equation*}
    f_{\varepsilon}(\mu)=\mu^{-3}(a_1(\mu))^{\varepsilon_1+\varepsilon_2}(a_2(\mu))^{2-(\varepsilon_1+\varepsilon_2)}.
\end{equation*}
Since $\mu\in[\mu_0,2-\mu_0]$ stays away from the thresholds, it follows from \eqref{varible substi local} and Corollary \ref{corollary} that $\mcaI_5$ satisfies \eqref{estimate of mcaI1}.
Moreover, for $\mcaI^{\pm,1}_{\varepsilon}$, all scalar coefficients and their first derivatives are uniformly bounded. Moreover, since $\beta>2$, we may choose $\sigma_j$ with
$$1/2<\sigma_0<\beta-3/2,\qquad 3/2<\sigma_1<\beta-1/2$$
so that the continuity and weighted bound for $\partial^{j}_\mu (R_V^\pm(\mu^4))$ given by Theorem \ref{LAP-theorem} apply. 
Therefore, arguing as for the remainder term $K^{\pm}_r$, Corollary \ref{corollary} yields 
\begin{equation*}
\sup\limits_{n,m\in\Z}|\mcaI^{\pm,1}_{\varepsilon}(t,n,m)|\lesssim |t|^{-1/4},\qquad t\neq0.
\end{equation*}
Thus, \eqref{esti for Kpm2} follows.
\end{proof}
\subsection{The estimate for the kernel $(K^{+}_3-K^{-}_3)(t,n,m)$}\label{subse of K3}
In this subsection, we establish \eqref{esti for Kpm3} under the assumptions in \eqref{cases and conditions 16}. We consider the regular, resonance, and eigenvalue cases separately.
\subsubsection{{\bf {Regular case}}}
Recall from \eqref{kernels of Ki} that
\begin{equation*}
K^{\pm}_{3}(t,n,m)=\int_{2-\mu_0}^{2}e^{-it\mu^4}\mu^3\big[R^{\pm}_0(\mu^4)v\big(M^{\pm}(\mu)\big)^{-1}vR^{\pm}_0(\mu^4)\big](n,m)d\mu.
\end{equation*}
Substituting the expansion \eqref{asy expan regular 2} into the above expression, we obtain
$$K^{\pm}_{3}(t,n,m)=\sum\limits_{B\in\mcaB_0}K^{\pm}_B(t,n,m)+\widetilde{K}^{\pm}_r(t,n,m),$$
where $\mcaB_0=\{\widetilde{Q}B_{0}\widetilde{Q},\ (2-\mu)^{1/2}B^{\pm,0}_{1}\}$ and
\begin{align}\label{kernels of Kpm3i}
\begin{split}
K^{\pm}_{B}(t,n,m)&=\int_{2-\mu_0}^{2}e^{-it\mu^4}\mu^3\big(R^{\pm}_0(\mu^4)vBvR^{\pm}_0(\mu^4)\big)(n,m)d\mu,\qquad B\in\mcaB_0,\\
\widetilde{K}^{\pm}_{r}(t,n,m)&=\int_{2-\mu_0}^{2}e^{-it\mu^4}\mu^3\big(R^{\pm}_0(\mu^4)v\Gamma^0_{1}(2-\mu)vR^{\pm}_0(\mu^4)\big)(n,m)d\mu.
\end{split}
\end{align}
Thus, it suffices to establish the corresponding decay estimates $\{K^{\pm}_{B}:B\in\mcaB_0\}$ and $\widetilde{K}^{\pm}_r$ .
\begin{proposition}\label{propo of Kpm31}
{ Let $H=\Delta^2+V$ with $|V(n)|\lesssim \left<n\right>^{-\beta}$ for $\beta>7$. Suppose that $16$ is a regular point of $H$. Then 
\begin{equation}\label{estimate of K31}
\sum\limits_{B\in\mcaB_0}\sup\limits_{n,m\in\Z}\left|K^{\pm}_{B}(t,n,m)\right|+\sup\limits_{n,m\in\Z}\left|\widetilde{K}^{\pm}_{r}(t,n,m)\right|\lesssim |t|^{-\frac{1}{4}},\quad t\neq0.
\end{equation}
Consequently, \eqref{esti for Kpm3} follows.
}
\end{proposition}
Unlike in Subsection \ref{subse of K1}, it is not straightforward to use the cancellation properties of $\widetilde{Q}$ to eliminate the singularity near $\mu=2$. To overcome this difficulty, we use the unitary operator
$(J\phi)(n)=(-1)^n\phi(n)$, which satisfies
\begin{equation}\label{rela bet 0 and 4 for Laplacian}
JR_{-\Delta}(z)J=-R_{-\Delta}(4-z),\qquad z\in\C\setminus[0,4].
\end{equation}
By the limiting absorption principle, $JR^{\pm}_{-\Delta}(\mu^2)J=-R^{\mp}_{-\Delta}(4-\mu^2)$. Using $J^2=I$ and the identity
\begin{equation}\label{deco of R0}
R^{\pm}_{0}(\mu^4)=\frac{1}{2\mu^2}\left(R^{\pm}_{-\Delta}(\mu^2)-R_{-\Delta}(-\mu^2)\right),\qquad\mu\in(0,2), 
\end{equation}
we can rewrite 
\begin{align}\label{R0vQBQvR0}
\begin{split}
R^{\pm}_0(\mu^4)v\widetilde{Q}B_{0}\widetilde{Q}vR^{\pm}_0(\mu^4)&=\frac{1}{4\mu^4}\Big[\Big(JR^{\mp}_{-\Delta}(4-\mu^2)\tilde{v}\widetilde{Q}+R_{-\Delta}(-\mu^2)v\widetilde{Q}\Big)B_{0}\widetilde{Q}\tilde{v}R^{\mp}_{-\Delta}(4-\mu^2)J\\
&\quad\quad\quad+\Big(JR^{\mp}_{-\Delta}(4-\mu^2)\tilde{v}\widetilde{Q}+R_{-\Delta}(-\mu^2)v\widetilde{Q}\Big)B_{0}\widetilde{Q}vR_{-\Delta}(-\mu^2)\Big].
\end{split}
\end{align}
This representation shows that the analysis of the singularity at $\mu=2$ reduces to that of  $R^{\mp}_{-\Delta}(4-\mu^2)\tilde{v}\widetilde{Q}$ and $\widetilde{Q}\tilde{v}R^{\mp}_{-\Delta}(4-\mu^2)$.
\begin{lemma}\label{cancelation lemma16}
Let $\widetilde{Q}$ be as in Definition \ref{definition of Sjtilde} and $\tilde{v}=Jv$. Then
\begin{align*}
    \big(R^{\mp}_{-\Delta}(4-\mu^2)\tilde{v}\widetilde{Q}f\big)(n)&=(2{\rm sin}\tilde{\theta}(\mu))^{-1}\sum\limits_{m\in\Z}\widetilde{\mcaB}^{\pm}(\tilde{\theta}(\mu),n,m)(\widetilde{Q}f)(m),\\
\widetilde{Q}\big(\tilde{v}R^{\mp}_{-\Delta}(4-\mu^2)f\big)&=(2{\rm sin}\tilde{\theta}(\mu))^{-1}\widetilde{Q}\big(\sum\limits_{m\in\Z}\widetilde{\mcaB}^{\pm}(\tilde{\theta}(\mu),m,\cdot)f(m)\big),
\end{align*}
where $\tilde{\theta}(\mu)\in(-\pi,0)$ satisfies $\cos\tilde{\theta}(\mu)=\frac{\mu^2}{2}-1$ and
\begin{equation*}
    \widetilde{\mcaB}^{\pm}(\tilde{\theta}(\mu),n,m)={\bm {\tilde{\theta}(\mu)}}\tilde{v}_1(m)\int_{0}^{1}({\rm sign}(n-\rho m))e^{\pm i\tilde{\theta}(\mu)|n-\rho m|}d\rho.
\end{equation*}
\end{lemma}
\begin{proof}[Proof of Lemma \ref{cancelation lemma16}]
It follows from \eqref{kernel of lapa boundary} that
\begin{equation}\label{Laplace 4-mu2}
R^{\mp}_{-\Delta}(4-\mu^2,n-m)=\frac{\pm ie^{\pm i\tilde{\theta}(\mu)|n-m|}}{2\sin\tilde{\theta}(\mu)}=\frac{\mp ie^{\pm i\tilde{\theta}(\mu)|n-m|}}{\mu\sqrt{4-\mu^2}},
\end{equation}
where $\tilde{\theta}(\mu)\in(-\pi,0)$ satisfies $\cos\tilde{\theta}(\mu)=\frac{\mu^2}{2}-1$.
Using the properties $\big<\widetilde{Q}f,\tilde{v}\big>=0$ and $\widetilde{Q}\tilde{v}=0$, and arguing as in the proof of Lemma \ref{cancelation lemma}, we obtain the desired result.
\end{proof}
Since $\tilde{\theta}(\mu)=O((2-\mu)^{\frac{1}{2}})$ as $\mu\rightarrow2$, this lemma shows that these operators can eliminate the singularity of $R^{\mp}_{-\Delta}(4-\mu^2)$.
\begin{proof}[Proof of Proposition \ref{propo of Kpm31}]
(1) We first consider $B=\widetilde{Q}B_{0}\widetilde{Q}$. By \eqref{R0vQBQvR0} and Lemma \ref{cancelation lemma16}, arguing as for $K^{\pm}_{01}$ in Proposition \ref{propo of K 01}, it suffices to establish the corresponding decay estimates for the oscillatory integrals  $\{{\mcaI}_{1j}:j=1,2\}\cup\{\widetilde{\mcaI}_2\}$, where 
\begin{equation}\label{mcaIj mcaI2}
{\mcaI}_{1j}(t,s,N)=\int_{2-\mu_0}^{2}e^{-it\mu^4}e^{i\widetilde{\theta}(\mu)s}B_{j}(\mu)e^{b(\mu)N}d\mu,\ \ \widetilde{\mcaI}_{2}(t,N)=\int_{2-\mu_0}^{2}e^{-it\mu^4}\mu^{-3}(a_2(\mu))^2e^{b(\mu)N}d\mu,
\end{equation}
with $(s,N)\in\R\times[0,\infty)$ and 
\begin{equation}\label{Bj mu}
B_1(\mu)=\mu^{-1}\frac{(\tilde{\theta}(\mu))^2}{\sin^2\tilde{\theta}(\mu)},\qquad B_2(\mu)=\mu^{-2}a_2(\mu)\frac{\tilde{\theta}(\mu)}{\sin\tilde{\theta}(\mu)}.
\end{equation}
The $|t|^{-1/4}$ decay estimate for $\widetilde{\mcaI}_2$   follows directly from the van der Corput lemma and \eqref{integral ebnm}.
For ${\mcaI}_{1j}$,  we use the change of variables $\theta=\tilde{\theta}(\mu)$, i.e.,
\begin{equation}\label{varible substi2}
\cos{\theta}=\frac{\mu^2}{2}-1 \Longrightarrow\  \frac{d\mu}{d\theta}=-\frac{\sin\theta}{\mu},\qquad {\theta}\rightarrow0\quad {\rm as}\quad  \mu\rightarrow2^-.
\end{equation}
This gives
\begin{equation*}
{\mcaI}_{1j}(t,s,N)=\int_{\tilde\theta(2-\mu_0)}^{0}e^{-it((2+2\cos\theta)^2-\frac{s}{t}\theta)}\widetilde{B}_j(\mu(\theta))e^{b(\mu(\theta))N}d\theta,\qquad t\neq0,
 \end{equation*}
where 
$$\widetilde{B}_1(\mu(\theta))=\frac{-{\theta}^2}{(2+2\cos{\theta})\sin\theta},\qquad \widetilde{B}_2(\mu(\theta))=-\frac{\theta a_2(\mu(\theta))}{(2+2\cos{\theta})^{3/2}}.$$
Since $\widetilde B_j$ and their first derivatives with respect to $\theta$ are uniformly bounded on
$[\tilde\theta(2-\mu_0),0)$ for $j=1,2$, Corollary \ref{corollary}, together with the analogue of \eqref{estim for ebutheta}, yields
\begin{equation}\label{estim of mcaI1j}
\sup_{s\in\R,N\ge0}
|{\mcaI}_{1j}(t,s,N)|
\lesssim |t|^{-1/4},\qquad t\neq0,
\qquad j=1,2.
\end{equation}
\vskip0.1cm
(2) We next consider $B=(2-\mu)^{1/2}B^{\pm,0}_{1}$. By \eqref{kernels of Kpm3i} and \eqref{kernel of R0 boundary}, it suffices to estimate finitely many integrals of the form 
\begin{align*}
\widetilde{\mcaI}_{\varepsilon}(t,s,N)&=\int_{2-\mu_0}^{2}e^{-it\mu^4}e^{i\theta(\mu)s}q_{\varepsilon}(\mu)e^{b(\mu)N}d\mu
\end{align*}
where $(s,N)\in \R\times [0,\infty)$, $\varepsilon=(\varepsilon_1,\varepsilon_2)$, $\varepsilon_j\in\{0,1\}$, and 
\begin{equation*}
    q_{\varepsilon}(\mu)=\mu^{-3}(2-\mu)^{\frac{1}{2}}(a_1(\mu))^{\varepsilon_1+\varepsilon_2}(a_2(\mu))^{2-(\varepsilon_1+\varepsilon_2)}.
\end{equation*}
Applying the change of variables $\theta=\theta(\mu)$ to each $\widetilde{\mcaI}_{\varepsilon}(t,s,N)$,
\begin{equation*}
 \cos\theta=1-\frac{\mu^2}{2} \Longrightarrow\  \frac{d\mu}{d\theta}=-\frac{1}{2}\sqrt{(2+\mu)(2-\mu)},\qquad {\theta}\rightarrow-\pi\quad {\rm as}\quad  \mu\rightarrow2.
 \end{equation*}
we obtain 
\begin{equation*}
\widetilde{\mcaI}_{\varepsilon}(t,s,N)=\int_{-\pi}^{\theta(2-\mu_0)}e^{-it((2-2\cos\theta)^2-\frac{s}{t}\theta)}\widetilde{q}_{\varepsilon}(\mu(\theta))e^{b(\mu(\theta))N}d\theta,
\end{equation*}
where 
\begin{equation*}
 \widetilde{q}_{\varepsilon}(\mu)=(2-\mu)^{\frac{2-(\varepsilon_1+\varepsilon_2)}{2}}p_{\varepsilon}(\mu),\qquad p_{\varepsilon}(\mu)\in C^{\infty}([2-\mu_0,2]).
\end{equation*}
Since $\mu'(\theta)=-\frac{1}{2}\sqrt{(2+\mu)(2-\mu)}$, it follows that for any $k=0,1$ and 
 $\varepsilon=(\varepsilon_1,\varepsilon_2)$ with $\varepsilon_j\in\{0,1\}$, 
\begin{equation*}
|\partial^k_{\theta}\big(\widetilde{q}_{\varepsilon}(\mu(\theta))\big)|\lesssim
1 ,\qquad \theta\in(-\pi,\theta(2-\mu_0)].
\end{equation*}
Moreover, the estimate \eqref{estim for ebutheta} is also valid on the interval $(-\pi,\theta(2-\mu_0)]$. Therefore, Corollary \ref{corollary} gives the desired estimate.

(3) For the remainder term $\widetilde{K}^{\pm}_r$, by \eqref{kernel of R0 boundary}, it suffices to estimate the oscillatory integrals, which have the same form as ${\mcaI}^{\pm}_{\varepsilon}(t,n,m)$ in Proposition \ref{propo of K4 0}, with the interval $(0,\mu_0]$ replaced by $[2-\mu_0,2)$ and $\Gamma_\varepsilon(\mu)$ replaced by $\widetilde{\Gamma}_{\varepsilon}(\mu)=\widetilde{h}_{\varepsilon}(\mu)\Gamma^0_1(2-\mu)$. Here  
$$ \widetilde{h}_{\varepsilon}(\mu)=\mu^{-3}(a_1(\mu))^{\varepsilon_1+\varepsilon_2}(a_2(\mu))^{2-(\varepsilon_1+\varepsilon_2)}.
$$
By \eqref{estimate of Gamma}, we note that
\begin{equation*}
\|\widetilde{\Gamma}_{\varepsilon}(\mu)\|_{\ell^2\rightarrow\ell^2}\lesssim (2-\mu)^{\frac{2-(\varepsilon_1+\varepsilon_2)}{2}},\qquad \mu\in[2-\mu_0,2).
\end{equation*}
Since $\mu'(\theta)$ contributes an additional factor of $(2-\mu)^{1/2}$ near $\mu=2$, the desired estimate follows by the same argument as in the proof of $K^{\pm}_r$. Combining (1)--(3), we obtain \eqref{estimate of K31}.
\end{proof}
\begin{remark}
{\rm We note that both changes of variables \eqref{varible substi1} and \eqref{varible substi2} play important roles in handling the oscillatory integrals. However, their effects on the endpoint singularities are different. Specifically, \eqref{varible substi1} preserves the order of the singularity near $\mu=0$, whereas \eqref{varible substi2} introduces an additional factor of order $(2-\mu)^{\frac{1}{2}}$, thereby reducing the singularity near $\mu=2$.
}
\end{remark}

\subsubsection{{\bf {Resonant case}}} By \eqref{asy expan resonance 2} and the estimates established in the regular case, it remains to consider
\begin{equation*}
    \mcaB_{1}=\{(2-\mu)^{-\frac{1}{2}}\widetilde{S}_0B^{\pm}_{-1}\widetilde{S}_0,\ \widetilde{S}_0B^{\pm,1}_{01},\ B^{\pm,1}_{02}\widetilde{S}_0\}
\end{equation*}

\begin{proposition}\label{propo of resonance 2}
Let $H=\Delta^2+V$ with $|V(n)|\lesssim \left<n\right>^{-\beta}$ for $\beta>11$. Assume that $16$ is a resonance of $H$, then for any $B\in\mcaB_{1}$, we have
\begin{equation}\label{estimate of K3 resonance}
\left|K^{\pm}_{B}(t,n,m)\right|\lesssim |t|^{-\frac{1}{4}},\quad t\neq0,\qquad {\rm uniformly\ in}\ n,m\in\Z.
\end{equation}
Therefore, \eqref{esti for Kpm3} holds.
\end{proposition}
\begin{proof}
Since $\big<\widetilde{S}_0f,\tilde{v}\big>=0,\widetilde{S}_0\tilde{v}=0$, Lemma \ref{cancelation lemma16} remains valid with $\widetilde Q$ replaced by $\widetilde S_0$. Arguing as in part {\rm(1)} of the proof of Proposition \ref{propo of Kpm31}, the estimate for $K^{\pm}_B$ with $B=(2-\mu)^{-\frac{1}{2}}\widetilde{S}_0B^{\pm}_{-1}\widetilde{S}_0$ reduces to estimating 
\begin{align*}
\widetilde{\mcaI}_{1j}(t,s,N)&=\int_{2-\mu_0}^{2}e^{-it\mu^4}e^{i\widetilde{\theta}(\mu)s}C_j(\mu)e^{b(\mu)N}d\mu,\qquad C_j(\mu)=(2-\mu)^{-\frac12}B_{j}(\mu),\qquad j=1,2,\\
\widetilde{\widetilde{\mcaI}}_{2}(t,N)&=\int_{2-\mu_0}^{2}e^{-it\mu^4}\mu^{-3}(a_2(\mu))^2(2-\mu)^{-\frac{1}{2}}e^{b(\mu)N}d\mu,
\end{align*}
while the estimates for $K^{\pm}_B$ with $B=\widetilde{S}_0B^{\pm,1}_{01},B^{\pm,1}_{02}\widetilde{S}_0$ reduce to those for $\mcaI_{12}$, $\widetilde{\mcaI}_2$ in \eqref{mcaIj mcaI2} and  
\begin{equation*}
    \widetilde{\mcaI}_{3j}(t,s,N)=\int_{2-\mu_0}^{2}e^{-it\mu^4}e^{i\widetilde{\theta}(\mu)s}D_j(\mu)e^{b(\mu)N}d\mu,\qquad D_j(\mu)=(\tilde{\theta}(\mu))^{-1}B_j(\mu),\qquad j=1,2,
\end{equation*}
where $B_j(\mu)$ is defined in \eqref{Bj mu}. Since $C_j(\mu)=O((2-\mu)^{-1/2})$, $D_j(\mu)=O((2-\mu)^{-1/2})$, and the change of variables \eqref{varible substi2} contributes a factor of order $(2-\mu)^{1/2}$, the transformed amplitudes and their first derivatives are uniformly bounded. Therefore, arguing as for $\mcaI_{1j}$, Corollary \ref{corollary} yields the $|t|^{-1/4}$ decay estimates for $\widetilde{\mcaI}_{1j}$, $\widetilde{\widetilde{\mcaI}}_2$ and $\widetilde{\mcaI}_{3j}$. Hence, \eqref{estimate of K3 resonance} follows.
\end{proof}
\subsubsection{{\bf {Eigenvalue case}}} Combining \eqref{asy expan eigenvalue 2} with the estimates established in the regular and resonance cases, \eqref{esti for Kpm3} reduces to estimating $K:=K^+_B-K^-_B$ with $B=(2-\mu)^{-1}\widetilde{S}_1B_{-2}\widetilde{S}_1$. Although Lemma \ref{cancelation lemma16} remains valid with $\widetilde{Q}$ replaced by  $\widetilde{S}_1$, the resulting cancellation is insufficient to remove the stronger singularity $(2-\mu)^{-1}$. We therefore use the identity
\begin{align*}
&R^+_0(\mu^4)v\widetilde{S}_1B_{-2}\widetilde{S}_1vR^+_0(\mu^4)-R^-_0(\mu^4)v\widetilde{S}_1B_{-2}\widetilde{S}_1vR^-_0(\mu^4)\\
=&\big((R^+_0-R^-_0)(\mu^4)\big)v\widetilde{S}_1B_{-2}\widetilde{S}_1vR^+_0(\mu^4)+R^-_0(\mu^4)v\widetilde{S}_1B_{-2}\widetilde{S}_1v\big((R^+_0-R^-_0)(\mu^4)\big).
\end{align*}
The two terms on the right-hand side can be treated in the same way. Hence, it suffices to estimate 
\begin{equation*}
{\widetilde{K}}(t,n,m)=\int_{2-\mu_0}^{2}e^{-it\mu^4}(2-\mu)^{-1}\mu^3\big[((R^+_0-R^-_0)(\mu^4))v\widetilde{S}_1B_{-2}\widetilde{S}_1vR^+_0(\mu^4)\big](n,m)d\mu.
\end{equation*}
By \eqref{rela bet 0 and 4 for Laplacian} and \eqref{deco of R0}, it is key to establish the following cancellation lemma.
\begin{lemma}\label{cancelation lemma16 eigenvalue}
Let $\widetilde{S}_1$ be as in Definition \ref{definition of Sjtilde} and $\tilde{v}=Jv$. Then 
\begin{equation*}
\big[(R^{-}_{-\Delta}-R^{+}_{-\Delta})(4-\mu^2)\tilde{v}\widetilde{S}_1f\big](n)={\bm{(\tilde{\theta}(\mu))^2}}(2\sin\tilde{\theta}(\mu))^{-1}\sum\limits_{m\in\Z}\widetilde{\mcaC}(\tilde{\theta}(\mu),n,m)(\widetilde{S}_1f)(m),
\end{equation*}
where $\tilde{\theta}(\mu)\in(-\pi,0)$ satisfies $\cos\tilde{\theta}(\mu)=\frac{\mu^2}{2}-1$ and
\begin{equation*}
\widetilde{\mcaC}(\tilde{\theta}(\mu),n,m)=i\tilde{v}_2(m)\int_{0}^{1}(\rho-1)(e^{i\tilde{\theta}(\mu)|n-\rho m|}+e^{-i\tilde{\theta}(\mu)|n-\rho m|})d\rho
\end{equation*}
\end{lemma}
\begin{proof} Let $F(s)=i(e^{-is}+e^{is})$. Then
$$(R^{-}_{-\Delta}-R^{+}_{-\Delta})(4-\mu^2,n-m)=(2\sin\tilde{\theta}(\mu))^{-1}F(-\tilde{\theta}(\mu)|n-m|).$$
Since $F'(0)=0$, together with $\big<\widetilde{S}_1f,\tilde{v}_k\big>=0$ and $\widetilde{S}_1\tilde{v}_k=0$ for $k=0,1$, the desired conclusion follows by the same argument as in the proof of Lemma \ref{cancelation lemma}~(2).
\end{proof}
This lemma shows that $\big((R^{-}_{-\Delta}-R^{+}_{-\Delta})(4-\mu^2)\big)\tilde{v}\widetilde{S}_1=O((2-\mu)^{\frac{1}{2}})$ as $\mu\rightarrow2^-$, which together with Lemma \ref{cancelation lemma16} gives 
$$(2-\mu)^{-1}\mu^3\big[((R^+_0-R^-_0)(\mu^4))v\widetilde{S}_1B_{-2}\widetilde{S}_1vR^+_0(\mu^4)\big](n,m)=O((2-\mu)^{-1/2}),\qquad \mu\rightarrow2^-.$$
By the argument as in Proposition \ref{propo of resonance 2}, we can obtain the following conclusion.
\begin{proposition}\label{propo of eigenvalue 2}
Let $H=\Delta^2+V$ with $|V(n)|\lesssim \left<n\right>^{-\beta}$ for $\beta>15$. Assume that $16$ is an eigenvalue of $H$, then \eqref{esti for Kpm3} is obtained.
\end{proposition}
Hence, combining this with Propositions \ref{propo of Kpm31} and \ref{propo of resonance 2}, we complete the proof of Theorem \ref{theorem of estimate of K1,K3}~(iii). Together with the results established in the previous two subsections, this completes the proof of Theorem \ref{main-theorem}.

\section{Proof of asymptotic expansion Theorem \ref{Asymptotic expansion theorem}}\label{proof of Asy theorem}
This section is devoted to presenting the proof of Theorem \ref{Asymptotic expansion theorem}, namely, the asymptotic expansions of $(M^{\pm}(\mu))^{-1}$. Recall from \eqref{Mpmmu} that $M^{\pm}(\mu)$ is given by
\begin{equation*}
M^{\pm}(\mu)=U+vR^{\pm}_{0}(\mu^4)v,\quad\mu\in(0,2),\quad U={\rm sign}\left(V(n)\right),\quad v(n)=\sqrt{|V(n)|}.
\end{equation*}
To this end, we first study the asymptotic behaviors of $R^{\pm}_{0}(\mu^4)$ near $\mu=0$ and $\mu=2$.
\begin{lemma}\label{Puiseux expan of R0}
Let $N$ be an integer and $\mu\in(0,2)$. 
\vskip0.1cm
{\rm (i)} For any $N\geq-3$, the following expansion holds in the norm of $\B(s,-s)$ for $s>\frac{9}{2}+N$:
\begin{equation}\label{Puiseux expan of R0 0}
R^{\pm}_0(\mu^4)=\sum\limits_{j=-3}^{N}\mu^{j}G^{\pm}_j+r^{\pm}_{N+1}(\mu),\qquad \left\|r^{\pm}_{N+1}(\mu)\right\|_{\B(s,-s)}=O\big(\mu^{N+1}\big),\quad\mu\rightarrow0^+,
\end{equation}
where $G^{\pm}_j\in\B(\sigma,-\sigma)$ for $\sigma>\frac{7}{2}+j$ with kernel $G^{\pm}_j(n-m)$ given by 
\begin{equation*}
   G^{\pm}_{2q-3}(n)= \frac{(\pm i(-1)^q-1)}{2^{2q+2}(2q)!}\prod\limits_{j=0}^{q-1}(4|n|^2-(2j+1)^2),\qquad  G^\pm_{2q-2}(n)=\frac{|n|(1+(-1)^{q+1})}{4(2q+1)!}\prod\limits_{j=1}^{q}(|n|^2-j^2).
\end{equation*}
In particular, $G^{\pm}_{4q+2}(n)=0$ for any $q\geq-1$, $G^{\pm}_{-3}(n)=\frac{-1\pm i}{4}$ and 
\begin{align}\label{expan coeffie of Gpm}
\begin{split}
  G^{\pm}_{-1}(n)&=\frac{1\pm i}{4}\big(\frac{1}{8}-\frac{1}{2}|n|^2\big), \qquad G^{\pm}_{1}(n)=\frac{-1\pm i}{32}\big(\frac{1}{3}|n|^4-\frac{5}{6}|n|^2+\frac{3}{16}\big), \\
 G^{\pm}_{0}(n)&=\frac{1}{12}\left(|n|^3-|n|\right), \qquad G^{\pm}_{3}(n)=\frac{-1\mp i}{4\cdot6!}\big(|n|^6-\frac{35}{4}|n|^4+\frac{259}{16}|n|^2-\frac{225}{64}\big).
\end{split}
\end{align}
Moreover, in the same sense, \eqref{Puiseux expan of R0 0} can be differentiated $N+4$ times in $\mu$.
\vskip0.2cm
{\rm(ii)}
For any $N\geq-1$, the following expansion holds in the norm of $\B(s,-s)$ for $s>\frac{5}{2}+N$:
\begin{equation}\label{Puiseux expan of R0 2}
JR^{\pm}_0((2-\mu)^4)J=\sum\limits_{j=-1}^{N}\mu^{\frac{j}{2}}\widetilde{G}^{\pm}_j+\mcaE^{\pm}_{N+1}(\mu),\qquad \left\|\mcaE^{\pm}_{N+1}(\mu)\right\|_{\B(s,-s)}=O\big(\mu^{\frac{N+1}{2}}\big),\quad \mu\rightarrow0^+,
\end{equation}
where $\widetilde{G}^{\pm}_j\in\B(\sigma,-\sigma)$ for $\sigma>\frac{3}{2}+j$ with kernel $\widetilde{G}^{\pm}_j(n-m)$. In particular, $\widetilde{G}^{\pm}_{-1}(n)=\frac{\pm i}{32}$ and
\begin{align}\label{expan coeffie of Gtutapm}
\begin{split}
\widetilde{G}^{\pm}_{0}(n)&=\frac{1}{32\sqrt2}\big(2\sqrt2|n|-{(2\sqrt2-3)}^{|n|}\big),\qquad  \widetilde{G}^{\pm}_{1}(n)=\frac{\mp i}{32}\big(2|n|^2-\frac{13}{8}\big),\\
\widetilde{G}^{\pm}_{2}(n)&=-\frac{1}{24}|n|^3+\frac{5}{48}|n|-\frac{\big(2\sqrt2-3\big)^{|n|}}{32}\Big(\frac{1}{2}|n|+\frac{7}{4\sqrt2}\Big).
\end{split}
\end{align}
Furthermore, in the same sense, \eqref{Puiseux expan of R0 2} can be differentiated $N+2$ times in $\mu$.
\end{lemma}
\begin{proof}
Let $x=|n|$ and fix $0<\mu_0\ll1$. Denote $\Phi^{\pm}(\mu,x)=\pm i a_1(\mu)e^{\mp ix\theta(\mu)}+a_2(\mu)e^{b(\mu)x}$. 
\vskip0.1cm
\textbf{(i) Expansion near $\mu=0$.}
For any integer $N\geq-3$, expanding $\Phi^{\pm}(\mu,x)$ in a Taylor series at $\mu=0$ up to order $N+3$, we obtain
\[
R^{\pm}_0(\mu^4,n)=\frac{1}{4}\mu^{-3}\Phi^{\pm}(\mu,x)
=\sum_{j=-3}^{N}\mu^jG_j^{\pm}(x)+r_{N+1}^{\pm}(\mu,x),
\]
where
\[
G_j^{\pm}(x)=\frac{(\partial_\mu^{j+3}\Phi^{\pm})(0,x)}{4(j+3)!},\qquad
r_{N+1}^{\pm}(\mu,x)
=\frac{\mu^{N+1}}{4(N+3)!}\int_0^1(1-t)^{N+3}
(\partial_\mu^{N+4}\Phi^{\pm})(t\mu,x)\,dt.
\]
Since $a_1,a_2,\theta,b\in C^{\infty}([0,\mu_0])$, $b(0)=0$, and $b(\mu)<0$ for $\mu>0$, we have, for $0\leq k\leq N+4$,
\[
\bigl|(\partial_\mu^k\Phi^{\pm})(\mu,x)\bigr|\lesssim\left<x\right>^k,
\qquad 0\leq\mu\leq\mu_0.
\]
Indeed, $|e^{\mp ix\theta(\mu)}|=1$ and $|e^{b(\mu)x}|\leq1$, and each $\mu$-derivative produces at most one additional factor of $x$. Differentiating the Taylor remainder (equivalently, applying Taylor's formula to the corresponding derivatives) and then using the Leibniz rule for the factor $\mu^{-3}$ yields
\begin{equation}\label{estim for remainder r}
\bigl|(\partial_\mu^{\ell}r_{N+1}^{\pm})(\mu,x)\bigr|
\lesssim \mu^{N+1-\ell}\left<x\right>^{N+4},
\qquad 0<\mu<\mu_0,\quad 0\leq\ell\leq N+4.  
\end{equation}

To compute $G_j^{\pm}(x)$, recall that $\theta(\mu)=-2\arcsin\frac{\mu}{2}$ and $b(\mu)=-2\operatorname{arsinh}\frac{\mu}{2}$. We introduce 
$$\Phi^{\pm}_1(\mu,x)=e^{\pm 2ix\arcsin{\mu}},\qquad \Phi_2(\mu,x)=e^{-2x\operatorname{arsinh}\mu}.$$ 
For each fixed $x$, these functions satisfy
\begin{align}\label{rel Phipm1 Phi2}
\begin{split}
(1-\mu^2)(\partial^2_\mu\Phi^{\pm}_1)(\mu,x)+4x^2\Phi^{\pm}_1(\mu,x)&=\mu (\partial_\mu\Phi^{\pm}_1)(\mu,x),\\
(1+\mu^2)(\partial^2_\mu\Phi_2)(\mu,x)+\mu(\partial_\mu\Phi_2)(\mu,x)&=4x^2\Phi_2(\mu,x).
\end{split}
\end{align}
Write
\begin{equation}\label{talor expa of Phi1 2}
\Phi^{\pm}_1(\mu,x)=\sum\limits_{k=0}^{\infty}C^{\pm}_k(x)\mu^k,\qquad \Phi_2(\mu,x)=\sum\limits_{k=0}^{\infty}D_k(x)\mu^k,\qquad 0<\mu<\mu_0.
\end{equation}
In particular, $C^{\pm}_0(x)=D_0(x)=1$, $C^{\pm}_1(x)=\pm 2ix$ and $D_1(x)=-2x$. Substituting the series into  \eqref{rel Phipm1 Phi2} gives
\[
C_{k+2}^{\pm}(x)=\frac{k^2-4x^2}{(k+2)(k+1)}C_k^{\pm}(x),
\qquad
D_{k+2}(x)=\frac{4x^2-k^2}{(k+2)(k+1)}D_k(x).
\]
Consequently,
\[
C_{2q}^{\pm}(x)=(-1)^qD_{2q}(x),
\qquad
C_{2q+1}^{\pm}(x)=\mp i(-1)^qD_{2q+1}(x),
\qquad q\geq0,
\]
where, with the convention that an empty product equals $1$,
\begin{equation}\label{Dkx}
   D_{2q}(x)=
       \frac{1}{(2q)!}\prod\limits_{j=0}^{q-1}(4x^2-4j^2),\qquad D_{2q+1}(x)=
       \frac{-2x}{(2q+1)!}\prod\limits_{j=0}^{q-1}(4x^2-(2j+1)^2).
\end{equation}
Since $\Phi^{\pm}(\mu,0)=\pm ia_1(\mu)+a_2(\mu)$, and for $x\neq0$, 
$$\Phi^{\pm}(\mu,x)=
    \frac{1}{2x}\big((\partial_{\mu}\Phi^{\pm}_1)(\frac{\mu}{2},x)+(\partial_{\mu}\Phi_2)(\frac\mu2,x)\big).$$ 
Combining the preceding identities with \eqref{talor expa of Phi1 2} and \eqref{Dkx}, we obtain, for all $x\geq0$,
\begin{equation*}
R^{\pm}_0(\mu^4,x)=\sum\limits_{q=0}^{+\infty} G^{\pm}_{2q-3}(x)\mu^{2q-3}+\sum\limits_{q=0}^{+\infty} G^\pm_{2q-2}(x)\mu^{2q-2}, \qquad 0<\mu\leq \mu_0,
\end{equation*}
where
\begin{equation*}
   G^{\pm}_{2q-3}(x)= \frac{(\pm i(-1)^q-1)}{2^{2q+2}(2q)!}\prod\limits_{j=0}^{q-1}(4x^2-(2j+1)^2),\qquad  G^\pm_{2q-2}(x)=\frac{x(1+(-1)^{q+1})}{4(2q+1)!}\prod\limits_{j=1}^{q}(x^2-j^2).
\end{equation*}
In particular, $G^\pm_{4k+2}(x)=0$ for $k\geq-1$. Since $\left<n-m\right>^{2j}\lesssim\left<n\right>^{2j}\left<m\right>^{2j}$, it follows that $$\sum\limits_{n,m\in\Z}\left<n\right>^{-2s}\left<n-m\right>^{2j}\left<m\right>^{-2s}<\infty$$ for any $s>1/2+j$. This, together with \eqref{estim for remainder r} proves \eqref{Puiseux expan of R0 0}.

\vskip0.1cm
\textbf{(ii) Expansion near the upper endpoint.}
Since $(-1)^{n+m}=(-1)^{|n-m|}$, we have
\[
\bigl(JR_0^{\pm}((2-\mu)^4)J\bigr)(n,m)
=(-1)^{|n-m|}R_0^{\pm}((2-\mu)^4,n-m).
\]
Hence it suffices to study, with $x\in\N\cup\{0\}$,
\begin{equation*}
\widetilde{R}^{\pm}_0(\mu,x)=\frac{(-1)^x}{4(2-\mu)^3}\Phi^{\pm}(2-\mu,x).
\end{equation*}
By the definition of $\theta(2-\mu)$, it follows that
$\theta(2-\mu)=-\pi+\alpha(\mu)$ with $0<\alpha(\mu)=4\arcsin{\frac{\sqrt{\mu}}{2}}.$
Let $T_x$ and $U_x$ denote the Chebyshev polynomials of the first and second kind, respectively. Thus
\[
T_x(y)=
\begin{cases}
\frac{x}{2}\sum\limits_{k=0}^{[x/2]}(-1)^k
\frac{(x-k-1)!}{k!(x-2k)!}(2y)^{x-2k},&x\geq1,\\
1,&x=0,
\end{cases},\ \ 
U_x(y)=\sum_{k=0}^{[x/2]}(-1)^k
\binom{x-k}{k}(2y)^{x-2k}\ {\rm for}\  x\geq0,
\]
with the convention $U_{-1}=0$.
Using Euler's formula, $\cos\alpha(\mu)=1-2\mu+\frac{\mu^2}{2}$ and 
$$\cos (\alpha x)=T_x(\cos \alpha),\qquad \sin (\alpha x)= U_{x-1}(\cos\alpha)\sin \alpha,$$
we obtain
\begin{equation}\label{eq:upper-cheb}
\widetilde{R}^{\pm}_0(\mu,x)=\frac{\pm iT_x(1+\beta(\mu))}{2\sqrt{\mu}(2-\mu)^3\sqrt{(4-\mu)}}+\frac{U_{x-1}(1+\beta(\mu))}{4(2-\mu)^2}-\frac{(q(\mu))^x}{2(2-\mu)^3d(\mu)}
\end{equation}
where
$$\beta(\mu)=\frac{\mu^2}{2}-2\mu,\qquad d(\mu)=\sqrt{8-4\mu+\mu^2},\qquad q(\mu)=-\frac{1}{4}\big(2-\mu-d(\mu)\big)^2.$$
Since $T_x$ has degree $x$ for $x\geq0$ and $U_{x-1}$ has degree $x-1$ for $x\geq1$ (with $U_{-1}=0$), we have
\[
T_x(1+\beta(\mu))
=\sum_{\ell=0}^{x}\frac{T_x^{(\ell)}(1)}{\ell!}(\beta(\mu))^\ell
,\qquad
U_{x-1}(1+\beta(\mu))
=\sum_{\ell=0}^{x-1}\frac{U_{x-1}^{(\ell)}(1)}{\ell!}(\beta(\mu))^\ell,
\]
where $T_x(1)=1$, $U_{x-1}(1)=x$, and for $1\leq j\leq x $ and $1\leq k\leq x-1$,
\begin{equation}\label{Tx1 Ux1}
T_x^{(j)}(1)=
    \frac{x^2 P_{j-1}(x)}
{(2j-1)!!}
,\qquad U_{x-1}^{(k)}(1)=
    \frac{xP_{k}(x)}
{(2k+1)!!}~
,\qquad P_{\ell}(x)=\displaystyle\prod_{s=1}^{\ell}(x^2-s^2).
\end{equation}
Write $q(\mu)=q(0)+\mu h(\mu)$, where $q(0)=2\sqrt{2}-3$ and  $h(\mu)$ is analytic in $\mu$. It follows that $(q(\mu))^x$ is analytic in $\mu$ and contains only integer powers
of $\mu$. Moreover, there exists $\rho\in(0,1)$ such that
\begin{equation}\label{qu}
\left|\partial_\mu^r (q(\mu))^x\right|
\lesssim_r \left<x\right>^r\rho^x,
\qquad 0\leq\mu\leq\mu_0,\quad r\in\N.
\end{equation}
Thus \eqref{eq:upper-cheb} becomes
\begin{equation}\label{final expan}
\widetilde R_0^\pm(\mu,x)
=
\sum_{k=0}^{\infty}
\mu^{k-\frac12}\widetilde G_{2k-1}^\pm(x)
+
\sum_{k=0}^{\infty}
\mu^k\widetilde G_{2k}(x),\qquad0<\mu<\mu_0, 
\end{equation} 
where $\widetilde G_{2k-1}^\pm(x)=\frac{\pm i}{2}E_k(x)$, $\widetilde G_{2k}(x)=F_k(x)$, and $E_k(x)$ and $F_k(x)$ are given by
\begin{equation*}
E(\mu,x)=\sum_{\ell=0}^{x}
\frac{T_x^{(\ell)}(1)(\beta(\mu))^\ell}
{(2-\mu)^3\sqrt{4-\mu}\,\ell!}=\sum_{k=0}^\infty E_k(x)\mu^k,\qquad0<\mu<\mu_0
\end{equation*}
and 
\begin{equation*}
F(\mu,x)=\sum_{\ell=0}^{x-1}
\frac{U_{x-1}^{(\ell)}(1)(\beta(\mu))^\ell}
{4(2-\mu)^2\ell!}
-\frac{(q(\mu))^x}{2(2-\mu)^3d(\mu)}=\sum\limits_{k=0}^{\infty}F_k(x)\mu^k,\qquad0<\mu<\mu_0.
\end{equation*}
The formulas \eqref{Tx1 Ux1} and \eqref{qu} above imply
$
E_k(x)=O(\left< x\right>^{2k})$,
while $F_k(x)$ is the sum of a polynomial-growth term of order at most
$2k+1$ and an exponentially decaying term. More generally, for any integer $N\geq-1$, truncating
\eqref{final expan} gives the desired \eqref{Puiseux expan of R0 2}. This completes the proof.
\end{proof}
The following lemma further characterizes the resonance types of $H$ in terms of the invertibility of certain operators, providing a direct approach to computing $\big(M^{\pm}(\mu)\big)^{-1}$ via Neumann series expansions. For brevity, we denote the kernel of $T$ restricted to a space $X$ by ${\rm Ker}T\big|_{X}=\{f\in X:Tf=0\}$.
\begin{lemma}\label{lemma of kernel and operaor}
Let $G_0=G^{\pm}_0$,  $\widetilde{G}_0=\widetilde{G}^{\pm}_0$, $\widetilde{G}_2=\widetilde{G}^{\pm}_2$, and $$G_{-1}=\frac{4}{1\pm i}G^{\pm}_{-1},\quad G_{1}=\frac{32}{-1\pm i}G^{\pm}_{1},\quad G_3=\frac{4\cdot6!}{-1\mp i}G^{\pm}_{3}(n),\quad \widetilde{G}_1=\mp i\widetilde{G}^{\pm}_1.$$
 Let $T_0=U+vG_0v$ and $\widetilde{T}_0=U+\tilde{v}\widetilde{G}_0\tilde{v}$. Then the following statements hold.
\vskip0.1cm
{\rm(1)}$~
{\rm Ker}QvG_{-1}vQ\big|_{Q\ell^2(\Z)}=S_0\ell^2(\Z).$
\vskip0.2cm
{\rm(2)} Denote by $D_0$ the inverse of $QvG_{-1}vQ+S_0$ on $Q\ell^2(\Z)$ and define
 \begin{equation}\label{T1}
T_1=S_1vG_1vS_1+\frac{8}{\|V\|_{\ell^1}}S_1vG_{-1}vPvG_{-1}vS_1+64S_1T_0D_0T_0S_1.
\end{equation}
Then
${\rm Ker}T_1\big|_{S_1\ell^2(\Z)}=S_2\ell^2(\Z).$
\vskip0.2cm
{\rm(3)} Denote by $D_2$ the inverse of $T_1+S_2$ on $S_1\ell^2(\Z)$ and let
\begin{align}\label{T2}
T_2&=\frac{1}{6!}\Big(S_2vG_3vS_2-\frac{8\cdot6!}{\|V\|_{\ell^1}}S_2T^2_0S_2-\frac{6!}{64}S_2vG_1vD_0vG_1vS_2\Big)\notag\\
&\quad+\frac{64}{\|V\|^2_{\ell^1}}\big(S_2T_0vG_{-1}vD_0-\frac{\|V\|_{\ell^1}}{8}S_2vG_1vD_0T_0D_0\big)D_2\\
&\quad\times \big(D_0vG_{-1}vT_0S_2-\frac{\|V\|_{\ell^1}}{8}D_0T_0D_0vG_1vS_2\big).\notag
\end{align}
Then
${\rm Ker}T_2\big|_{S_2\ell^2(\Z)}=S_3\ell^2(\Z).$
\vskip0.2cm
{\rm(4)} Define
\begin{equation}\label{T1tuta}
\widetilde{T}_1=\widetilde{S}_0\tilde{v}\widetilde{G}_1\tilde{v}\widetilde{S}_0+\frac{32}{\|V\|_{\ell^1}}\widetilde{S}_0\widetilde{T}^2_0\widetilde{S}_0.
\end{equation}
Then
${\rm Ker}\widetilde{T}_1\big|_{\widetilde{S}_0\ell^2(\Z)}=\widetilde{S}_1\ell^2(\Z).$
\vskip0.2cm
{\rm(5)} Let $H=\Delta^2+V$ on $\Z$ and $|V(n)|\lesssim\left<n\right>^{-\beta}$ with $\beta>9$. Define $\widetilde{T}_2=\widetilde{S}_1\tilde{v}\widetilde{G}_2\tilde{v}\widetilde{S}_1$. Then
$$S_3\ell^2(\Z)=\{0\},\qquad{\rm Ker}\widetilde{T}_2\big|_{\widetilde{S}_1\ell^2(\Z)}=\{0\}.$$
\end{lemma}
\begin{remark}\label{remark of charac from inverti of operators}
{\rm Using Theorem \ref{relation space and resonance types} and 
$$S_1\ell^2(\Z)={\rm Ker}S_0T_0S_0\big|_{S_0\ell^2(\Z)},\qquad \widetilde{S}_0\ell^2(\Z)={\rm Ker}\widetilde{Q}\widetilde{T}_0\widetilde{Q}\big|_{\widetilde{Q}\ell^2(\Z)},$$
the lemma further yields the following characterizations of the resonance types of $H$:
\vskip0.1cm
(i) 0 is a regular point of $H$ if and only if $S_0T_0S_0$ is invertible on $S_0\ell^2(\Z)$.

(ii) 0 is a first kind resonance of $H$ if and only if $S_0T_0S_0$ is not invertible on $S_0\ell^2(\Z)$ but $T_1$ is invertible on $S_1\ell^2(\Z)$.

(iii) 0 is a second kind resonance of $H$ if and only if $T_1$ is not invertible on $S_1\ell^2(\Z)$.

(iv) 16 is a regular point of $H$ if and only if $\widetilde{Q}\widetilde{T}_0\widetilde{Q}$ is invertible on $\widetilde{Q}\ell^2(\Z)$.

(v) 16 is a resonance of $H$ if and only if $\widetilde{Q}\widetilde{T}_0\widetilde{Q}$ is not invertible on $\widetilde{Q}\ell^2(\Z)$ but $\widetilde{T}_1$ is invertible on $\widetilde{S}_0\ell^2(\Z)$.

(vi) 16 is an eigenvalue of $H$ if and only if $\widetilde{T}_1$ is not invertible on $\widetilde{S}_0\ell^2(\Z)$.
}
\end{remark}
The proof of this lemma will be delayed to the end of this section. Furthermore, the following lemma will be frequently used in the proof.
\begin{lemma}\label{lemm of expa}
{\rm (\cite[Lemma 2.1]{JN01})} Let $\mscH$ be a complex Hilbert space. Let $A$ be a closed operator and $S$ a projection. Suppose $A+S$ has a bounded inverse. Then $A$ has a bounded inverse if and only if
$$a\equiv S-S(A+S)^{-1}S$$
has a bounded inverse in $S\mscH$, and in this case
$$A^{-1}=(A+S)^{-1}+(A+S)^{-1}Sa^{-1}S(A+S)^{-1}.$$
\end{lemma}
With these three lemmas in hand, we now begin the proof of Theorem \ref{Asymptotic expansion theorem}.
\begin{proof}[Proof of Theorem {\rm\ref{Asymptotic expansion theorem}}]{\bf{(i) ${\bm{0}}$ is a regular point of $H$ and {$\bm{\beta>15}$}.}}
Taking $N=3$ in \eqref{Puiseux expan of R0 0}, we have the expansion in $\B(s,-s)$ for $s>\frac{15}{2}$,
\begin{equation*}
R^{\pm}_0(\mu^4)=\mu^{-3}G^{\pm}_{-3}+\mu^{-1}G^{\pm}_{-1}+G^{\pm}_0+\sum\limits_{k=1}^{3}\mu^kG^{\pm}_{k}+\Gamma_4(\mu),\qquad \mu \rightarrow0.
\end{equation*}
Since $\beta>15$ and $M^{\pm}(\mu)=U+vR^{\pm}_0(\mu^4)v$, using $vG^{\pm}_{-3}v=a^{\pm}P$ with $a^{\pm}=\frac{-1\pm i}{4}\|V\|_{\ell^1}$, we obtain 
 \begin{equation*}
 M^{\pm}(\mu)=\frac{a^{\pm}}{\mu^3}\Big(P+\frac{\mu^2}{a^{\pm}}vG^{\pm}_{-1}v+\frac{\mu^3}{a^{\pm}}T_0+\frac{1}{a^{\pm}}\sum\limits_{k=1}^{3}\mu^{k+3}vG^{\pm}_{k}v+\Gamma_7(\mu)\Big)=:\frac{a^{\pm}}{\mu^3}\widetilde{M}^{\pm}(\mu)
 \end{equation*}
 on $\ell^2(\Z)$ as $\mu\rightarrow0$. Then
\begin{equation}\label{Relat of M and M tuta} \left(M^{\pm}(\mu)\right)^{-1}=\frac{\mu^3}{a^{\pm}}\big(\widetilde{M}^{\pm}(\mu)\big)^{-1}.
\end{equation}

{\bf{\underline{Step1.}}} By Lemma \ref{lemm of expa}, $\widetilde{M}^{\pm}(\mu)$ is invertible on $\ell^{2}(\Z)$ if and only if $$ M^{\pm}_1(\mu):=Q-Q\big(\widetilde{M}^{\pm}(\mu)+Q\big)^{-1}Q$$ is invertible on $Q\ell^{2}(\Z)$ and in this case, one has
\begin{align}\label{Relat of M tuta and M1}
\big(\widetilde{M}^{\pm}(\mu)\big)^{-1}=\big(\widetilde{M}^{\pm}(\mu)+Q\big)^{-1}\Big[I+Q\left(M^{\pm}_1(\mu)\right)^{-1}Q\big(\widetilde{M}^{\pm}(\mu)+Q\big)^{-1}\Big].
\end{align}
Using Neumann series expansions, we obtain 
\begin{align}\label{expr of Bkpm}
\big(\widetilde{M}^{\pm}(\mu)+Q\big)^{-1}=I-\sum\limits_{k=1}^{5}\mu^{k+1}B^{\pm}_{k}+\Gamma_7(\mu),\quad \mu\rightarrow0,
\end{align}
where
\begin{itemize}
\item $B^{\pm}_{1}=\frac{1}{a^{\pm}}vG^{\pm}_{-1}v$,\quad$B^{\pm}_{2}=\frac{1}{a^{\pm}}T_0$,\quad $B^{\pm}_{3}=\frac{1}{a^{\pm}}vG^{\pm}_{1}v-\left(\frac{1}{a^{\pm}}vG^{\pm}_{-1}v\right)^2$,
\vskip0.2cm
\item $B^{\pm}_{4}=-\left(\frac{1}{a^{\pm}}\right)^2\left(vG^{\pm}_{-1}vT_0+T_0vG^{\pm}_{-1}v\right)$,
\item
$B^{\pm}_{5}=\frac{1}{a^{\pm}}vG^{\pm}_{3}v-\left(\frac{1}{a^{\pm}}\right)^2\left(vG^{\pm}_{-1}v\cdot vG^{\pm}_{1}v+T^2_0+vG^{\pm}_{1}v \cdot vG^{\pm}_{-1}v\right)+\left(\frac{1}{a^{\pm}}vG^{\pm}_{-1}v\right)^3$.
\end{itemize}
This, together with $Q^2=Q$, and $QB^{\pm}_1Q=b^{\pm}QvG_{-1}vQ$ with $b^{\pm}=\mp i\|V\|^{-1}_{\ell^1}$, gives
\begin{equation*}
M^{\pm}_1(\mu)=b^{\pm}\mu^2\Big(QvG_{-1}vQ+\frac{1}{b^{\pm}}\sum\limits_{k=2}^{5}\mu^{k-1}QB^{\pm}_{k}Q+Q\Gamma_5(\mu)Q\Big)=:b^{\pm}\mu^2\widetilde{M}^{\pm}_1(\mu),
\end{equation*}
Hence,
\begin{equation}\label{Relat of M1 and M1 tuta}
\left(M^{\pm}_1(\mu)\right)^{-1}=\frac{1}{b^{\pm}\mu^2}\big(\widetilde{M}_{1}^{\pm}(\mu)\big)^{-1}.
\end{equation}
\vskip0.15cm
{\bf{\underline{Step2.}}} Since ${\rm Ker}QvG_{-1}vQ\big|_{Q\ell^2(\Z)}=S_0\ell^{2}(\Z)$, Lemma \ref{lemm of expa} implies that $\widetilde{M}_{1}^{\pm}(\mu)$ is invertible on $Q\ell^{2}(\Z)$ if and only if $$M^{\pm}_{2}(\mu):=S_0-S_0\big(\widetilde{M}_{1}^{\pm}(\mu)+S_0\big)^{-1}S_0$$ is invertible on $S_0\ell^{2}(\Z)$. In this case,
\begin{equation}\label{Relat of M1 tuta and M2}
\big(\widetilde{M}_{1}^{\pm}(\mu)\big)^{-1}=\big(\widetilde{M}_{1}^{\pm}(\mu)+S_0\big)^{-1}\Big(I+S_0\big(M^{\pm}_{2}(\mu)\big)^{-1}S_0\big(\widetilde{M}_{1}^{\pm}(\mu)+S_0\big)^{-1}\Big),
\end{equation}
where
\begin{equation}\label{expre of Bpm tuta}
\big(\widetilde{M}_{1}^{\pm}(\mu)+S_0\big)^{-1}=D_0-\sum\limits_{k=1}^{4}\mu^{k}\widetilde{B}_{k}^{\pm}+D_0\Gamma_5(\mu)D_0,\qquad\mu\rightarrow0,
\end{equation}
with $D_0$ being the inverse of $QvG_{-1}vQ+S_0$ on $Q\ell^2(\Z)$~(hence $D_0Q=D_0=QD_0$) and
\begin{itemize}
\item $\widetilde{B}_{1}^{\pm}=\frac{1}{b^{\pm}}D_0B^{\pm}_2D_0$,\quad$\widetilde{B}_{2}^{\pm}=\frac{1}{b^{\pm}}D_0B^{\pm}_3D_0-\left(\frac{1}{b^{\pm}}\right)^2\left(D_0B^{\pm}_2\right)^2D_0$,
\vskip0.2cm
\item
$\widetilde{B}_{3}^{\pm}=\frac{1}{b^{\pm}}D_0B^{\pm}_4D_0-\left(\frac{1}{b^{\pm}}\right)^2\left(D_0B^{\pm}_2D_0B^{\pm}_3D_0+D_0B^{\pm}_3D_0B^{\pm}_2D_0\right)+\left(\frac{1}{b^{\pm}}D_0B^{\pm}_2\right)^3D_0,$
\vskip0.2cm
\item $\widetilde{B}_{4}^{\pm}=\frac{1}{b^{\pm}}D_0B^{\pm}_5D_0-\left(\frac{1}{b^{\pm}}\right)^2\left(D_0B^{\pm}_2D_0B^{\pm}_4D_0+\left(D_0B^{\pm}_3\right)^2D_0+D_0B^{\pm}_4D_0B^{\pm}_2D_0\right)$

$\quad+\left(\frac{1}{b^{\pm}}\right)^3\left(\left(D_0B^{\pm}_2\right)^2D_0B^{\pm}_3D_0+D_0B^{\pm}_2D_0B^{\pm}_3D_0B^{\pm}_2D_0+D_0B^{\pm}_3\left(D_0B^{\pm}_2\right)^2D_0\right)$

$\quad-\left(\frac{1}{b^{\pm}}D_0B^{\pm}_2\right)^4D_0$.
\end{itemize}
Using $S_0D_0=S_0=D_0S_0$ and $S_0\widetilde{B}_{1}^{\pm}S_0=\frac{1}{a^{\pm}_{-1}}S_0T_0S_0$ with $a^{\pm}_{-1}=\frac{1\pm i}{4}$, we obtain
\begin{align*}
M^{\pm}_2(\mu)=\frac{\mu}{a^{\pm}_{-1}}\Big(S_0T_0S_0+a^{\pm}_{-1}\sum\limits_{k=2}^{4}\mu^{k-1}S_0\widetilde{B}_{k}^{\pm}S_0+S_0\Gamma_4(\mu)S_0\Big)=:\frac{\mu}{a^{\pm}_{-1}}\widetilde{M}^{\pm}_2(\mu),
\end{align*}
and thus
\begin{equation}\label{Relat of M2 and M2 tuta}
\left(M^{\pm}_2(\mu)\right)^{-1}=\frac{a^{\pm}_{-1}}{\mu}\big(\widetilde{M}_{2}^{\pm}(\mu)\big)^{-1}.
\end{equation}
Since $0$ is a regular point of $H$, Remark \ref{remark of charac from inverti of operators} implies that $S_0T_0S_0$ is invertible on $S_0\ell^2(\Z)$. Denote its inverse by $D_1$. Then, using the Neumann expansion and $S_0D_1=D_1=D_1S_0$, we obtain
\begin{align}\label{inverse of Mpm2}
\left(M^{\pm}_2(\mu)\right)^{-1}=\frac{a^{\pm}_{-1}}{\mu}D_1+\sum\limits_{k=0}^{2}\widetilde{D}^{\pm}_k\mu^k+D_1\Gamma_3(\mu)D_1,\quad\mu\rightarrow0,
\end{align}
where
\begin{itemize}
\item
$\widetilde{D}^{\pm}_0=-\left(a^{\pm}_{-1}\right)^2D_1\widetilde{B}_{2}^{\pm}D_1$,\quad $\widetilde{D}^{\pm}_{1}=-\big(a^{\pm}_{-1}\big)^2\big(D_1\widetilde{B}_{3}^{\pm}D_1-a^{\pm}_{-1}\big(D_1\widetilde{B}_{2}^{\pm}\big)^2D_1\big)$,
\vskip0.2cm
\item $\widetilde{D}^{\pm}_2=-\big(a^{\pm}_{-1}\big)^2D_1\widetilde{B}_{4}^{\pm}D_1+\big(a^{\pm}_{-1}\big)^3\big(D_1\widetilde{B}_{2}^{\pm}D_1\widetilde{B}_{3}^{\pm}D_1+D_1\widetilde{B}_{3}^{\pm}D_1\widetilde{B}_{2}^{\pm}D_1\big)$
    $-\big(a^{\pm}_{-1}\big)^4\big(D_1\widetilde{B}_{2}^{\pm}\big)^3D_1$.
\end{itemize}
Taking \eqref{inverse of Mpm2} into \eqref{Relat of M1 tuta and M2},\eqref{Relat of M1 and M1 tuta},\eqref{Relat of M tuta and M1} and \eqref{Relat of M and M tuta} in order, the desired \eqref{asy expan of regular 0} is obtained.
\vskip0.2cm
{\bf{(ii) ${\bm{0}}$ is a first kind resonance of $H$ and {$\bm{\beta>19}$}.}} For this case, we take $N=5$ in \eqref{Puiseux expan of R0 0} and repeat {\bf{\underline{Step1}}} and  {\bf{\underline{Step2}}} in {\bf{(i)}}. However, since $S_0T_0S_0$ is not invertible on $S_0\ell^2(\Z)$, we need to further analyze the $\widetilde{M}^{\pm}_2(\mu)$ in {\bf{\underline{Step2}}} of {\bf{(i)}}. Since ${\rm Ker}S_0T_0S_0\big|_{S_0\ell^2(\Z)}=S_1\ell^2(\Z)$, Lemma \ref{lemm of expa} yields the following further reduction.
\vskip0.1cm
{\bf{\underline{Step3.}}} \ $\widetilde{M}_{2}^{\pm}(\mu)$ is invertible on $S_0\ell^{2}(\Z)$ if and only if $$ M^{\pm}_{3}(\mu):=S_1-S_1\big(\widetilde{M}_{2}^{\pm}(\mu)+S_1\big)^{-1}S_1$$ is invertible on $S_1\ell^{2}(\Z)$. In this case,
\begin{equation}\label{Relat of M2 tuta and M3}
\big(\widetilde{M}_{2}^{\pm}(\mu)\big)^{-1}=\big(\widetilde{M}_{2}^{\pm}(\mu)+S_1\big)^{-1}\Big(I+S_1\big(M^{\pm}_{3}(\mu)\big)^{-1}S_1\big(\widetilde{M}_{2}^{\pm}(\mu)+S_1\big)^{-1}\Big).
\end{equation}
Let $D_1$ be the inverse of $S_0T_0S_0+S_1$ on $S_0\ell^2(\Z)$.~Then $D_1S_0=D_1=S_0D_1$ and
\begin{equation}\label{expre of Bpm ttuta}
\big(\widetilde{M}_{2}^{\pm}(\mu)+S_1\big)^{-1}=D_1-\sum\limits_{k=1}^{5}\mu^{k}\tilde{\tilde{B}}^{\pm}_{k}+D_1\Gamma_6(\mu)D_1,\quad\mu\rightarrow0,
\end{equation}
where
\begin{itemize}
\item $\tilde{\tilde{B}}^{\pm}_{1}=a^{\pm}_{-1}D_1\widetilde{B}^{\pm}_2D_1$,\quad$\tilde{\tilde{B}}^{\pm}_{2}=a^{\pm}_{-1}D_1\widetilde{B}^{\pm}_3D_1-\big(a^{\pm}_{-1}D_1\widetilde{B}^{\pm}_2\big)^2D_1$,
\vskip0.2cm
\item $\tilde{\tilde{B}}^{\pm}_{3}=a^{\pm}_{-1}D_1\widetilde{B}^{\pm}_4D_1-{(a^{\pm}_{-1})}^2\big(D_1\widetilde{B}^{\pm}_2D_1\widetilde{B}^{\pm}_3D_1+D_1\widetilde{B}^{\pm}_3D_1\widetilde{B}^{\pm}_2D_1\big)+\big(a^{\pm}_{-1}D_1\widetilde{B}^{\pm}_2\big)^3D_1$.
\end{itemize}
Using $S_1D_1=S_1=D_1S_1$, $S_1D_0=S_1=D_0S_1$ and $S_1\tilde{\tilde{B}}^{\pm}_{1}S_1=a^{\pm}_{1}T_1$ with $a^{\pm}_{1}=\frac{-1\pm i}{32}$, we obtain
\begin{align*}
M^{\pm}_3(\mu)=a^{\pm}_{1}\mu\Big(T_1+\frac{1}{a^{\pm}_{1}}\sum\limits_{k=2}^{5}\mu^{k-1}S_1\tilde{\tilde{B}}^{\pm}_{k}S_1+S_1\Gamma_5(\mu)S_1\Big)=:a^{\pm}_{1}\mu\widetilde{M}^{\pm}_3(\mu),
\end{align*}
and then
\begin{equation}\label{Relat of M3 and M3 tuta}
\left(M^{\pm}_3(\mu)\right)^{-1}=\frac{1}{a^{\pm}_{1}\mu}\left(\widetilde{M}_{3}^{\pm}(\mu)\right)^{-1}.
\end{equation}
By Remark \ref{remark of charac from inverti of operators}, since $T_1$ is invertible on $S_1\ell^2(\Z)$, using Von-Neumann expansion, we derive
\begin{align}\label{inverse of Mtuta3}
\big(\widetilde{M}_{3}^{\pm}(\mu)\big)^{-1}
&=D_2-\sum\limits_{k=1}^{4}\mu^k\tilde{\tilde{\tilde{B}}}^{\pm}_{k}+D_2\Gamma_5(\mu)D_2,\quad \mu\rightarrow0,
\end{align}
where $D_2$ is the inverse of $T_1$ on $S_1\ell^2(\Z)$~(then $D_2S_1=D_2=S_1D_2$) and
\begin{equation}\label{expre of Btttpm tuta} \tilde{\tilde{\tilde{B}}}^{\pm}_{1}=\frac{1}{a^{\pm}_{1}}D_2\tilde{\tilde{B}}^{\pm}_2D_2,\qquad\tilde{\tilde{\tilde{B}}}^{\pm}_{2}=\frac{1}{a^{\pm}_{1}}D_2\tilde{\tilde{B}}^{\pm}_3D_2-\Big(\frac{1}{a^{\pm}_{1}}D_2\tilde{\tilde{B}}^{\pm}_2\Big)^2D_2.
\end{equation}
Then \eqref{asy expan of 1st 0} follows by successively combining \eqref{inverse of Mtuta3},\eqref{Relat of M3 and M3 tuta},\eqref{Relat of M2 tuta and M3},\eqref{Relat of M2 and M2 tuta},\eqref{Relat of M1 tuta and M2},\eqref{Relat of M1 and M1 tuta},\eqref{Relat of M tuta and M1} and \eqref{Relat of M and M tuta}.
\vskip0.2cm
{\bf{(iii) ${\bm{0}}$ is a second kind resonance of $H$ and {$\bm{\beta>27}$}.}} Take $N=9$ in \eqref{Puiseux expan of R0 0} and repeat {\bf{\underline{Step1}}} to {\bf{\underline{Step3}}} in {\bf{(i)}} and {\bf{(ii)}}. Since $T_1$ is not invertible on $S_1\ell^2(\Z)$, we need to further analyze the $\widetilde{M}^{\pm}_3(\mu)$ in {\bf{\underline{Step3}}} of {\bf{(ii)}}. Since ${\rm Ker}T_1\big|_{S_1\ell^2(\Z)}=S_2\ell^2(\Z)$, Lemma \ref{lemm of expa} yields the following further reduction.
\vskip0.15cm
{\bf{\underline{Step4:}}} \ $\widetilde{M}_{3}^{\pm}(\mu)$ is invertible on $S_1\ell^{2}(\Z)$ if and only if $$ M^{\pm}_{4}(\mu):=S_2-S_2\big(\widetilde{M}_{3}^{\pm}(\mu)+S_2\big)^{-1}S_2$$ is invertible on $S_2\ell^{2}(\Z)$. In this case,
\begin{equation}\label{Relat of M3 tuta and M4}
\big(\widetilde{M}_{3}^{\pm}(\mu)\big)^{-1}=\big(\widetilde{M}_{3}^{\pm}(\mu)+S_2\big)^{-1}\Big(I+S_2\big(M^{\pm}_{4}(\mu)\big)^{-1}S_2\big(\widetilde{M}_{3}^{\pm}(\mu)+S_2\big)^{-1}\Big).
\end{equation}
Let $D_2$ be the inverse of $T_1+S_2$ on $S_1\ell^2(\Z)$.
Then 
\begin{equation*}
\left(\widetilde{M}_{3}^{\pm}(\mu)+S_2\right)^{-1}=D_2-\sum\limits_{k=1}^{8}\mu^k\tilde{\tilde{\tilde{B}}}^{\pm}_{k}+D_2\Gamma_9(\mu)D_2, \quad \mu\rightarrow0.
\end{equation*}
We claim that  \begin{equation}\label{S2Bttt1S2,S2Bttt2S2}
S_2\tilde{\tilde{\tilde{B}}}^{\pm}_{1}S_2=0, \qquad S_2\tilde{\tilde{\tilde{B}}}^{\pm}_{2}S_2=d^{\pm}T_2,\qquad d^{\pm}=\pm8i.
\end{equation}
Assuming this relation for the moment, and using $S_2D_2=S_2=D_2S_2$, we obtain
\begin{equation}\label{rel betw M4 and wideM4}
M^{\pm}_4(\mu)=d^{\pm}\mu^2\Big(T_2+\frac{1}{d^{\pm}}\sum\limits_{k=3}^{8}\mu^{k-2}S_2\tilde{\tilde{\tilde{B}}}^{\pm}_{k}S_2+S_2\Gamma_7(\mu)S_2\Big)=:d^{\pm}\mu^2\widetilde{M}^{\pm}_4(\mu).
\end{equation}
By Remark \ref{remark of charac from inverti of operators}, the invertibility of $T_2$ on $S_2\ell^2(\Z)$ implies that $\widetilde{M}^{\pm}_4(\mu)$ is invertible on $S_2\ell^2(\Z)$, and
\begin{align}\label{inverse of Mtuta4}
\big(\widetilde{M}_{4}^{\pm}(\mu)\big)^{-1}
&=D_3+\sum\limits_{k=1}^{6}\mu^k\hat{B}^{\pm}_{k}+D_3\Gamma_7(\mu)D_3,\qquad D_3=T^{-1}_2,\qquad \mu\rightarrow0.
\end{align}
Thus, \eqref{asy expan 2nd 0} follows by successively combining \eqref{inverse of Mtuta4}, \eqref{rel betw M4 and wideM4}, \eqref{Relat of M3 tuta and M4}, \eqref{Relat of M3 and M3 tuta}, \eqref{Relat of M2 tuta and M3}, \eqref{Relat of M2 and M2 tuta}, \eqref{Relat of M1 tuta and M2}, \eqref{Relat of M1 and M1 tuta}, \eqref{Relat of M tuta and M1}, and \eqref{Relat of M and M tuta}.

Now, we verify \eqref{S2Bttt1S2,S2Bttt2S2}. The first identity follows from \eqref{expre of Btttpm tuta}, \eqref{expre of Bpm ttuta}, \eqref{expre of Bpm tuta}, \eqref{expr of Bkpm}, together with the following relations:
\begin{align*}
S_2D_j&=S_2=D_jS_2~(j=0,1,2),\quad D_0Q=D_0=QD_0,\quad  S_0D_1=D_1=D_1S_0,\\
S_2B^{\pm}_4S_2&=0,\quad S_2B^{\pm}_2Q=0=QB^{\pm}_2S_2,\quad S_2B^{\pm}_3S_0=0=S_0B^{\pm}_3S_2.
\end{align*}
For the second, it follows from \eqref{expre of Btttpm tuta} that 
\begin{align*}
S_2\tilde{\tilde{\tilde{B}}}^{\pm}_{2}S_2=\frac{1}{a^{\pm}_1}S_2\tilde{\tilde{B}}^{\pm}_3S_2-\Big(\frac{1}{a^{\pm}_{1}}\Big)^2S_2\tilde{\tilde{B}}^{\pm}_2D_2\tilde{\tilde{B}}^{\pm}_2S_2.
\end{align*}
Since $S_2\widetilde{B}^{\pm}_2D_1=0=D_1\widetilde{B}^{\pm}_2S_2$, $S_2B^{\pm}_2Q=0=QB^{\pm}_2S_2$, and $D_1D_2=D_2=D_2D_1$, we can compute
\begin{align}
\begin{split}
\frac{1}{a^{\pm}_1}S_2\tilde{\tilde{B}}^{\pm}_3S_2&=\frac{1}{a^{\pm}_1}S_2vG^{\pm}_3vS_2-\frac{1}{a^{\pm}_1a^{\pm}}S_2T^2_0S_2-\frac{1}{a^{\pm}_1a^{\pm}_{-1}}S_2vG^{\pm}_1vD_0vG^{\pm}_1vS_2,\\
\Big(\frac{1}{a^{\pm}_{1}}\Big)^2S_2\tilde{\tilde{B}}^{\pm}_2D_2\tilde{\tilde{B}}^{\pm}_2S_2&=\frac{1}{{\big(a^{\pm}_1\big)}^2}\Big(\frac{1}{a^{\pm}}S_2T_0vG^{\pm}_{-1}vD_0+\frac{1}{a^{\pm}_{-1}}S_2vG^{\pm}_1vD_0T_0D_0\Big)D_2\\
&\quad\times\Big(\frac{1}{a^{\pm}}D_0vG^{\pm}_{-1}vT_0S_2+\frac{1}{a^{\pm}_{-1}}D_0T_0D_0vG^{\pm}_1vS_2\Big),
\end{split}
\end{align}
which gives that $S_2\tilde{\tilde{\tilde{B}}}^{\pm}_{2}S_2=d^{\pm}T_2$ and this proves \eqref{S2Bttt1S2,S2Bttt2S2}.
\vskip0.2cm
{\bf{(iv) ${\bm{16}}$ is a regular point of $H$ and {$\bm{\beta>7}$}.}} Take $N=1$ in \eqref{Puiseux expan of R0 2}. Then as $\mu\rightarrow0$,
\begin{equation*}
M^{\pm}(2-\mu)=U+\tilde{v}JR^{\pm}_0((2-\mu)^4)J\tilde{v}
=\frac{d^{\pm}_{-1}}{\mu^{\frac{1}{2}}}\Big(\widetilde{P}+\frac{\mu^{\frac{1}{2}}}{d^{\pm}_{-1}}\widetilde{T}_0+\frac{\mu}{d^{\pm}_{-1}}\tilde{v}\widetilde{G}^{\pm}_{1}\tilde{v}+\Gamma_{\frac{3}{2}}(\mu)\Big)=:\frac{d^{\pm}_{-1}}{\mu^{\frac{1}{2}}}\widetilde{M}^{\pm}(\mu),
\end{equation*}
where we used the fact that $\tilde{v}\widetilde{G}^{\pm}_{-1}\tilde{v}=d^{\pm}_{-1}\widetilde{P}$ with $d^{\pm}_{-1}=\frac{\pm i}{32}\|V\|_{\ell^1}$. Then
\begin{align}\label{Relat of M and M tuta at 2} \left(M^{\pm}(2-\mu)\right)^{-1}=\frac{\mu^{\frac{1}{2}}}{d^{\pm}_{-1}}\big(\widetilde{M}^{\pm}(\mu)\big)^{-1}.
\end{align}

{\bf{\underline{Step1:}}} By Lemma \ref{lemm of expa}, $\widetilde{M}^{\pm}(\mu)$ is invertible on $\ell^{2}(\Z)$ if and only if $ M^{\pm}_1(\mu):=\widetilde{Q}-\widetilde{Q}\big(\widetilde{M}^{\pm}(\mu)+\widetilde{Q}\big)^{-1}\widetilde{Q}$ is invertible on $\widetilde{Q}\ell^{2}(\Z)$ and in this case, one has
\begin{align}\label{Relat of M tuta and M1 at 2}
\big(\widetilde{M}^{\pm}(\mu)\big)^{-1}=\big(\widetilde{M}^{\pm}(\mu)+\widetilde{Q}\big)^{-1}\Big[I+\widetilde{Q}\left(M^{\pm}_1(\mu)\right)^{-1}\widetilde{Q}\big(\widetilde{M}^{\pm}(\mu)+\widetilde{Q}\big)^{-1}\Big].
\end{align}
As $\mu\rightarrow0$, the Neumann series yields
\begin{equation}\label{expr of Ckpm}
\big(\widetilde{M}^{\pm}(\mu)+\widetilde{Q}\big)^{-1}=I-\sum\limits_{k=1}^{2}\mu^{\frac{k}{2}}C^{\pm}_{k}+\Gamma_{\frac{3}{2}}(\mu),\qquad C^{\pm}_{1}=\frac{1}{d^{\pm}_{-1}}\widetilde{T}_0,\quad C^{\pm}_{2}=\frac{1}{d^{\pm}_{-1}}\tilde{v}\widetilde{G}^{\pm}_1\tilde{v}-\Big(\frac{1}{d^{\pm}_{-1}}\widetilde{T}_0\Big)^2.
\end{equation}
Therefore,
\begin{equation}\label{Relat of M1 and M1 tuta at 2}
M^{\pm}_1(\mu)=\frac{\mu^{\frac{1}{2}}}{d^{\pm}_{-1}}\Big(\widetilde{Q}\widetilde{T}_0\widetilde{Q}+d^{\pm}_{-1}\mu^{\frac{1}{2}}\widetilde{Q}C^{\pm}_{2}\widetilde{Q}+\widetilde{Q}\Gamma_{1}(\mu)\widetilde{Q}\Big)=:\frac{\mu^{\frac{1}{2}}}{d^{\pm}_{-1}}\widetilde{M}^{\pm}_1(\mu).
\end{equation}
By Remark \ref{remark of charac from inverti of operators}, $\widetilde{Q}\widetilde{T}_0\widetilde{Q}$ is invertible on $\widetilde{Q}\ell^{2}(\Z)$. Denote its inverse by $E_0$. Then  $E_0\widetilde{Q}=E_0=\widetilde{Q}E_0$ and the Neumann series yields
\begin{equation}\label{inverse of M1 2}
\left( M^{\pm}_1(\mu)\right)^{-1}=\frac{d^{\pm}_{-1}}{\mu^{\frac{1}{2}}}E_0-\left(d^{\pm}_{-1}\right)^2E_0C^{\pm}_2E_0+E_0\Gamma_{\frac{1}{2}}(\mu)E_0,\quad\mu\rightarrow0.
\end{equation}
Then, combining \eqref{inverse of M1 2}, \eqref{Relat of M tuta and M1 at 2}, and \eqref{Relat of M and M tuta at 2}, we obtain the desired expansion \eqref{asy expan regular 2}.
\vskip0.1cm
{\bf{(v) ${\bm{16}}$ is a resonance of $H$ and {$\bm{\beta>11}$}.}} Take $N=3$ in \eqref{Puiseux expan of R0 2} and repeat {\bf{\underline{Step1}}} in {\bf{(iv)}}. Since ${\rm Ker}\widetilde{Q}\widetilde{T}_0\widetilde{Q}\big|_{\widetilde{Q}\ell^2(\Z)}=\widetilde{S}_0\ell^2(\Z)$, Lemma \ref{lemm of expa} yields the following further reduction.
\vskip0.15cm
{\bf{\underline{Step2:}}}\ $\widetilde{M}_{1}^{\pm}(\mu)$ is invertible on $\widetilde{Q}\ell^{2}(\Z)$ if and only if $ M^{\pm}_{2}(\mu):=\widetilde{S}_0-\widetilde{S}_0\big(\widetilde{M}_{1}^{\pm}(\mu)+\widetilde{S}_0\big)^{-1}\widetilde{S}_0$ is invertible on $\widetilde{S}_0\ell^{2}(\Z)$. In this case,
\begin{equation}\label{Relat of M1 tuta and M2 at 2}
\big(\widetilde{M}_{1}^{\pm}(\mu)\big)^{-1}=\big(\widetilde{M}_{1}^{\pm}(\mu)+\widetilde{S}_0\big)^{-1}\Big(I+\widetilde{S}_0\big(M^{\pm}_{2}(\mu)\big)^{-1}\widetilde{S}_0\big(\widetilde{M}_{1}^{\pm}(\mu)+\widetilde{S}_0\big)^{-1}\Big),
\end{equation}
where
\begin{align*}
\big(\widetilde{M}_{1}^{\pm}(\mu)+\widetilde{S}_0\big)^{-1}=E_0-\sum\limits_{k=1}^{3}\mu^{\frac{k}{2}}\widetilde{C}^{\pm}_{k}+E_0\Gamma_{2}(\mu)E_0,\qquad E_0=(\widetilde{Q}\widetilde{T}_0\widetilde{Q}+\widetilde{S}_0)^{-1},
\end{align*}
with 
\begin{align}\label{Ctutak}
\begin{split}
\widetilde{C}^{\pm}_{1}&=d^{\pm}_{-1}E_0C^{\pm}_{2}E_0,\quad \widetilde{C}^{\pm}_{2}=d^{\pm}_{-1}E_0C^{\pm}_{3}E_0-\big(d^{\pm}_{-1}\big)^2\big(E_0C^{\pm}_2\big)^2E_0,\\
C^{\pm}_{3}&=\frac{1}{d^{\pm}_{-1}}\tilde{v}\widetilde{G}^{\pm}_2\tilde{v}-\Big(\frac{1}{d^{\pm}_{-1}}\Big)^2\Big(\widetilde{T}_0\tilde{v}\widetilde{G}^{\pm}_1\tilde{v}+\tilde{v}\widetilde{G}^{\pm}_1\tilde{v}\widetilde{T}_0\Big)+\Big(\frac{1}{d^{\pm}_{-1}}\widetilde{T}_0\Big)^3.
\end{split}
\end{align}
Using $\widetilde{S}_0E_0=\widetilde{S}_0=E_0\widetilde{S}_0$, we obtain
\begin{equation}\label{Relat of M2 and M2 tuta at 2}
M^{\pm}_2(\mu)=d^{\pm}_1\mu^{\frac{1}{2}}\Big(\widetilde{T}_1+\frac{1}{d^{\pm}_{1}}\sum\limits_{k=2}^{3}\mu^{\frac{k-1}{2}}\widetilde{S}_0\widetilde{C}_{k}^{\pm}\widetilde{S}_0+\widetilde{S}_0\Gamma_{\frac{3}{2}}(\mu)\widetilde{S}_0\Big)=:d^{\pm}_1\mu^{\frac{1}{2}}\widetilde{M}^{\pm}_2(\mu),\quad d^{\pm}_1=\pm i.
\end{equation}
By Remark \ref{remark of charac from inverti of operators}, $\widetilde{T}_1$ is invertible on $\widetilde{S}_0\ell^2(\Z)$. Using Von-Neumann expansion yields
\begin{equation}\label{inverse of Mtuta2}
\big(\widetilde{M}_{2}^{\pm}(\mu)\big)^{-1}=E_1-\mu^{\frac{1}{2}}\tilde{\tilde{C}}^{\pm}_{1}-\mu \tilde{\tilde{C}}^{\pm}_2+E_1\Gamma_{\frac{3}{2}}(\mu)E_1,\qquad \tilde{\tilde{C}}^{\pm}_{1}=\frac{1}{d^{\pm}_1}E_1\widetilde{C}^{\pm}_2E_1,\quad \mu\rightarrow0,
\end{equation}
where $E_1$ is the inverse of $\widetilde{T}_1$ on $\widetilde{S}_0\ell^2(\Z)$.
Combining \eqref{inverse of Mtuta2}, \eqref{Relat of M2 and M2 tuta at 2}, \eqref{Relat of M1 tuta and M2 at 2}, \eqref{Relat of M1 and M1 tuta at 2}, \eqref{Relat of M tuta and M1 at 2} and \eqref{Relat of M and M tuta at 2}, \eqref{asy expan resonance 2} is established. 
\vskip0.2cm
{\bf{(vi) ${\bm{16}}$ is an eigenvalue of $H$ and {$\bm{\beta>15}$}.}} Take $N=5$ in \eqref{Puiseux expan of R0 2} and repeat {\bf{\underline{Step1}}} in {\bf{(iv)}} and {\bf{\underline{Step2}}} in {\bf{(v)}}. Since ${\rm Ker}\widetilde{T}_1\big|_{\widetilde{S}_0\ell^2(\Z)}=\widetilde{S}_1\ell^2(\Z)$, Lemma \ref{lemm of expa} yields the following further reduction.
\vskip0.15cm
{\bf{\underline{Step3:}}}\ $\widetilde{M}_{2}^{\pm}(\mu)$ is invertible on $\widetilde{S}_0\ell^{2}(\Z)\Leftrightarrow M^{\pm}_{3}(\mu):=\widetilde{S}_1-\widetilde{S}_1\big(\widetilde{M}_{2}^{\pm}(\mu)+\widetilde{S}_1\big)^{-1}\widetilde{S}_1$ is invertible on $\widetilde{S}_1\ell^{2}(\Z)$. In this case,
\begin{equation}\label{Relat of M2 tuta and M3 at 2}
\big(\widetilde{M}_{2}^{\pm}(\mu)\big)^{-1}=\big(\widetilde{M}_{2}^{\pm}(\mu)+\widetilde{S}_1\big)^{-1}\Big(I+\widetilde{S}_1\big(M^{\pm}_{3}(\mu)\big)^{-1}\widetilde{S}_1\big(\widetilde{M}_{2}^{\pm}(\mu)+\widetilde{S}_1\big)^{-1}\Big),
\end{equation}
where
\begin{align*}
\big(\widetilde{M}_{2}^{\pm}(\mu)+\widetilde{S}_1\big)^{-1}=E_1-\sum\limits_{k=1}^{4}\mu^{\frac{k}{2}}\tilde{\tilde{C}}^{\pm}_{k}+E_1\Gamma_{5/2}(\mu)E_1,\quad \mu\rightarrow0,
\end{align*}
with $E_1=(\widetilde{T}_1+\widetilde{S}_1)^{-1}$ on $\widetilde{S}_0\ell^2(\Z)$ and $\tilde{\tilde{C}}^{\pm}_{1}$ defined in \eqref{inverse of Mtuta2}. Using $\widetilde{T}_0\widetilde{S}_1=0=\widetilde{Q}C^{\pm}_2\widetilde{S}_1$ and $E_1\widetilde{S}_1=\widetilde{S}_1=\widetilde{S}_1E_1$, we have
\begin{equation*}
M^{\pm}_3(\mu)=\frac{1}{d^{\pm}_1}\mu^{\frac{1}{2}}\Big(\widetilde{T}_2+{d^{\pm}_1}\sum\limits_{k=2}^{4}\mu^{\frac{k-1}{2}}\widetilde{S}_1\tilde{\tilde{C}}^{\pm}_{k}\widetilde{S}_1+\widetilde{S}_1\Gamma_{2}(\mu)\widetilde{S}_1\Big)=:\frac{\mu^{\frac{1}{2}}}{d^{\pm}_1}\widetilde{M}^{\pm}_3(\mu).
\end{equation*}
Since $\widetilde{T}_2$ is invertible on $\widetilde{S}_1\ell^2(\Z)$, it follows that $\widetilde{M}^{\pm}_3(\mu)$ is invertible on $\widetilde{S}_1\ell^2(\Z)$ for sufficiently small $\mu>0$. Applying the Neumann series and successively combining \eqref{Relat of M2 tuta and M3 at 2}, \eqref{Relat of M2 and M2 tuta at 2}, \eqref{Relat of M1 tuta and M2 at 2}, \eqref{Relat of M1 and M1 tuta at 2}, \eqref{Relat of M tuta and M1 at 2}, and \eqref{Relat of M and M tuta at 2}, we obtain the desired expansion \eqref{asy expan eigenvalue 2}. This completes the proof.
\end{proof}

Finally, we prove Lemma \ref{lemma of kernel and operaor}.
\begin{proof}[Proof of Lemma {\ref{lemma of kernel and operaor}}]
{\bf{\underline{(1)}}} For any $f\in Q\ell^{2}(\Z)$, since $\left<f,v\right>=0$ and $G_{-1}(n,m)=\frac{1}{8}-\frac{1}{2}|n-m|^2,$
a direct calculation yields
\begin{equation}\label{exp of QvGvQ}
QvG_{-1}vQf=\left<f,v_1\right>Q(v_1).
\end{equation}
Since $Q(v_1)\not\equiv0$, we obtain the desired conclusion. 
\vskip0.1cm
{\bf{\underline{(2)}}} The key point is to show that, for any $f\in S_1\ell^2(\Z)$,
\begin{align}\label{exp of T1f}
T_1f=4\big<f,v_2\big>S_1v_2+64\frac{\big<T_0f,v'\big>}{\|v'\|^4_{\ell^2}}S_1T_0v',\qquad v'=Q(v_1).
\end{align}
Once this is established, the inclusion ``$\supseteq$'' is immediate, while ``$\subseteq$'' follows from
$$0=\big<T_1f,f\big>=4\left|\big<f,v_2\big>\right|^2+64\frac{\left|\big<T_0f,v'\big>\right|^2}{\|v'\|^4_{\ell^2}}.$$
Indeed, using $\big<S_1f,v_k\big>=0$, $S_1v_k=0$ for $k=0,1$, and the kernels of $G_{-1}$ and $G_1$, we obtain
\begin{equation}\label{S1vG1vS1}
S_1vG_1vS_1f=2\big<f,v_2\big>S_1v_2,\quad S_1vG_{-1}vPvG_{-1}vS_1f=\frac{\|V\|_{\ell^1}}{4}\big<f,v_2\big>S_1v_2.
\end{equation}
Moreover, by $S_0v'=0$, $\big<v,v'\big>=0$, and \eqref{exp of QvGvQ}, we have
$(QvG_{-1}vQ+S_0)v'=\big<v',v'\big>v'$, which yields
\begin{equation}\label{D0v}
D_0v'=\frac{v'}{\|v'\|^2_{\ell^2}}.
\end{equation}
This, together with $D_0P=0$, $S_0T_0f=0$, and $(Q-S_0)\ell^2(\Z)={\rm Span}\{v'\}$ gives
\begin{equation}\label{S1T0D0S1f}
S_1T_0D_0T_0S_1f=S_1T_0D_0(P+Q-S_0+S_0)T_0f=S_1T_0D_0(Q-S_0)T_0f=\frac{\big<T_0f,v'\big>}{\|v'\|^4_{\ell^2}}S_1T_0v'.
\end{equation}
Combining \eqref{S1vG1vS1} and \eqref{S1T0D0S1f} with \eqref{T1} yields \eqref{exp of T1f}.
\vskip0.1cm
{\bf{\underline{(3)}}} First of all, we claim that, for any $f\in S_2\ell^2(\Z)$,
\begin{align}\label{exp of T2f}
T_2f&=-\frac{1}{6!}\Big(40\big<f,v_3\big>S_2(v_3)+\frac{8\cdot6!}{\|V\|^2_{\ell^1}}\big<T_0f,v\big>S_2T_0v\Big)\notag\\
&\quad+\frac{64}{\|V\|^2_{\ell^1}}\left[\left(\frac{1}{4}\big<T_0f,v\big>\big<D_2S_1v_2,S_1v_2\big>-\frac{\|V\|_{\ell^1}}{12}\frac{\big<f,v_3\big>}{\|v'\|^2_{\ell^2}}\big<D_2S_1T_0v',S_1v_2\big>\right)S_2T_0v\right.\\
&\qquad\qquad\quad-\left.\left(\frac{\|V\|_{\ell^1}}{12}\frac{\big<T_0f,v\big>}{\|v'\|^2_{\ell^2}}\big<D_2S_1v_2,S_1T_0v'\big>-\frac{\|V\|^2_{\ell^1}}{36}\frac{\big<f,v_3\big>}{\|v'\|^4_{\ell^2}}\big<D_2S_1T_0v',S_1T_0v'\big>\right)S_2(v_3)\right].\notag
\end{align}
Assuming \eqref{exp of T2f} for the moment, the inclusion ``$\supseteq"$ is immediate. To prove ``$\subseteq$", let
\begin{align*}
X&:=\big<f,v_3\big>=X_1+iX_2,\qquad Y:=\big<T_0f,v\big>=Y_1+iY_2, \qquad X_1,X_2,Y_1,Y_2\in\R,\\
g_1&=S_1v_2,\qquad g_2=S_1T_0v',\qquad h_j=D_2g_j~(j=1,2).
\end{align*}
Then
\begin{align}\label{T2ff}
0=\big<T_2f,f\big>=a(X^2_1+X^2_2)+b(Y^2_1+Y^2_2)-dc_1(X_1Y_1+X_2Y_2)+dc_2(X_2Y_1-X_1Y_2),
\end{align}
where $c_1=\operatorname{Re}\big<h_2,g_1\big>$, $c_2=\operatorname{Im}\big<h_2,g_1\big>$, and $a,b,d$ are defined by
\begin{equation}\label{a,b,c,d}
a=\frac{1}{9}\Big(\frac{16\big<h_2,g_2\big>}{\|v'\|^4_{\ell^2}}-\frac{1}{2}\Big),\quad b=\frac{8}{\|V\|^2_{\ell^1}}\big(2\big<h_1,g_1\big>-1\big),\quad
d=\frac{32}{3\|V\|_{\ell^1}\|v'\|^2_{\ell^2}}.
\end{equation}
Consider the real quadratic form
$$F(x_1,x_2,x_3,x_4):=a(x^2_1+x^2_2)+b(x^2_3+x^2_4)-dc_1(x_1x_3+x_2x_4)+dc_2(x_2x_3-x_1x_4)$$
with coefficients $a,b,d,c_1,c_2$ defined in \eqref{a,b,c,d}. We claim that
\begin{equation}\label{4ab}
a<0,\qquad 4ab-d^2(c^2_1+c^2_2)>0.
\end{equation}
It follows that $F$ is negative definite. Hence, \eqref{T2ff} implies $X=Y=0$, proving the inclusion ``$\subseteq$". To verify \eqref{4ab}, since $h_j\in S_1\ell^2(\Z)$, \eqref{exp of T1f} gives
\begin{equation*}
g_j=D^{-1}_2h_j=(T_1+S_2)h_j=4\big<h_j,g_1\big>g_1+64\frac{\big<h_j,g_2\big>}{\|v'\|^4_{\ell^2}}g_2+S_2h_j,\qquad j=1,2.
\end{equation*}
It follows that $$\big<g_1,h_1\big>\leq\frac{1}{4},\qquad \big<g_2,h_2\big>\leq\frac{1}{64}\|v'\|^4_{\ell^2},$$
and both inequalities are strict if $\big<h_2,g_1\big>\neq0$.
Hence, $a<0$ and $b<0$. If $\big<h_2,g_1\big>=0$, then $d^2(c^2_1+c^2_2)-4ab<0$. If $\big<h_2,g_1\big>\neq0$, using 
$$\big|\big<g_2,h_1\big>\big|^2=\big|\big<g_2,D_2g_1\big>\big|^2\leq\big<g_2,D_2g_2\big>\big<g_1,D_2g_1\big>=\big<h_2,g_2\big>\big<h_1,g_1\big>,$$
we obtain
\begin{align*}
d^2(c^2_1+c^2_2)-4ab&=\frac{32^2}{9\|V\|^2_{\ell^1}}\Big(\frac{\big|\big<g_2,h_1\big>\big|^2}{\|v'\|^4_{\ell^2}}-\frac{\big<h_2,g_2\big>\big<h_1,g_1\big>}{\|v'\|^4_{\ell^2}}+\frac{\big<h_2,g_2\big>}{2\|v'\|^4_{\ell^2}}+\frac{\big<h_1,g_1\big>}{32}-\frac{1}{64}\Big)\\
&\leq \frac{32^2}{9\|V\|^2_{\ell^1}}\Big(\frac{\big<h_2,g_2\big>}{2\|v'\|^4_{\ell^2}}+\frac{\big<h_1,g_1\big>}{32}-\frac{1}{64}\Big)<0.
\end{align*}

Now we come to prove \eqref{exp of T2f}. First, using $\big<S_2f,v_k\big>=0$, $S_2v_k=0$ for $k=0,1,2$ and 
$$G_3(n,m)=|n-m|^6-\frac{35}{4}|n-m|^4+\frac{259}{16}|n-m|^2-\frac{225}{64},$$
we obtain
\begin{equation}\label{S2vG3}
S_2vG_3vS_2f=-20\big<f,v_3\big>S_2v_3.
\end{equation}
Moreover, using $QT_0f=0$, $QD_0=D_0=D_0Q$ and \eqref{D0v}, we have
\begin{align}\label{S2T0S2,S2VG1VD0}
\begin{split}
S_2{T^2_0}S_2f&=S_2T_0(P+Q)T_0f=S_2T_0PT_0f=\frac{\big<T_0f,v\big>}{\|V\|_{\ell^1}}S_2T_0v,\\
S_2vG_1vD_0vG_1vS_2f&=S_2vG_1vQD_0QvG_1vf=\frac{16}{9}\big<f,v_3\big>S_2v_3.
\end{split}
\end{align}
To compute the second term in  \eqref{T2}, we use $D_0D_2=D_2=D_2D_0$ and $D_2=S_1D_2S_1$. Using the cancellation properties of $S_1$ and $S_2$, together with \eqref{D0v}, we obtain
\begin{align}\label{second term of T2}
\begin{split}
S_2T_0vG_{-1}vS_1D_2S_1vG_{-1}vT_0S_2f&=\frac{1}{4}\big<T_0f,v\big>\big<D_2S_1v_2,S_1v_2\big>S_2T_0v,\\
S_2T_0vG_{-1}vS_1D_2S_1T_0D_0vG_1vS_2f&=\frac{2\big<f,v_3\big>}{3\|v'\|^2_{\ell^2}}\big<D_2S_1T_0v',S_1v_2\big>S_2T_0v,\\
S_2vG_1vD_0T_0S_1D_2S_1vG_{-1}vT_0S_2f&=\frac{2\big<T_0f,v\big>}{3\|v'\|^2_{\ell^2}}\big<D_2S_1v_2,S_1T_0v'\big>S_2v_3,\\
S_2vG_1vD_0T_0S_1D_2S_1T_0D_0vG_1vS_2f&=\frac{16\big<f,v_3\big>}{9\|v'\|^4_{\ell^2}}\big<D_2S_1T_0v',S_1T_0v'\big>S_2v_3.
\end{split}
\end{align}
Combining \eqref{S2vG3}, \eqref{S2T0S2,S2VG1VD0}, and \eqref{second term of T2} with \eqref{T2} yields \eqref{exp of T2f}.
\vskip0.1cm
{\bf{\underline{(4)}}} Similarly, the desired conclusion follows by computing
\begin{equation*}
\widetilde{T}_1f=\frac{1}{8}\big<f,\tilde{v}_1\big>\widetilde{S}_0\tilde{v}_1+\frac{32}{\|V\|^2_{\ell^1}}\big<\widetilde{T}_0f,\tilde{v}\big>\widetilde{S}_0\widetilde{T}_0\tilde{v},\qquad f\in \widetilde{S}_0\ell^2(\Z).
\end{equation*}
\vskip0.1cm
{\bf{\underline{(5)}}} {\bf{(a)}}~By Lemma \ref{charac of solution and space}, it suffices to show that $0$ is not an eigenvalue of $H$. Let $H\phi=0$ with $\phi\in\ell^2(\Z)$, and set $f=Uv\phi$. Then $f\in S_3\ell^2(\Z)$ and $\phi=-G_0vf$. Furthermore, using the kernel of $G_0$, together with $\left<f,v_k\right>=0$ for $k=0,1,2,3$, and $vf=V\phi$, we obtain
\begin{equation*}
\phi(n)=-\frac{1}{12}\sum\limits_{m\in\Z}^{}\left(|n-m|^3-|n-m|\right)v(m)f(m)=-\frac{1}{6}\sum\limits_{m<n}\big((n-m)^3-(n-m)\big)V(m)\phi(m).
\end{equation*}
Fix $n_0\in\Z$. For $n\geq n_0$, we have
\begin{align*}
6|\phi(n)|&\leq\sum\limits_{n_0\leq m\leq n-1}\big((n-m)^3-(n-m)\big)|V(m)\phi(m)|\\
&\quad+\sup\limits_{m<n_0}|\phi(m)|\sum\limits_{m<n_0}\big((n-m)^3-(n-m)\big)|V(m)|.
\end{align*}
Since $(n-m)^3-(n-m)\leq (n-m)^3\lesssim\left<n\right>^3\left<m\right>^3$ for $n\geq m$, multiplying both sides by $\left<n\right>^{-3}$ gives
\begin{equation}\label{rela of u and omega}
u(n)\leq\sum\limits_{n_0\leq m\leq n-1}B(m)u(m)+A(n_0)=:\omega(n),
\end{equation}
where $C$ is a constant and $$u(n)=\left<n\right>^{-3}|\phi(n)|,\qquad B(m)=C\left<m\right>^{6}|V(m)|,\qquad A(n_0)=C\|{\left<\cdot\right>}^3V(\cdot)\|_{\ell^1}\sup\limits_{m<n_0}|\phi(m)|.$$
We next show that $u=0$, and hence $\phi=0$. Since for any $k\geq n_0$, 
$$\omega(k+1)=\omega(k)+B(k)u(k)\leq \omega(k)+B(k)\omega(k)=\omega(k)(1+B(k)),$$
it follows that
$$u(n)\leq\omega(n)\leq \omega(n_0)
\prod_{k=n_0}^{n-1}\bigl(1+B(k)\bigr)\leq A(n_0)e^{\|B(\cdot)\|_{\ell^1}},$$
where we use the inequality that $1+x\leq e^{x}$ for $x\geq0$. Since $\phi\in\ell^2(\Z)$, letting $n_0\to-\infty$ in the above estimate yields $u(n)=0$. Since $n\in\Z$ is arbitrary, $u=0$, and therefore $\phi=0$.
\vskip0.1cm
{\bf{(b)}} Let $f\in \widetilde{S}_1\ell^2(\Z)$ and satisfy $\widetilde{S}_1\tilde{v}\widetilde{G}_2\tilde{v}\widetilde{S}_1f=0$. Then
\begin{equation*}
0=\big<\widetilde{S}_1\tilde{v}\widetilde{G}_2\tilde{v}\widetilde{S}_1f,f\big>=\big<\widetilde{G}_2\tilde{v}f,\tilde{v}f\big>.
\end{equation*}
On the other hand, by
 \eqref{Puiseux expan of R0 2},
$\widetilde{G}_{-1}\tilde{v}f=0$, and $\big<\widetilde{G}_1\tilde{v}f,\tilde{v}f\big>=0,$
we have
\begin{align*}
\big<\widetilde{G}_2\tilde{v}f,\tilde{v}f\big>&=\lim\limits_{\omega\rightarrow0}\frac{1}{\omega}\big<(R_0((2-\omega)^4)-J\widetilde{G}_0J)vf,vf\big>,\qquad\omega\in\C\setminus\R,\\
&=\lim\limits_{\omega\rightarrow0}\frac{1}{\omega}\int_{-\pi}^{\pi}\Big(\frac{1}{(2-2{\rm cos}x)^2-(2-\omega)^4}-\frac{1}{(2-2{\rm cos}x)^2-16}\Big)|(\mcaF (vf))(x)|^2dx\\
&=\lim\limits_{\omega\rightarrow0}\int_{-\pi}^{\pi}\frac{((2-\omega)^2+4)(\omega-4)}{(16-(2-2{\rm cos}x)^2+(2-\omega)^4-16)(16-(2-2{\rm cos}x)^2)}|(\mcaF (vf))(x)|^2dx\\
&=-32\big<\widetilde{G}_0\tilde{v}f,\widetilde{G}_0\tilde{v}f\big>.
\end{align*}
In the second equality, we used the Plancherel identity and the relation $vf=(\Delta^2-16)J\widetilde{G}_0Jvf$, since $\widetilde{G}_0\tilde{v}f\in\ell^2$ by Lemma \ref{charac of solution and space}. The last equality follows from the dominated convergence theorem by taking $\omega\to0$ through the resolvent set with $|\omega|\leq1$ and $\operatorname{Im}((2-\omega)^4-16))>0$. Therefore, $\widetilde{G}_0\tilde{v}f=0$. Hence, $f=0$ follows from  $0=\widetilde{T}_0f=Uf+\tilde{v}\widetilde{G}_0\tilde{v}f$. This completes the proof.
\end{proof}

\section{Proof of Theorem \ref{relation space and resonance types}}\label{proof of relation space and resonance types}
This section is devoted to the proof of Theorem \ref{relation space and resonance types}. To this end, we first recall 
$$T_{0}=U+vG_{0}v,\qquad\widetilde{T}_{0}=U+\tilde{v}\widetilde{G}_{0}\tilde{v},\qquad U(n)=\operatorname{sign}V(n),\qquad v(n)=\sqrt{|V(n)|},$$
where $G_0$, $\widetilde{G}_0$ are integral operators with the following kernels, respectively:
\begin{equation*}
G_0(n,m)=\frac{1}{12}\left(|n-m|^3-|n-m|\right),\quad
\widetilde{G}_0(n,m)=\frac{1}{32\sqrt2}\left(2\sqrt2 |n-m|-\Big(2\sqrt2-3\right)^{|n-m|}\Big).
\end{equation*}
Additionally, $Q,S_j(j=0,1,2,3),\widetilde{Q},\widetilde{S}_0,\widetilde{S}_1$ 
are orthogonal projections onto the following spaces:
\begin{itemize}
\item[(i)]$Q\ell^2(\Z)=\left({\rm span}\{v\}\right)^{\bot}$,\ $S_0\ell^2(\Z)=\left({\rm span}\{v,v_1\}\right)^{\bot}$,\  $S_1\ell^2(\Z)=\big\{f\in S_0\ell^2(\Z): S_0T_0f=0\big\}$,
    \vskip0.15cm
\item[(ii)]$S_2\ell^2(\Z)=\big\{f\in S_1\ell^2(\Z):\left<f,v_{2}\right>=0,\ QT_0f=0\big\}$,
\vskip0.15cm
\item[(iii)] $S_3\ell^2(\Z)=\big\{f\in S_2\ell^2(\Z):\left<f,v_{3}\right>=0,\ T_0f=0\big\}$,
\vskip0.15cm
\item[(iv)] $\widetilde{Q}\ell^2(\Z)=\left({\rm span}\{\tilde{v}\}\right)^{\bot}$,\quad $\widetilde{S}_0\ell^2(\Z)=\big\{f\in \widetilde{Q}\ell^2(\Z): \widetilde{Q}\widetilde{T}_0f=0\big\}$,
\vskip0.15cm
\item[(v)] $\widetilde{S}_1\ell^2(\Z)=\big\{f\in \widetilde{S}_0\ell^2(\Z):\left<f,\tilde{v}_1\right>=0,\ \widetilde{T}_0f=0\big\}$.
\end{itemize}

To prove Theorem \ref{relation space and resonance types}, it suffices to establish the following lemma.
\begin{lemma}\label{charac of solution and space}
Let $H=\Delta^2+V$ on $\Z$ and $|V(n)|\lesssim \left<n\right>^{-\beta}$ with $\beta>9$. Then
\begin{itemize}
\item [{\rm(i)}]$f\in S_1\ell^2(\Z)\Longleftrightarrow \exists\ \phi\in W_{3/2}(\Z)$ such that $H\phi=0$. Moreover, $f=Uv\phi$ and $\phi(n)=-(G_0vf)(n)+c_1n+c_2$,
    where
    \begin{equation}\label{exp of c1,c2}
    c_1=\frac{\left<T_0f,v'\right>}{\|v'\|^2_{\ell^2}},\quad c_2=\frac{\left<T_0f,v\right>}{\|V\|_{\ell^1}}-\frac{\left<v_1,v\right>}{\|V\|_{\ell^1}}c_1,\quad v'=Q(v_1)=v_1-\frac{\left<v_1,v\right>}{\|V\|_{\ell^1}}v.
    \end{equation}
\item [{\rm(ii)}]$f\in S_2\ell^2(\Z)\Longleftrightarrow \exists\ \phi\in W_{1/2}(\Z)$ such that $H\phi=0$. Moreover, $f=Uv\phi$ and
    $$\phi=-G_0vf+\frac{\left<T_0f,v\right>}{\|V\|_{\ell^1}}.$$
\item [{\rm(iii)}] $f\in S_3\ell^2(\Z)\Longleftrightarrow \exists\ \phi\in \ell^2(\Z)$ such that $H\phi=0$. Moreover, $f=Uv\phi$ and $\phi=-G_0vf$.
\item [{\rm(iv)}] $f\in \widetilde{S}_0\ell^2(\Z)\Longleftrightarrow \exists\ \phi\in W_{1/2}(\Z)$ such that $H\phi=16\phi$. Moreover, $f=Uv\phi$ and
    \begin{equation}\label{exp of c}
    \phi=-J\widetilde{G}_{0}\tilde{v}f+J\frac{\big<\widetilde{T}_0f,\tilde{v}\big>}{\|V\|_{\ell^1}}.
    \end{equation}
\item [{\rm(v)}] $f\in \widetilde{S}_1\ell^2(\Z)\Longleftrightarrow \exists\ \phi\in \ell^2(\Z)$ such that $H\phi=16\phi$. Moreover, $f=Uv\phi$ and $\phi=-J\widetilde{G}_{0}\tilde{v}f$.
\end{itemize}
\end{lemma}
\begin{remark}
{\rm As a consequence of this lemma, we have
\begin{align*}
S_1\ell^2(\Z)&=\{0\}\Leftrightarrow H\phi=0\ {\rm has\  only\ the\  zero\ solution\ in\ }W_{3/2}(\Z),\\
S_2\ell^2(\Z)&=\{0\}\Leftrightarrow H\phi=0\ {\rm has\  only\ the\  zero\ solution\ in\ }W_{1/2}(\Z),\\
S_3\ell^2(\Z)&=\{0\}\Leftrightarrow H\phi=0\ {\rm has\  only\ the\  zero\ solution\ in\ }\ell^2(\Z),\\
\widetilde{S}_0\ell^2(\Z)&=\{0\}\Leftrightarrow H\phi=16\phi\ {\rm has\  only\ the\  zero\ solution\ in\ }W_{1/2}(\Z),\\
\widetilde{S}_1\ell^2(\Z)&=\{0\}\Leftrightarrow H\phi=16\phi\ {\rm has\  only\ the\  zero\ solution\ in\ }\ell^2(\Z),
\end{align*}
from which Theorem \ref{relation space and resonance types} follows. 
}
\end{remark}
\begin{proof}
{\textbf{\underline{(i)}}}~``$\bm{\Longrightarrow}$"  Let $f\in S_1\ell^2(\Z)$. Then $f\in S_0\ell^2(\Z)$ and $S_0T_0f=0$. Denote by $P_0$ the orthogonal projection onto span$\{v,v_1\}$. Then $S_0=I-P_0$, and it follows that
\begin{equation}\label{Uf of S}
Uf=-vG_0vf+P_0T_0f.
\end{equation}
Let
\begin{equation}\label{v'}
v'=Q(v_1)=v_1-\frac{\left<v_1,v\right>}{\|V\|_{\ell^1}}v,
\end{equation}
so that $\{v',v\}$ forms an orthogonal basis for span$\{v,v_1\}$. In this case, we have
\begin{equation}\label{P0T0f 1}
P_0T_0f=\frac{\left<P_0T_0f,v\right>}{\|V\|_{\ell^1}}v+\frac{\left<P_0T_0f,v'\right>}{\|v'\|^2_{\ell^2}}v'=\frac{\left<T_0f,v\right>}{\|V\|_{\ell^1}}v+\frac{\left<T_0f,v'\right>}{\|v'\|^2_{\ell^2}}v'.
\end{equation}
Substituting \eqref{v'} into the second equality of \eqref{P0T0f 1}, we further obtain 
\begin{align}\label{P0T0f2}
\begin{split}
P_0T_0f&=\frac{\left<T_0f,v\right>}{\|V\|_{\ell^1}}v+\frac{\left<T_0f,v'\right>}{\|v'\|^2_{\ell^2}}\left(v_1-\frac{\left<v_1,v\right>}{\|V\|_{\ell^1}}v\right)\\
&=\frac{\left<T_0f,v'\right>}{\|v'\|^2_{\ell^2}}v_1+\left(\frac{\left<T_0f,v\right>}{\|V\|_{\ell^1}}-\frac{\left<T_0f,v'\right>\left<v_1,v\right>}{\|v'\|^2_{\ell^2}\|V\|_{\ell^1}}\right)v\\
&=:c_1v_1+c_2v.
\end{split}
\end{align}
Multiplying both sides of \eqref{Uf of S} by $U$ and substituting $P_0T_0f$ with \eqref{P0T0f2}, we obtain
 $$f=-UvG_0vf+U(c_1v_1+c_2v)=Uv(-G_0vf+c_1n+c_2).$$

 Let $\phi(n)=-(G_0vf)(n)+c_1n+c_2$. We claim that $\phi\in W_{3/2}(\Z)$. Since $|c_1n+c_2|\lesssim1+|n|\in W_{3/2}(\Z)$, it suffices to prove $G_0vf\in W_{3/2}(\Z)$. Indeed, by $\left<f,v\right>=0$ and $\left<f,v_1\right>=0$, 
 \begin{equation*}
 12(G_0vf)(n)=\sum\limits_{m\in\Z}^{}\left(|n-m|^3-|n-m|-n^2|n|+3|n|nm\right)v(m)f(m)=:\sum\limits_{m\in\Z}^{}K_1(n,m)(vf)(m).
 \end{equation*}
 We decompose 
 \begin{align*}
 K_1(n,m)&=|n-m|(n^2-2nm+m^2-1)-n^2|n|+3|n|nm\\
 &=\left(n^2(|n-m|-|n|)+|n|nm\right)-2n(|n-m|-|n|)m+(m^2-1)|n-m|\\
 &=:K_{11}(n,m)-K_{12}(n,m)+K_{13}(n,m).
 \end{align*}
 For $K_{11}(n,m)$, if $n\neq m$, then
 \begin{align*}
 |K_{11}(n,m)|&=\left|\frac{\left(n^2(|n-m|-|n|)+|n|nm\right)(|n-m|+|n|)}{|n-m|+|n|}\right|\\
 &=\left|\frac{n^2m^2+|n|nm(|n-m|-|n|)}{|n-m|+|n|}\right|\leq 2|n|m^2.
 \end{align*}
 Since $K_{11}(n,n)=0$, we have $|K_{11}(n,m)|\leq 2|n|m^2$ for all $n,m\in\Z$.
 Moreover, by the triangle inequality, 
 $$|K_{12}(n,m)|\leq 2|n|m^2,\qquad |K_{13}(n,m)|\leq (1+|n|)\left<m\right>^3.$$
 Therefore, $|K_1(n,m)|\lesssim (1+|n|)\left<m\right>^3 $. Since $\beta>9$, it follows that
  $$|(G_0vf)(n)|\lesssim \sum\limits_{m\in\Z}^{}|K_1(n,m)||v(m)f(m)|\lesssim \left<n\right>\sum\limits_{m\in\Z}^{}\left<m\right>^3|v(m)||f(m)|\lesssim \left<n\right>\in W_{3/2}(\Z). $$
Hence, $\phi\in W_{3/2}(\Z)$. Next, we show that $H\phi=0$. Since $\Delta^2G_0vf=vf$ and $vf=vUv\phi=V\phi$, we obtain $H\phi=(\Delta^2+V)\phi=-\Delta^2G_0vf+V\phi=-vf+vf=0.$
\vskip0.1cm
``$\bm{\Longleftarrow}$" Suppose that $\phi\in W_{3/2}(\Z)$ satisfies $H\phi=0$.  Let $f=Uv\phi$. We will show that $f\in S_1\ell^2(\Z)$ and $\phi(n)=-(G_0vf)(n)+c_1n+c_2$, where $c_1,c_2$ are defined in \eqref{exp of c1,c2}.

\textbf{On the one hand}, $f\in S_0\ell^2(\Z)$, i.e., for $k=0,1$, it can be verified that
$$\left<f,v_k\right>=\sum_{n\in\Z}(Uv\phi)(n)n^kv(n)=\sum_{n\in\Z}n^kV(n)\phi(n)=0.$$
To see this, choose $\eta(x)\in C^{\infty}_0(\R)$ such that $\eta(x)=1$ for $|x|\leq1$ and $\eta(x)=0$ for $|x|>2$. For any $k=0,1$ and $\delta>0$, define
$$F_k(\delta)=\sum_{n\in\Z}n^kV(n)\phi(n)\eta(\delta n).$$
Under the assumptions on $V$ and $\phi\in W_{3/2}(\Z)$, the dominated convergence theorem gives
$$\left<f,v_k\right>=\lim_{\delta\rightarrow0}F_k(\delta).$$
Moreover, for any $\delta>0$, using $V(n)\phi(n)=-(\Delta^2\phi)(n)$ and the self-adjointness of $\Delta^2$, we obtain
$$F_k(\delta)=-\sum_{n\in\Z}(\Delta^2\phi)(n)n^k\eta(\delta n)=-\sum_{n\in\Z}\phi(n)(\Delta^2G_{\delta,k})(n),$$
where $G_{\delta,k}(x)=x^k\eta(\delta x)$. Next we prove that for any $k=0,1,$ $0<\delta<\frac{1}{3}$, and $s>0$,
\begin{equation}\label{estimate of Delta2G}
\|\left<\cdot\right>^s(\Delta^2G_{\delta,k})(\cdot)\|_{\ell^2(\Z)}\lesssim\delta^{\frac{7}{2}-k-s}.
\end{equation}
Once this is established, choose $\frac{3}{2}<s_0<\frac{5}{2}$. Since $\phi\in W_{3/2}(\Z)$, the Cauchy--Schwarz inequality yields
 $$|F_k(\delta)|\lesssim\delta^{\frac{7}{2}-k-s_0}\|\left<\cdot\right>^{-s_0}\phi(\cdot)\|_{\ell^2(\Z)}.$$
 This implies $\lim\limits_{\delta\rightarrow0}F_k(\delta)=0$, and hence $f\in S_0\ell^2(\Z)$.
 To derive \eqref{estimate of Delta2G}, we first apply the differential mean value theorem to get
 $$(\Delta^2G_{\delta,k})(n)=G^{(4)}_{\delta,k}(n-1+\Theta)\qquad\text{for some }\Theta\in[0,4].$$
By Leibniz rule and $\operatorname{supp}\eta^{(j)}\subseteq\{x:1\leq|x|\leq2\}$ for any $j\in\N^{+}$, one has
\begin{equation}\label{derive of G}
\left|(\Delta^2G_{\delta,k})(n)\right|=\left|G^{(4)}_{\delta,k}(n-1+\Theta)\right|\lesssim\delta^{4-k}\sum\limits_{\ell=0}^{k}\left|\eta^{(4-\ell)}(\delta(n-1+\Theta))\right|.
\end{equation}
 For any $\ell\in\N^+$, any $s>0$ and $0<\delta<\frac{1}{3}$, we have
\begin{equation*}
\|\left<\cdot\right>^s\eta^{(\ell)}(\delta(\cdot-1+\Theta))\|^2_{\ell^2(\Z)}\lesssim\sum\limits_{|n|\leq3\delta^{-1}}^{}|n|^{2s}\lesssim\delta^{-2s-1}.
\end{equation*}
Combining this estimate with \eqref{derive of G} and the triangle inequality yields \eqref{estimate of Delta2G}.

\textbf{On the other hand}, we show that $\phi(n)=-(G_0vf)(n)+c_1n+c_2$, from which
$$S_0T_0f=S_0(U+vG_0v)f=S_0v\phi+S_0vG_0vf=S_0v\phi+S_0v(-\phi +c_1n+c_2)=S_0v\phi-S_0v\phi=0.$$
Indeed, since $f\in S_0\ell^2(\Z)$, it follows from ``$\bm{\Longrightarrow}$" that $G_0vf\in W_{3/2}(\Z)$. Hence, $\tilde{\phi}:=\phi+G_0vf\in W_{3/2}(\Z)$ and $\Delta^2\tilde{\phi}=H\phi=0$. It indicates that $\tilde{\phi}=\tilde{c}_{1}n+\tilde{c}_{2}$ for some constants $\tilde{c}_1$ and $\tilde{c}_2$. Next we determine that $\tilde{c}_1=c_1,\tilde{c}_2=c_2$. In view of
$$0=H\phi=(\Delta^2+V)\phi=-vf-VG_0vf+V(\tilde{c}_{1}n+\tilde{c}_{2})=U(-vT_0f+\tilde{c}_1nv^2+\tilde{c}_2v^2),$$
we obtain $\tilde{c}_1v_1+\tilde{c}_2v=T_0f$. Based on this, we further derive
\begin{equation}
\tilde{c_1}\left<v_1,v\right>+\tilde{c}_2\left<v,v\right>=\left<T_0f,v\right>,\qquad 
\tilde{c_1}\left<v_1,v_1\right>+\tilde{c}_2\left<v,v_1\right>=\left<T_0f,v_1\right>.
\end{equation}
This, together with $v_1=v^{'}+\frac{\left<v_1,v\right>}{\|V\|_{\ell^1}}v$ and $\big<v^{'},v\big>=0$, yields
$$ \tilde{c}_1=\frac{\left<T_0f,v'\right>}{\|v'\|^2_{\ell^2}},\quad \tilde{c}_2=\frac{\left<T_0f,v\right>}{\|V\|_{\ell^1}}-\frac{\left<v_1,v\right>}{\|V\|_{\ell^1}}\tilde{c}_1.$$
Therefore, $f\in S_1\ell^2(\Z)$ and (i) is proved.
\vskip0.1cm
{\bf{\underline{(ii)}}}~``$\bm{\Longrightarrow}$" Given any $f\in S_{2}\ell^2(\Z)\subseteq S_{1}\ell^2(\Z)$. Then, by {\bf{\underline{(i)}}}, there exists $\phi\in W_{3/2}(\Z)$ such that
$$\phi(n)=-(G_0vf)(n)+c_1n+c_2\quad {\rm and} \quad  H\phi=0,$$
where $c_1,c_2$ are defined as in \eqref{exp of c1,c2}. Next we show that $c_1=0$ and $G_0vf\in W_{1/2}(\Z)$, and hence the desired result is established. Indeed, since $f\in S_2\ell^2(\Z)$, it follows that $QT_0f=0$ and thus
$$c_1=\frac{\left<T_0f,v'\right>}{\|v'\|^2_{\ell^2}}=\frac{\left<QT_0f,v'\right>}{\|v'\|^2_{\ell^2}}=0.$$

On the other hand, by the definition of $G_0$ and $\left<f,v_k\right>=0$ for $k=0,1,2$, we obtain
\begin{equation*}
 12(G_0vf)(n)=\sum\limits_{m\in\Z}^{}\underbrace{\left(|n-m|^3-|n-m|-n^2|n|+|n|+3|n|nm-3|n|m^2\right)}_{=:K_2(n,m)}v(m)f(m).
 \end{equation*}
For $n\neq0$, we decompose 
 \begin{align*}
 K_2(n,m)&=(m^2-1)(|n-m|-|n|)+(-2)\big(mn(|n-m|-|n|)+|n|m^2\big)\\
 &\quad+\Big(\frac{n^2m^2}{|n-m|+|n|}-\frac{1}{2}|n|m^2\Big)+\Big(\frac{|n|nm(|n-m|-|n|)}{|n-m|+|n|}+\frac{1}{2}|n|m^2\Big)\\
&=:\sum\limits_{j=1}^{4}K_{2j}(n,m).
 \end{align*}
 Obviously, $|K_{21}(n,m)|\lesssim {\left<m\right>}^3$ uniformly in $n$. Moreover, by finding the common denominator and triangle inequality, we get
  \begin{align*}
  |K_{22}(n,m)|&=2\Big|\frac{nm^3+|n|m^2(|n-m|-|n|)}{|n-m|+|n|}\Big|\lesssim {\left<m\right>}^3 ,\\
  |K_{23}(n,m)|&=\Big|\frac{|n|m^2(|n-m|-|n|)}{2(|n-m|+|n|)}\Big|\lesssim{\left<m\right>}^3,\\
  |K_{24}(n,m)|&=\Big|\frac{n^2m^2(|n-m|-|n|)+\frac{1}{2}|n|m^4}{{(|n-m|+|n|)}^2}\Big|\lesssim{\left<m\right>}^3.
  \end{align*}
Combining these estimates with the assumption of $V$, we obtain $G_0vf\in W_{1/2}(\Z)$.
\vskip0.1cm
``$\bm{\Longleftarrow}$" Assume that $\phi\in W_{1/2}(\Z)$ satisfies $H\phi=0$. Define $f=Uv\phi$, next we prove that $f\in S_2\ell^2(\Z)$. Since $W_{1/2}(\Z)\subseteq W_{3/2}(\Z)$, {\bf{\underline{(i)}}} yields that $f\in S_1\ell^2(\Z)$ and
$\phi(n)=-(G_0vf)(n)+c_1n+c_2,$
where $c_1,c_2$ are as defined in \eqref{exp of c1,c2}.

Since $QT_0f=c_1v'$ by $(Q-S_0)\ell^2(\Z)=$Span$\{v'\}$, it suffices to show that $\left<f,v_2\right>=0$ and $c_1=0$. For the former, it follows by the similar method used in ``$\bm{\Longleftarrow}$" of {\bf{\underline{(i)}}}. For the latter, observe that $G_0vf\in W_{1/2}(\Z)$ (from (ii)``$\bm{\Longrightarrow}$" direction above). Thus,
$c_1n=\phi(n)+(G_0vf)(n)-c_2\in W_{1/2}(\Z),$
which implies that $c_1=0$.
\vskip0.1cm
{\bf{\underline{(iii)}}}``$\bm{\Longrightarrow}$" For any $f\in S_3\ell^2(\Z)\subseteq S_2\ell^2(\Z)$, part {\bf{\underline{(ii)}}} indicates that there exists $\phi\in W_{1/2}(\Z)$ such that
$\phi=-G_0vf+\frac{\left<T_0f,v\right>}{\|V\|_{\ell^1}}$ and $H\phi=0.$
Moreover, since $T_0f=0$, then $\phi=-G_0vf$. Therefore, it remains to prove that $G_0vf\in\ell^2(\Z)$. In fact, using  $\left<f,v_k\right>=0$ for $k=0,1,2,3$, we write
\begin{equation*}
 12(G_0vf)(n)=\sum\limits_{m\in\Z}^{}\underbrace{\Big(|n-m|^3-|n-m|-|n|^3+|n|+3|n|nm-3|n|m^2+\frac{nm^3-nm}{|n|}\Big)}_{=:K_3(n,m)}(vf)(m).
 \end{equation*}
For $n\neq0$, we decompose 
\begin{align*}
 K_3(n,m)&=\Big((m^2-1)(|n-m|-|n|)+\frac{nm(m^2-1)}{|n|}\Big)+\Big(\frac{nm^3}{|n|}-\frac{2nm^3}{|n-m|+|n|}\Big)\\
 &\quad+\frac{5}{2}\Big(\frac{|n|m^2(|n|-|n-m|)}{|n-m|+|n|}-\frac{nm^3}{2|n|}\Big)+\Big(\frac{n^2m^2(|n-m|-|n|)}{{(|n-m|+|n|)}^2}+\frac{nm^3}{4|n|}\Big)\\
 &\quad+\frac{|n|m^4}{2{(|n-m|+|n|)}^2}=:\sum\limits_{j=1}^{5}K_{3j}(n,m).
 \end{align*}
Clearly, $|K_{35}(n,m)|\lesssim \frac{{\left<m\right>}^4}{|n|}$. Moreover, finding the common denominator and triangle inequality, 
 \begin{align*}
  |K_{31}(n,m)|&=\Big|\frac{(m^2-1)m^2|n|+nm(m^2-1)(|n-m|-|n|)}{(|n-m|+|n|)|n|}\Big|\lesssim \frac{{\left<m\right>}^4}{|n|} ,\\
  |K_{32}(n,m)|&=\Big|\frac{nm^3(|n-m|-|n|)}{(|n-m|+|n|)|n|}\Big|\lesssim\frac{{\left<m\right>}^4}{|n|},\\
  |K_{33}(n,m)|&=\frac{5}{2}\Big|\frac{m^3n|n|(|n|-|n-m|)-\frac{1}{2}|n|m^5}{{(|n-m|+|n|)}^2|n|}\Big|\lesssim\frac{{\left<m\right>}^4}{|n|},\\
  |K_{34}(n,m)|&=\Big|\frac{n^3m^3(|n-m|-|n|)-\frac{1}{4}nm^4(2n-m)(|n|+|n-m|)+\frac{n|n|m^5}{2}}{{(|n-m|+|n|)}^3|n|}\Big|\lesssim\frac{{\left<m\right>}^4}{|n|}.
  \end{align*}
  Combining these estimates with the assumption of $V$, we obtain $G_0vf\in \ell^2(\Z)$.
\vskip0.1cm
``$\bm{\Longleftarrow}$" Suppose that $\phi\in\ell^2(\Z)$ satisfies $H\phi=0$. Let $f=Uv\phi$. By {\bf{\underline{(ii)}}}, we have $f\in S_2\ell^2(\Z)$ and
$\phi=-G_0vf+\frac{\left<T_0f,v\right>}{\|V\|_{\ell^1}}.$
Thus, it remains to show that $\left<f,v_3\right>=0$ and $PT_0f=0$. The first identity follows by the same cutoff argument as in the ``$\bm{\Longleftarrow}$" part of {\bf{\underline{(i)}}}. For the second, the (iii)``$\bm{\Longrightarrow}$" gives
 $G_0vf\in \ell^2(\Z)$. Hence,
$\frac{\left<T_0f,v\right>}{\|V\|_{\ell^1}}=\phi+G_0vf\in \ell^2(\Z),$ which forces $\left<T_0f,v\right>=0$,
and therefore $PT_0f=0$.
\vskip0.3cm
{\textbf{\underline{(iv)}}}~``$\bm{\Longrightarrow}$" Let $f\in \widetilde{S}_0\ell^2(\Z)$, i.e., $f\in \widetilde{Q}\ell^2(\Z)$ and $\widetilde{Q}\widetilde{T}_{0}f=0$. By $\widetilde{Q}=I-\widetilde{P}$ and $\widetilde{P}=\left\|V\right\|^{-1}_{\ell^1}\left<\cdot,\tilde{v}\right>\tilde{v}$, we have
\begin{equation}\label{de of c}
Uf=-\tilde{v}\widetilde{G}_0\tilde{v}f+\widetilde{P}\widetilde{T}_0f=-\tilde{v}\widetilde{G}_0\tilde{v}f+c\tilde{v},\qquad c=\left\|V\right\|^{-1}_{\ell^1}\big<\widetilde{T}_0f,\tilde{v}\big>.
\end{equation}
Multiplying $U$ from both sides, we obtain
$f=Uv(-J\widetilde{G}_0\tilde{v}f+Jc).$ 

Set $\phi=-J\widetilde{G}_0\tilde{v}f+Jc.$ We prove that $\phi\in W_{1/2}(\Z)$. It is enough to show that $\widetilde{G}_0\tilde{v}f\in W_{1/2}(\Z)$. By the definition of $\widetilde{G}_0$ and $\left<f,\tilde{v}\right>=0$, we have
\begin{equation*}
32\sqrt2(\widetilde{G}_0\tilde{v}f)(n)
=\sum\limits_{m\in\Z}^{}\Big(2\sqrt2 (|n-m|-|n|)-\left(2\sqrt2-3\right)^{|n-m|}\Big)\tilde{v}(m)f(m).
\end{equation*}
Since $0<3-2\sqrt2<1$, together with the triangle inequality, we obtain
\begin{align*}
\left|(\widetilde{G}_0\tilde{v}f)(n)\right|\lesssim \sum\limits_{m\in\Z}(1+|m|)|\tilde{v}(m)f(m)|\lesssim1\in W_{1/2}(\Z).
\end{align*}
Hence, $\phi\in W_{1/2}(\Z)$. Moreover, by $(\Delta^2-16)J\widetilde{G}_0Jvf=vf,(\Delta^2-16)(Jc)=0$ and $vf=V\phi$,
$$(H-16)\phi=(\Delta^2-16+V)\phi=(\Delta^2-16)(-J\widetilde{G}_0\tilde{v}f+Jc)+V\phi=-vf+vf=0.$$
\vskip0.1cm
``$\bm{\Longleftarrow}$" Given that $\phi\in W_{1/2}(\Z)$ satisfies $H\phi=16\phi$. Let $f=Uv\phi$. Then $f\in \widetilde{S}_0\ell^2(\Z)$ and $\phi(n)=-(J\widetilde{G}_0\tilde{v}f)(n)+(-1)^nc$ with $c$ defined in \eqref{de of c}. Indeed, let $\eta$ be as in \textbf{\underline{(i)}}. For any $\delta>0$, define
$$\widetilde{F}(\delta)=\sum\limits_{n\in\Z}(JV\phi)(n)\eta(\delta n).$$
Using $V\phi=-(\Delta^2-16)\phi$ and $J\Delta J=-\Delta-4$, 
the same cutoff argument as in part \textbf{\underline{(i)}} yields
$$\left<f,\tilde{v}\right>=\lim\limits_{\delta\rightarrow0}\widetilde{F}(\delta)=-\lim\limits_{\delta\rightarrow0}\sum\limits_{n\in\Z} (J\phi)(n)[(\Delta^2+8\Delta)(\eta(\delta\cdot))](n)=0.$$
We next show that $\phi(n)=-(J\widetilde{G}_0\tilde{v}f)(n)+(-1)^nc$. Set $\tilde{\phi}=\phi+J\widetilde{G}_0\tilde{v}f$. By an argument similar to \textbf{\underline{(i)}}, $\tilde{\phi}\in W_{
1/2}(\Z)$ and $(\Delta^2-16)\tilde{\phi}=0$. Equivalently, $(\Delta^2+8\Delta)J\tilde{\phi}=0$, which implies that $J\tilde{\phi}=\tilde{c}$ for some constant $\tilde{c}$. Hence, $\phi(n)=-(J\widetilde G_0\tilde vf)(n)+(-1)^n\widetilde c.$ Using $H\phi=16\phi$ and the method as in {\bf{\underline{(i)}}}, we obtain $\tilde{c}\tilde{v}=\widetilde{T}_0f$, and thus $\tilde{c}=\frac{\left<\widetilde{T}_0f,\tilde{v}\right>}{\|V\|_{\ell^1}}$. Moreover, we can verify $$\widetilde{Q}\widetilde{T}_{0}f=\widetilde{Q}(U+\tilde{v}\widetilde{G}_{0}\tilde{v})f=\widetilde{Q}v\phi+\widetilde{Q}\tilde{v}\widetilde{G}_0\tilde{v}f=\widetilde{Q}v\phi+\widetilde{Q}\tilde{v}(-J\phi+c)=0.$$
Together with $\left<f,\tilde{v}\right>=0$, this yields $f\in \widetilde{S}_0\ell^2(\Z)$ and (iv) is proved. 
\vskip0.1cm
{\textbf{\underline{(v)}}}~``$\bm{\Longrightarrow}$" Let $f\in \widetilde{S}_1\ell^2(\Z)\subseteq \widetilde{S}_0\ell^2(\Z)$. By (iv) and $\widetilde{T}_0f=0$, there exists $\phi\in W_{1/2}(\Z)$ such that
$\phi=-J\widetilde{G}_0\tilde{v}f$ and $H\phi=16\phi$.
Thus, it remains to show that $\widetilde{G}_0\tilde{v}f\in\ell^2(\Z)$. To this end, let $q=3-2\sqrt{2}$. We write
\begin{equation}\label{exp of wideG0}
32\sqrt2\big(\widetilde{G}_0\tilde{v}f\big)(n)=2\sqrt2\sum\limits_{m\in\Z}|n-m|(\tilde{v}f)(m)-\sum\limits_{m\in\Z}{(-q)}^{|n-m|}(\tilde{v}f)(m)=:F_1(n)-F_2(n).
\end{equation}
For $n\neq0$, since $\big<f,\tilde{v}_k\big>=0$ for $k=0,1$, we can rewrite $F_1(n)$ and $F_2(n)$ as follows:
\begin{equation}\label{exp of W1}
F_1(n)=2\sqrt{2}\sum\limits_{m\in\Z}|n-m|(\tilde{v}f)(m)=\sum\limits_{m\in\Z}2\sqrt{2}\Big(|n-m|-|n|+\frac{nm}{|n|}\Big)(\tilde{v}f)(m)
\end{equation}
and
\begin{align}\label{exp of W2}
F_2(n)&=\sum\limits_{m\in\Z}{(-1)}^{|n-m|}\big({q}^{|n-m|}-q^{|n|}\big)(\tilde{v}f)(m)+\sum\limits_{m\in\Z}\big({(-1)}^{|n-m|}-(-1)^{|n|}\big)q^{|n|}(\tilde{v}f)(m)\notag\\
&=:F_{21}(n)+F_{22}(n).
\end{align}
Since $0<q<1$, it is clear that $F_{22}\in\ell^2(\Z)$. Let $\widetilde{K}_1(n,m)=|n-m|-|n|+\frac{nm}{|n|}$. We next prove
\begin{equation}\label{estimate of Ktuta1 and q}
|\widetilde{K}_1(n,m)|+\big|{q}^{|n-m|}-q^{|n|}\big|\lesssim {|n|}^{-1}{\left<m\right>}^2,\qquad n\neq0,
\end{equation}
which together with \eqref{exp of wideG0}--\eqref{exp of W2} yields $\widetilde G_0\tilde vf\in\ell^2(\Z).$ To prove \eqref{estimate of Ktuta1 and q}, first observe that
\begin{equation}\label{estimate of Ktuta1}
|\widetilde{K}_1(n,m)|=\left|\frac{nm(|n-m|-|n|)}{(|n-m|+|n|)|n|}+\frac{m^2}{|n-m|+|n|}\right|\lesssim \frac{{\left<m\right>}^2}{|n|}.
\end{equation}
Next, let $a={\rm ln}q<0$. When $|n-m|\neq|n|$, by the mean value theorem, we have
\begin{equation*}
{q}^{|n-m|}-q^{|n|}=ae^{a\big(\min\{|n|,|n-m|\}+\theta\big||n-m|-|n|\big|\big)}(|n-m|-|n|),\qquad \theta\in (0,1).
\end{equation*}
Denote $\widetilde{K}_2(\theta,n,m)=e^{a\big(\min\{|n|,|n-m|\}+\theta\big||n-m|-|n|\big|\big)}
$. Since $a<0$, for $k=0,1$, we have
\begin{equation}\label{min K2tuta}
\sup\limits_{n,m\in\Z,\theta\in(0,1)}(\min\{|n|,|n-m|\})^{k}|\widetilde{K}_2(\theta,n,m)|\lesssim 1.
\end{equation}
If $|n-m|>|n|$, using $|n-m|-|n|=\widetilde{K}_{1}(n,m)-\frac{nm}{|n|}$, together with \eqref{estimate of Ktuta1} and \eqref{min K2tuta}, we obtain
\begin{equation}\label{estimate of q 1}
\left|{q}^{|n-m|}-q^{|n|}\right|=|a|\left|\widetilde{K}_2(\theta,n,m)\Big(\widetilde{K}_{1}(n,m)-\frac{nm}{|n|}\Big)\right|\lesssim\frac{{\left<m\right>}^2}{|n|}. \end{equation}
If $|n-m|<|n|$, using $|n-m|-|n|=\widetilde{K}_{1}(n,m)-\frac{m^2}{|n|}-\frac{(n-m)m}{|n|}$, \eqref{estimate of Ktuta1}, and \eqref{min K2tuta}, we obtain
\begin{equation}\label{estimate of q 2}
\left|{q}^{|n-m|}-q^{|n|}\right|=|a|\left|\widetilde{K}_2(\theta,n,m)\Big(\widetilde{K}_{1}(n,m)-\frac{m^2}{|n|}-\frac{(n-m)m}{|n|}\Big)\right|\lesssim\frac{{\left<m\right>}^2}{|n|}. \end{equation}
Therefore, combining \eqref{estimate of q 1} and \eqref{estimate of q 2} yields \eqref{estimate of Ktuta1 and q}.
\vskip0.1cm
``$\bm{\Longleftarrow}$" Assume that $\phi\in\ell^2(\Z)$ satisfies $H\phi=16\phi$. Define $f=Uv\phi$. We shall prove that $f\in\widetilde{S}_1\ell^2(\Z)$. First, since $\ell^2(\Z)\subseteq W_{1/2}(\Z)$, it follows from (iv) that
$f\in \widetilde{S}_0\ell^2(\Z)$ and
$$\phi=-J\widetilde{G}_0\tilde{v}f+J\frac{\big<\widetilde{T}_0f,\tilde{v}\big>}{\|V\|_{\ell^1}}.$$
Thus, it remains to show $\big<f,\tilde{v}_1\big>=0$ and $\widetilde{P}\widetilde{T}_0f=0$. The first identity follows by the same argument as in the ``$\bm{\Longleftarrow}$" part of {\bf{\underline{(iv)}}}. For the second, by (v)``$\bm{\Longrightarrow}$",
 we obtain $J\widetilde{G}_0\tilde{v}f\in \ell^2(\Z)$ and 
$$J\frac{\big<\widetilde{T}_0f,\tilde{v}\big>}{\|V\|_{\ell^1}}=\phi+J\widetilde{G}_0\tilde{v}f\in \ell^2(\Z),$$
which forces $\big<\widetilde{T}_0f,\tilde{v}\big>=0$ and therefore $\widetilde{P}\widetilde{T}_0f=0$. This completes the proof.
\end{proof}
 
\appendix
\section{Commutator estimates and Mourre Theory}\label{section of Appendix}
In this appendix, we first recall the abstract commutator estimates of Jensen, Mourre, and Perry \cite{JMP84}, which yield regularity properties of the resolvent with respect to the spectral parameter between suitably weighted spaces. We then collect some sufficient conditions for the regularity of bounded self-adjoint operators with respect to conjugate operators, which will be used to verify the hypotheses of these abstract results in our setting.

We begin with some notation and definitions. Let $X$ be a separable complex Hilbert space with inner product $\left<\cdot,\cdot\right>$, and let $T$ be a self-adjoint operator on $X$ with domain $\mcaD(T)$.
\begin{itemize}
\item
Define
$$X_{+2}:=\left(\mcaD(T),\left<\cdot,\cdot\right>_{+2}\right), \qquad \left<\varphi,\phi\right>_{+2}:=\left<\varphi,\phi\right>+\left<T\varphi,T\phi\right>,\quad\forall\ \varphi,\phi\in\mcaD(T),$$
\end{itemize}
and let $X_{-2}$ be the dual space of $X_{+2}$.
\begin{itemize}
\item Let $A$ be a self-adjoint operator on $X$. The commutator form $[T,A]$ defined on $\mcaD(T)\cap\mcaD(A)$,
    \end{itemize}
     is given by
\begin{equation}\label{Sesqui-form}
   [T,A](\varphi,\phi):=\left<A\varphi,T\phi\right>-\left<T\varphi,A\phi\right>,\qquad \forall\ \varphi,\phi\in\mcaD(T)\cap\mcaD(A).
\end{equation}

\begin{definition}\label{condi-veri}
{\rm Let $T$ be as above and $n\geq1$ be an integer. A self-adjoint operator $A$ on $X$ is said to be conjugate to $T$ at the point $E\in\R$, and $T$ is said to be $n$-smooth with respect to $A$, if the following conditions (a)$\sim$(e) are satisfied:
\begin{itemize}
\item [(a)] $\mcaD(A)\cap\mcaD(T)$ is a core for $T$.
\item [(b)]$e^{i\theta A}$ maps $\mcaD(T)$ into $\mcaD(T)$ and for each $\phi\in\mcaD(T)$,
$$\sup_{|\theta|\leq1}\|Te^{i\theta A}\phi\|<\infty.$$
\item [({\rm$c_n$})] The form $i[T,A]$ defined on $\mcaD(T)\cap\mcaD(A)$, is bounded from below and closable. The self-adjoint operator associated with its closure is denoted by $iB_1$. Assume $\mcaD(T)\subseteq\mcaD(B_1)$. If $n>1$, assume for $j=2,\cdots,n$ that the form $i[iB_{j-1},A]$, defined on $\mcaD(T)\cap\mcaD(A)$, is bounded from below and closable. The associated self-adjoint operator is denoted by $iB_{j}$, and it is assumed that $\mcaD(T)\subseteq\mcaD(B_j)$.
\item [({\rm$d_n$})] The form $[B_n, A]$, defined on $\mcaD(T)\cap\mcaD(A)$, extends to a bounded operator from $X_{+2}$ to $X_{-2}$.
\item [(e)] There exist $\alpha>0, \delta>0$, and a compact operator $K$ on $X$ such that
\begin{equation}\label{Mourre}
E_{T}(\mcaJ)iB_1E_{T}(\mcaJ)\geq\alpha E_{T}(\mcaJ)+E_{T}(\mcaJ)KE_{T}(\mcaJ),
\end{equation}
where $\mcaJ=(E-\delta,E+\delta)$ is called the interval of conjugacy.
\end{itemize}
}
\end{definition}
\begin{theorem}\label{Resol-Smoo}
{\rm(\cite[Theorem 2.2]{JMP84})} Let $T$ be as above and $n\geq1$ an integer. Let $A$ be a conjugate operator to $T$ at $E\in\R$. Assume that $T$ is $n$-smooth with respect to $A$. Let $\mcaJ$ be the interval of conjugacy and ${\rm I}\subseteq \mcaJ\cap\sigma_c(T)$~a relatively compact interval, where $\sigma_c(T)$ denotes the continuous spectrum of $T$. Let $s>n-\frac{1}{2}$.
 \begin{itemize}
\item [{\rm (i)}]For $\operatorname{Re}z\in {\rm I},\ \operatorname{Im}z\neq0$, one has
$$\|\left<A\right>^{-s}(T-z)^{-n}\left<A\right>^{-s}\|\leq c.$$
\item [{\rm (ii)}] For $\operatorname{Re}z,\operatorname{Re}z'\in {\rm I},\ 0<|\operatorname{Im}z|\leq1,\ 0<|\operatorname{Im}z'|\leq1$, there exists a constant $C$, independent of $z,z'$, such that
    $$\|\left<A\right>^{-s}\big((T-z)^{-n}-(T-z')^{-n}\big)\left<A\right>^{-s}\|\leq C|z-z'|^{\delta_1},$$
    where
    \begin{equation}\label{delta}
    \delta_1=\delta_1(s,n)=\frac{1}{1+\frac{sn}{s-n+\frac{1}{2}}}.
    \end{equation}
\item [{\rm (iii)}] Let $\lambda\in {\rm I}$. The norm limits
$$\lim\limits_{\varepsilon\downarrow0}\left<A\right>^{-s}(T-\lambda\pm i\varepsilon)^{-n}\left<A\right>^{-s}$$
exist and equal
$$\frac{1}{(n-1)!}\left(\frac{d}{d\lambda}\right)^{n-1}\big(\left<A\right>^{-s}(T-\lambda\pm i0)^{-1}\left<A\right>^{-s}\big),$$
where
$$\left<A\right>^{-s}(T-\lambda\pm i0)^{-1}\left<A\right>^{-s}=\lim\limits_{\varepsilon\downarrow0}\left<A\right>^{-s}(T-\lambda\pm i\varepsilon)^{-1}\left<A\right>^{-s}.$$
The norm limits are H\"{o}lder continuous with exponent $\delta_1(s,n)$ given above.
\end{itemize}
\end{theorem}
To apply Theorem \ref{Resol-Smoo} in our setting, we focus on the case where $T$ is bounded. Then conditions (a) and (b) are automatic, while the commutator conditions $({\rm c}_n)$ and $({\rm d}_n)$ are ensured by the $C^{n+1}(A)$-regularity of $T$. We therefore recall the relevant notion of $C^k(A)$ regularity below.

\begin{definition}\label{def of regularity}
{\rm
Given an integer $k\in\N$, we say that $T$ belongs to the class $C^{k}(A)$, and write $T\in C^{k}(A)$, if the sesquilinear form $ad^{j}_{A}(T)$, initially defined on $\mcaD(A^k)\times\mcaD(A^k)$, extends continuously to $X\times X$ for all $0\leq j\leq k$. We identify this extension with the associated bounded operator on $X$ and denote it by the same symbol. The iterated commutators $ad^{k}_{A}(T)$ are defined recursively by
\begin{align}\label{adk}
\begin{split}
ad^{0}_A(T)&=T,\qquad
ad^{1}_{A}(T)=[T,A]=TA-AT,\\
ad^{k}_{A}(T)&=ad^{1}_{A}\left(ad^{k-1}_{A}(T)\right)=\sum\limits_{i,j\geq0,\ i+j=k}^{}\frac{k!}{i!j!}(-1)^{i}A^{i}TA ^{j},\qquad k\geq1.
\end{split}
\end{align}
If $T\in C^k(A)$ for every $k\in\N$, we say that $T$ belongs to the class $C^\infty(A)$.
}
\end{definition}
\begin{remark}
{\rm This property is often referred to as the regularity of $T$ with respect to $A$ in many contexts. Equivalently,  $T\in C^{k}(A)$ holds if and only if the $X$-valued function $f(t)\phi$ on $\R$ has the usual $C^{k}(\R)$ regularity for every $\phi\in X$, where $f$ is defined by
\begin{align*}
f:~~\R\longrightarrow\B(X),\qquad t\longmapsto f(t)=e^{itA}Te^{-itA}.
\end{align*}}
\end{remark}
Moreover, the classes $C^k(A)$ are stable under addition, as recorded in the following lemma.
\begin{lemma}\label{Regularity lemma}
{  For any $k\in\N^{+}$, let $T_1,T_2$ be bounded self-adjoint operators on $X$ such that $T_1,T_2\in C^{k}(A)$. Then, $T_1+T_2\in C^{k}(A)$ and $ad^{k}_{A}(T_1+T_2)=ad^{k}_{A}(T_1)+ad^{k}_{A}(T_2)$.
}
\end{lemma}
\begin{proof}
The result follows from the case $k=1$ established in \cite[Section 2]{GGM04} and induction on $k$.
\end{proof}
For our application, we take $X=\ell^2(\Z)$ and $T=H=\Delta^2+V$, where $|V(n)|\lesssim \langle n\rangle^{-\beta}$ for some $\beta>0$, and let $A$ be defined as in \eqref{A}. We then establish the following regularity property of $H$ with respect to $A$.
\begin{lemma}\label{regularity of H}
{ Let $H=\Delta^2+V$, where $|V(n)|\lesssim \left<n\right>^{-\beta}$ with $\beta>1$ and let $A$ be defined as in \eqref{A}. Then, $H\in C^{[\beta]}(A)$, where $[\beta]$ denotes the greatest integer not exceeding $\beta$.}
\end{lemma}
\begin{proof}
From \cite[Lemma 4.1]{BS99}, we have $ad^{1}_{iA}(-\Delta)=-\Delta(4+\Delta)$. Hence,
$$ad^{1}_{iA}(\Delta^2)=(-\Delta)[ad^{1}_{iA}(-\Delta)]+[ad^{1}_{iA}(-\Delta)](-\Delta)=2\Delta^2(4+\Delta).$$
This shows $\Delta^2\in C^1(A)$. By induction, for any $k\in\N^{+}$, $ad^{k}_{iA}(\Delta^2)$ can be expressed as a polynomial in $-\Delta$ of degree $2+k$, establishing $\Delta^2\in C^{k}(A)$ for all $k$. Hence, $\Delta^2\in C^{\infty}(A)$.

For the potential $V$, by \cite[Proposition 5.1]{BS99}, $V\in C^k(A)$ for a positive integer $k$ provided that
$$V(n)\rightarrow0\ {\rm and}\  |(\mcaP^{k}V)(n)|=O(|n|^{-k}),\qquad|n|\rightarrow\infty.$$
Under our assumption on $V$, $V\in C^{[\beta]}(A)$. Thus, by Lemma \ref{Regularity lemma}, we conclude that $H\in C^{[\beta]}(A)$.
\end{proof}
\normalem

\end{document}